\documentclass{amsart}
\usepackage{amsmath,amssymb,amsthm,amsfonts,amscd,mathrsfs,mathtools}
\usepackage{tikz-cd}
\usepackage[hyphens]{url}
\usepackage{enumitem}
\usepackage{caption}
\usepackage{hyperref}
\usepackage[utf8]{inputenc}

\theoremstyle{plain}
\newtheorem{thm}{Theorem}[section]
\newtheorem{lem}[thm]   {Lemma}
\newtheorem{cor}[thm]   {Corollary}
\newtheorem{prop}[thm]  {Proposition}

\newtheorem{athm}{Theorem}

\newtheorem{acor}{Corollary}[athm]

\theoremstyle{definition}
\newtheorem{defn}[thm]  {Definition}
\newtheorem{conj}[thm]{Conjecture}
\newtheorem{ex}[thm]{Example}
\newtheorem{rem}[thm]{Remark}
\newtheorem{nota}[thm]{Notation}

\newcommand{\fA}{\mathfrak{A}} 
\newcommand{\bB}{\mathbf{B}} 
\newcommand{\bbB}{\mathbb{B}} 
\newcommand{\bS}{\mathbb{S}} 
\newcommand{\cC}{\mathcal{C}} 
\newcommand{\cD}{\mathcal{D}} 
\newcommand{\fe}{\mathfrak{e}} 
\newcommand{\cF}{\mathcal{F}} 
\newcommand{\cG}{\mathcal{G}} 
\newcommand{\fI}{\mathfrak{I}} 
\newcommand{\fM}{\mathfrak{M}} 
\newcommand{\cO}{\mathcal{O}}
\newcommand{\Q}{\mathbb{Q}} 
\newcommand{\R}{\mathbb{R}} 
\newcommand{\cS}{\mathcal{S}} 
\newcommand{\ft}{\mathfrak{t}} 
\newcommand{\fU}{\mathfrak{U}} 
\newcommand{\Z}{\mathbb{Z}} 

\newcommand{\bDelta}{\mathbf{\Delta}} 
\newcommand{\red}{\mathrm{red}}
\newcommand{\Ab}{\mathrm{Ab}} 
\newcommand{\free}{\mathrm{free}}
\DeclareMathOperator*{\bigast}{\raisebox{-0.6ex}{\scalebox{2.5}{$\ast$}}}
\newcommand{\Alg}{\mathrm{Alg}} 
\newcommand{\cat}{\mathrm{cat}} 
\newcommand{\Cat}{\mathrm{Cat}} 
\newcommand{\Seg}{\mathrm{Seg}} 

\newcommand{\inftwo}{{(\infty,2)}}
\newcommand{\Catinfty}{\Cat_\infty}
\newcommand{\RelCatinfty}{\mathrm{Rel}\Catinfty}
\newcommand{\Catinftwo}{\Cat_\inftwo}
\newcommand{\twosimeq}{{2\simeq}} 
\newcommand{\CAlg}{\mathrm{CAlg}} 
\newcommand{\cof}{\mathrm{cof}} 
\newcommand{\fib}{\mathrm{fib}} 
\newcommand{\colim}{\mathrm{colim}}
\newcommand{\Cob}{\mathrm{Cob}} 
\newcommand{\Fin}{\mathrm{Fin}} 
\newcommand{\Fun}{\mathrm{Fun}} 
\newcommand{\Hom}{\mathrm{Hom}}
\newcommand{\llax}{{\text{l-lax}}} 
\newcommand{\rlax}{{\text{r-lax}}} 
\newcommand{\Mod}{\mathrm{Mod}} 
\newcommand{\one}{\mathbf{1}} 
\newcommand{\op}{\mathrm{op}}
\newcommand{\PrL}{\mathrm{Pr}^L}
\newcommand{\SPrL}{\mathrm{S}\PrL}
\newcommand{\Pic}{\mathrm{Pic}} 
\newcommand{\PSh}{\mathscr{P}} 
\DeclareFontFamily{U}{min}{}
\DeclareFontShape{U}{min}{m}{n}{<-> udmj30}{}
\newcommand\yo{\!\text{\usefont{U}{min}{m}{n}\symbol{'207}}\!} 
\newcommand{\Day}{\mathrm{Day}} 
\newcommand{\Sp}{\mathrm{Sp}} 
\newcommand{\Cospan}{\mathrm{Cospan}} 
\newcommand{\bCospan}{\mathbb{C}\mathrm{ospan}} 
\newcommand{\Span}{\mathrm{Span}} 
\newcommand{\bSpan}{\mathbb{S}\mathrm{pan}} 
\newcommand{\Top}{\mathrm{Top}} 

\newcommand{\CFrob}{\mathrm{CFrob}} 
\newcommand{\GrCob}{\mathrm{GrCob}} 
\newcommand{\GFT}{\mathrm{GFT}} 
\newcommand{\Gaf}{\mathrm{Gaf}} 
\newcommand{\bGr}{{\mathbb{G}\mathrm{r}}} 
\newcommand{\Gr}{\mathrm{Gr}} 
\newcommand{\AVH}{\mathfrak{AVH}}
\newcommand{\tree}{\mathrm{tree}}
\newcommand{\OC}{\mathcal{OC}}

\newcommand{\Th}{\mathrm{Th}} 
\newcommand{\hAut}{\mathrm{hAut}} 
\newcommand{\map}{\mathrm{map}} 
\newcommand{\Out}{\mathrm{Out}} 
\newcommand{\Aut}{\mathrm{Aut}} 
\newcommand{\Poin}{\mathrm{Poin}} 
\newcommand{\asph}{\mathrm{asph}}

\newcommand{\hPsi}{{\hat\Psi}}
\newcommand{\Pxi}{{\Psi,\xi}} 

\newcommand{\set}[1]{\left\{#1\right\}}
\newcommand{\hto}{\hookrightarrow}
\newcommand{\ot}{\leftarrow}
\newcommand{\del}{\partial}
\newcommand{\Id}{\mathrm{Id}} 
\newcommand{\ul}[1]{\underline{#1}} 

\begin{document}


\title{String topology and graph cobordisms}

\author{Andrea Bianchi}
\address{Max Planck Institute for Mathematics,
Vivatsgasse 7, Bonn,
Germany
}
\email{bianchi@mpim-bonn.mpg.de}
\dedicatory{To Annika, thankful for her patience and support}

\date{\today}
\begin{abstract}
We introduce a symmetric monoidal $\infty$-category $\mathrm{GrCob}$ of graph cobordisms between spaces, and use the homology of its morphism spaces to define string operations. Precisely, for an $E_\infty$-ring spectrum $R$ and an oriented $d$-dimensional $R$-Poincar\'e duality space $M$, we construct a ``graph field theory'' $\mathrm{GFT}_M$, i.e. a symmetric monoidal functor from a suitable $R$-linearisation of $\mathrm{GrCob}^\mathrm{op}$ to the category $\mathrm{Mod}_R$ of $R$-modules in spectra; the graph field theory takes an object $X\in\mathrm{GrCob}^\mathrm{op}$, i.e. a space, to the $R$-module $\mathrm{map}(X,M)\otimes R$ of $R$-chains on the mapping space from $X$ to $M$; by selecting suitable graph cobordisms we recover the basic string operations given by restriction, cross product with the fundamental class, and the Chas--Sullivan operations.

The construction is natural with respect to oriented homotopy equivalences of $R$-Poincar\'e duality spaces; in particular, restricting to the endomorphisms of $\emptyset\in\mathrm{GrCob}^\mathrm{op}$, we obtain characteristic classes of $R$-oriented $M$-fibrations parametrised by the suitably twisted homology of $\mathbf{B}\mathrm{Out}(F_n)$, recovering results of Berglund and Barkan--Steinebrunner.

Finally, we describe explicitly the morphism spaces in $\mathrm{GrCob}$, answering along the way a question by Hatcher. This allows us to construct a symmetric monoidal functor from the open-closed cobordism $\infty$-category $\mathcal{OC}$ to $\mathrm{GrCob}$. Composing with $\mathrm{GFT}_M$, we obtain an open-closed field theory with values in $\mathrm{Mod}_R$, attaining values $LM\otimes R$ and $M\otimes R$ at the circle and at the interval, respectively. We expect this to recover and extend constructions of Cohen, Godin and others.

\end{abstract}
\maketitle
\tableofcontents
\section{Introduction}
\subsection{String topology and string operations}
\label{subsec:basic}
In the first part of the introduction, we let $R$ be a discrete commutative ring for (singular) homology coefficients, and we let $M$ be a smooth, closed, $R$-oriented manifold of dimension $d\ge0$.

String topology, introduced by Chas and Sullivan \cite{ChasSullivan}, is mostly concerned with the study of the free loop space $LM=\map(S^1,M)$:
the topology of this space is related to the study of closed geodesics for a Riemannian metric on $M$ \cite{LyusternikFet,GromollMeyer}, and to the symplectic homology of the Liouville domain $T^*M$ \cite{ViterboICM,ViterboI,ViterboII,AbbSch,SalamonWeber,Kragh,Abou}.
In this article we study more generally mapping spaces $M^X=\map(X,M)$, for arbitrary spaces $X$, and their homology with $R$-coefficients.

The collection of the homology groups $H_*(M^X;R)$, for varying $X$, carries additional algebraic structure, arising in the form of ``string operations'' $H_*(M^X;R)\to H_*(M^Y;R)$, for suitably related spaces $X$ and $Y$. The most basic operations are the following three.
\begin{enumerate}
 \item For each continuous map $Y\to X$ of topological spaces, we obtain a map $M^X\to M^Y$ inducing an operation $H_*(M^X;R)\to H_*(M^Y;R)$ in homology.
 \item Let $X$ be a space and let $Y=X\sqcup*$ be the space obtained from $X$ by adjoining a point. Then $M^Y\cong M^X\times M$, and the cross product with the fundamental class of $M$ gives an operation
 \[
 \begin{tikzcd}[column sep=40pt]
  H_*(M^X;R)\ar[r,"{-\times[M]_R}"] & H_{*+d}(M^X\times M;R)\cong H_{*+d}(M^Y;R).
 \end{tikzcd}
 \]
 \item Let $X$ be a space and let $Y= X\sqcup_{\partial I}I$ be a space obtained from $X$ by glueing an interval $I=[0,1]$ along its boundary, via a map $\partial I\to X$. One can then define a string operation $H_*(M^X;R)\to H_{*-d}(M^Y;R)$
 in the spirit of Chas and Sullivan \cite{ChasSullivan}: up to homotopy equivalent replacements, we may assume that $\del I\to X$ is a cofibration, and hence replace $Y$ by $X/\del I$; then $M^{X/\del I}\hookrightarrow M^X$ admits an oriented normal bundle of dimension $d$, and we consider the following composite, involving the Thom isomorphism
 \[
 H_*(M^X;R)\to H_*(M^X,M^X\setminus M^{X/\del I};R)\cong H_{*-d}(M^{X/\del I};R)\cong H_{*-d}(M^Y;R).
 \]
\end{enumerate}

\subsection{Graph cobordisms}
Graph cobordisms generalise simultaneously the situations (1)-(3) in the previous list.
\begin{defn}
 \label{defn:graphcob}
Let $X$ and $Y$ be two spaces. A \emph{graph cobordism} from $Y$ to $X$ is the datum of a diagram of spaces $Y\to W\hookleftarrow X$, together with a finite, relative cell structure on the pair $(W,X)$ consisting only of 0-cells and 1-cells.

More precisely, we fix finite sets $V,E$ and a map $\sigma\colon E\times\del I\to X\sqcup V$, we let $W$ be the pushout of $(X\sqcup V)\leftarrow E\times\del I\to E\times I$, and we fix a map $Y\to W$.
\end{defn}

Given a graph cobordism $Y\to W\hookleftarrow X$, we can define as follows a string operation $H_*(M^X;R)\to H_{*+d\cdot\chi(W,X)}(M^Y;R)$:
\begin{itemize}
 \item we fist define an operation $H_*(M^X;R)\to H_{*+d\cdot |V|}(M^{X\sqcup V};R)$, where $V$ is the set of 0-cells in $(W,X)$, by iterating operations of type (2);
 \item we then map $H_{*+d\cdot|V|}(M^{X\sqcup V};R)\to H_{*+d\cdot\chi(W,X)}(M^W;R)$ by iterating operations of type (3);
 \item finally, we map $H_{*+d\cdot\chi(W,X)}(M^W;R)\to H_{*+d\cdot\chi(W,X)}(M^Y;R)$ by an operation of type (1).
\end{itemize}

It turns out that the operations defined in this way do not depend on the actual cell decomposition of $(W,X)$, but only on the relative homotopy type of the pair $(W,X)$. 
In order to prove this, and in order to be able to define ``higher string operations'', we would like to assemble the collection of all graph cobordisms from $Y$ to $X$ into a \emph{moduli space} $\fM_\Gr(Y,X)$. Intuitively, deformations of graph cobordisms (corresponding to paths in $\fM_\Gr(Y,X)$) should arise by changing continously the attaching maps of 1-cells, or the map $Y\to W$; or even by suitably modifying the relative cell structure of $(W,X)$, though without affecting the relative homotopy type of this pair. The main desideratum for the space $\fM_\Gr(Y,X)$ is the following:
for a suitable local coefficient system $\xi_d^R$ on $\fM_\Gr(Y,X)$, the homology with local coefficients $H_*(\fM_\Gr(Y,X);\xi_d^R)$
should give rise to higher string operations $H_*(M^X;R)\to H_*(M^Y;R)$, possibly with homological degree shifts.

Instead of working with a single pair of spaces $X,Y$, we shall introduce an entire category $\GrCob$ of \emph{graph cobordisms between spaces}. We will define $\GrCob$ as a certain symmetric monoidal $\infty$-category; the reader unfamiliar with $\infty$-categories may think of $\GrCob$ as a category enriched in topological spaces. Objects of $\GrCob$ are given by spaces; and for spaces $X$ and $Y$ we will \emph{define} $\fM_\Gr(Y,X)$ as the morphism space from $Y$ to $X$ in $\GrCob$.

We shall also define a local coefficient system $\xi_d^R$ on each morphism space of $\GrCob$; these coefficient system satisfy a suitable compatibility condition, so that taking homology of morphisms spaces yields a new category $H_*(\GrCob,\xi_d^R)$: this is now a symmetric monoidal category enriched in $\Mod_R(\Ab^\Z)$, i.e. in the category of $\Z$-graded $R$-modules. The objects of $H_*(\GrCob,\xi_d^R)$ are again spaces, and the morphism object from $X$ to $Y$ is the graded $R$-module $H_*(\fM_\Gr(X,Y);\xi_d^R)$.

A formulation of the main result of the article (see Corollary \ref{cor:A1}) is the existence of a symmetric monoidal functor $H_*(\GrCob,\xi_d^R)^\op\to \Mod_R(\Ab^\Z),$ of symmetric monoidal categories enriched in $\Mod_R(\Ab^\Z)$, that attains the value $H_*(M^X;R)\in \Mod_R(\Ab^\Z)$ at the object $X\in H_*(\GrCob,\xi_d^R)^\op$, and whose behaviour on morphism objects generalises the operations (1)-(3).

\subsection{Graph field theories}
\label{subsec:graphfieldtheories}
We will use the language of $\infty$-categories throughout the article, by which we will be able to focus on the essence of the construction, and also to achieve a notable improvement of the above mentioned result:
\begin{itemize}
\item $R\in\CAlg(\Sp)$ will be a generic $E_\infty$-ring spectrum;
\item $M$ will be a generic oriented $R$-Poincar\'e duality space;
\item the above functor will be constructed at the ``chain level'', i.e. between suitable $\infty$-categories enriched in $\Mod_R:=\Mod_R(\Sp)$;
\item any oriented homotopy equivalence $M\simeq M'$ of Poincar\'e duality spaces will yield a natural equivalence between the corresponding functors; in particular, all constructed higher string operations are homotopy invariant.
\end{itemize}
\begin{rem}
All constructions of this article are natural in $R$ with respect to maps of $E_\infty$-ring spectra; in fact one can also replace $\Mod_R$ by a general presentably symmetric monoidal stable $\infty$-category $\cC$, and require that $M$ be an oriented $\cC$-Poincar\'e duality space. We refrain from discussing this level of generality, and fix $R$ once and for all throughout the article.
\end{rem}

In most constructions of the article, starting with the very definition of $\GrCob$,
in Section \ref{sec:loccoeffgraphcobspaces}, 
we rely on Theorem \ref{thm:BarkanSteinebrunner},
a recent result of Barkan--Steinebrunner \cite{BarkanSteinebrunner}.
For $d\in\Z$ and for an $E_\infty$-ring spectrum $R$, we consider the first delooping $\bB\Pic(R)\in\CAlg(\cS)$ of the symmetric monoidal space $\Pic(R)\subset\Mod_R$, and we construct a specific symmetric monoidal functor $\xi_d^R\colon\GrCob\to \bB\Pic(R)$, sending a graph cobordism $Y\to W\hookleftarrow X$ to an $R$-module which is non-canonically isomorphic to $R[d\cdot\chi(W,X)]$.
This will be our ``compatible choice'' of coefficient systems on the morphism spaces of $\GrCob$, see Remark \ref{rem:compatibleloccoef}.
When $R$ is a discrete ring, the coefficient system $\xi_d^R$ on the space $\fM_\Gr(X,Y)$ agrees with the one associating with a graph cobordism $Y\to W\hookleftarrow X$ the (suitably shifted) determinant of the relative homology $H_*(W,X;R)$, considered as a finitely generated, free $R$-module: see Proposition \ref{prop:xihatxi} for details, and Example \ref{ex:xifromliterature} for other instances of the functor $\xi_d^R$ from the literature when $R$ is a discrete commutative ring.

The previous discussion allows us to define a new symmetric monoidal category, denoted $\GrCob_!\xi_d^R$ and enriched in the symmetric monoidal $\infty$-category $\Mod_R$: the objects of $\GrCob_!\xi_d^R$ are still spaces as for $\GrCob$, and we have an equivalence of $R$-modules as follows, giving a description of morphism objects in $\GrCob_!\xi_d^R$:
\[
\GrCob_!\xi_d^R\ (Y,X)\simeq\fM_\Gr(Y,X)_!\xi_d^R\simeq\colim_{\fM_\Gr(Y,X)}\xi_d^R\in\Mod_R.
\]

Taking homotopy groups of morphism modules, we may define yet another category $H_*(\GrCob;\xi_d^R)$; the latter is a symmetric monoidal category enriched in $\Mod_{\pi_*(R)}(\Ab^\Z)$, the category of graded modules over the discrete graded ring $\pi_*(R)$. Objects of $H_*(\GrCob;\xi_d^R)$ are again spaces, and the morphism object from $Y$ to $X$ is the graded module given by the homology with twisted coefficients $H_*(\fM_\Gr(Y,X);\xi_d^R)$. The construction of these enriched symmetric monoidal categories is done in Section \ref{sec:GFTfunctor}, following the theory of Gepner--Haugseng \cite{GepnerHaugseng}. The main result of the article is the following.

\begin{athm}
 \label{thm:A}
Let $R$ be an $E_\infty$-ring spectrum and let $M$ be an oriented $R$-Poincar\'e duality space of dimension $d\ge0$. Then there is a symmetric monoidal $R$-linear functor, which we refer to as a ``graph field theory''
\[
 \GFT_M\colon(\GrCob_!\xi_d^R)^\op\to\Mod_R
\]
satisfying the following properties:
\begin{itemize}
 \item[(i)] $\GFT_M$ sends a space $X$ to $M^X\otimes R$;
 \item[(ii)] $\GFT_M$ recovers the basic operations (1)-(3) from Subsection \ref{subsec:basic} at the level of homology.
\end{itemize}
In particular, for any spaces $X$ and $Y$ we obtain a morphism in $\Mod_R$
\[
 (M^X\otimes R)\otimes_R (\colim_{\fM_\Gr(Y,X)}\xi_d^R)\to M^Y\otimes R.
\]
The construction is natural with respect to oriented homotopy equivalences of oriented $R$-Poincar\'e duality spaces.
\end{athm}
The proof of Theorem \ref{thm:A} up to property (ii) is achieved in Section \ref{sec:GFTfunctor}, relying on several constructions from Section \ref{sec:ModRPoincare}. Property (ii) is then established in Section 
\ref{sec:CSproduct}: this final step will in particular rely on the formalism of \cite{BHKK}, allowing one to study relative Poincar\'e duality via parametrised $R$-modules in spectra.

The last statement of Theorem \ref{thm:A} implies in particular that an oriented homotopy equivalence $f\colon M\simeq M'$ between $R$-oriented closed manifolds induces an $R$-linear symmetric monoidal equivalence $\GFT_M\cong\GFT_{M'}$ evaluating $f^X\otimes R$ at $X$. This gives in particular an alternative proof of the homotopy invariance of the Chas--Sullivan string product \cite{CohenKleinSullivan, Crabb, GruherSalvatore}.

Theorem \ref{thm:A} solves \cite[Conjecture 3.18]{HepworthLahtinen}.
By taking homotopy groups of morphism objects, we obtain the following corollary, dealing with \emph{plain} categories enriched in graded $\pi_*(R)$-modules in abelian groups (for $R$ a discrete ring, this simplifies to ``graded $R$-modules in abelian groups'').
\begin{acor}
\label{cor:A1}
In the hypotheses of Theorem \ref{thm:A}, there is a symmetric monoidal $\pi_*(R)$-linear functor $\GFT'_M\colon H_*(\GrCob;\xi_d^R)^\op\to\Mod_{\pi_*(R)}(\Ab^\Z)$ sending a space $X$ to $H_*(M^X;R)$ and recovering the operations (1)-(3) from Subsection \ref{subsec:basic}.In particular, for spaces $X$ and $Y$, we obtain higher homology string operations
\[
 H_*(M^X;R)\otimes_{\pi_*(R)} H_*(\fM_\Gr(Y,X);\xi_d^R)\to H_*(M^Y;R).
\]
All these higher string operations are invariant under oriented homotopy equivalences of oriented $R$-Poincar\'e duality spaces.
\end{acor}
We can in fact interpret the homotopy invariance in the last statement of Theorem \ref{thm:A} more systematically as follows. For a given oriented $R$-Poincar\'e duality space $M$, we can consider the classifying space $\bB\hAut^+(M)$ of orientation-preserving homotopy automorphisms of $M$; see Subsection \ref{subsec:hAut+} for details. We can then upgrade Theorem \ref{thm:A} to a map of spaces/$\infty$-groupoids
\[
\bB\hAut^+(M)\to\Fun^{\otimes,R}((\GrCob\otimes\xi_d^R)^\op,\Mod_R)^\simeq.
\]
Postcomposing with the evaluation at the endomorphisms of the monoidal unit $\emptyset$, we obtain the following corollary, which is originally due Berglund in the case $R=\Q$, and to Barkan--Steinebrunner in the general case, see Remark \ref{rem:historical2}
\begin{acor}[Berglund, Barkan--Steinebrunner]
\label{cor:A2}
In the hypotheses of Theorem \ref{thm:A} there is a natural $R$-linear pairing
\[
(\bB\hAut^+(M)\otimes R)\otimes_R(\colim_{\fM_\Gr(\emptyset,\emptyset)}\xi_d^R)\to R.
\]
Applying the hom-tensor adjunction and then passing to homotopy groups, we obtain a $\pi_*(R)$-linear map
\[
H_*(\fM_\Gr(\emptyset,\emptyset);\xi_d^R)\to H^*(\bB\hAut^+(M);R).
\]
\end{acor}
To interpret concretely the result of Corollary \ref{cor:A2}, we anticipate that $\fM_\Gr(\emptyset,\emptyset)$ contains, for all $n\ge0$, a connected component that is homotopy equivalent to $\bB\Out(F_n)$, the classifying space of the group of outer automorphism of a free group of rank $n$; this will follow from Corollary \ref{cor:GrinGrCob}. Hence Corollary \ref{cor:A2} gives in particular a map $H_*(\bB\Out(F_n);\xi_d^R)\to H^*(\bB\hAut^+(M);R)$.
For $R=\Q$, a similar map
\[
H_*(\bB\Out(F_n);\xi_d^\Q)\to H^*(\bB\hAut_\partial^+(M\setminus \mathring{D}^d);\Q)
\]
has been constructed in the cases in which $M$ is a connected sum of $g\ge0$ copies of $S^k\times S^k$ \cite{BerglundMadsen}, or a connected sum of $g$ copies of $S^k\times S^l$ with $k\neq l$ \cite{Stoll}.
\begin{rem}[A remark on attributions]
\label{rem:historical2}
Before the beginning of this project, Berglund has announced a generalisation of the results \cite{BerglundMadsen,Stoll}: through a \emph{generalized fiber integration map}, he has originally claimed the result of Corollary \ref{cor:A2} for $R=\Q$ and $M$ simply connected \cite{Berglund}. Slightly thereafter, Barkan--Steinebrunner \cite{BarkanSteinebrunner} have claimed the entire statement of Corollary \ref{cor:A2}; their proof is quite parallel to ours, the main difference being that we obtain Corollary \ref{cor:A2} as an immediate corollary of Theorem \ref{thm:A}, while Barkan--Steinebrunner prove it directly. 

In particular, the map $(\bB\hAut^+(M)\otimes R)\otimes_R(\colim_{\fM_\Gr(\emptyset,\emptyset)}\xi_d^R)\to R$ constructed by Barkan--Steinebrunner agrees with ours, and so do the characteristic classes of oriented $R$-Poincar\'e fibrations obtained from the twisted homology of $\bB\Out(F_n)$. Barkan--Steinebrunner are also able to prove that their characteristic classes agree with Berglund's ones (when the latter are defined) in even degrees; conjecturally this should hold in all degrees.
\end{rem}

\subsection{Morphism spaces in \texorpdfstring{$\GrCob$}{GrCob} and open-closed field theories}
The statement of Theorem \ref{thm:A} invites us to understand the homotopy type of the morphism spaces of $\GrCob$; as this $\infty$-category is introduced as a localisation of another $\infty$-category $\Gr_\cS$, the description of morphism spaces is not immediate.
We will give in Section \ref{sec:GrCobspaces} a relatively simple description of morphism spaces in $\GrCob$ which genuinely resembles Definition \ref{defn:graphcob}; building on this, we will in particular prove the following theorem.
\begin{athm}
\label{thm:B}
Let $\Cospan(\cS)$ denote the symmetric monoidal $\infty$-category of cospans of spaces.
Then the canonical functor $D_0\colon\GrCob\to\Cospan(\cS)$ (see Remark \ref{rem:D0localises} for details)
satisfies the following properties:
\begin{enumerate}
\item the induced map on core groupoids $\GrCob^\simeq\xrightarrow{\simeq}\Cospan(\cS)^\simeq$ is an equivalence; in particular we have $\GrCob^\simeq\simeq\cS^\simeq$;
\item for all $X,Y\in\cS$ with $X$ aspherical (i.e. $\pi_i(X)=0$ for all $i\ge2$ and all choices of basepoint), the induced map on morphism spaces
\[
\fM_\Gr(Y,X)\to\Cospan(\cS)(Y,X)
\]
exhibits the source as the subspace of the target spanned by those cospans $Y\to W\ot X$ such that $W$ can be obtained from $X$ by attaching finitely many cells of dimension at most 1.
\end{enumerate}
In particular $D_0$ restricts to an inclusion of $\infty$-categories $\GrCob_\asph\hto\Cospan(\cS)$, where $\GrCob_\asph\subseteq\GrCob$ denotes the symmetric monoidal full $\infty$-subcategory spanned by aspherical spaces.
\end{athm}
Theorem \ref{thm:B} is the combination of Corollary \ref{cor:GrCobcore} and Proposition \ref{prop:GuirardelLevitt}. 
\begin{rem}
For $X\in\cS$, the space $\fM_\Gr(\emptyset,X)$ was considered by Hatcher in a talk given in 2012 \cite{HatcherBanff}, via a different but equivalent construction. In this talk Hatcher asked a question that can be rephrased as follows: for which $X$ does $D_0$ induce an inclusion $\fM_\Gr(\emptyset,X)\hto(\cS_{X/})^\simeq$ on morphism spaces? Theorem \ref{thm:B} shows that this is the case when $X$ is aspherical; we will see instead in Example \ref{ex:Hatcher} that $\fM_\Gr(\emptyset,X\to(\cS_{X/})^\simeq$ is not an inclusion of spaces if $X=S^k$ is a sphere of dimension $k\ge2$, and we believe that the argument in this example can be generalised to any non-aspherical space. This solves Hatcher's question, notably in the positive for aspherical spaces. This also shows, however, that the entire functor $D_0\colon\GrCob\to\Cospan(\cS)^\simeq$ is not an inclusion of $\infty$-categories.
\end{rem}
\begin{rem}
As a simple and immediate application of Theorem \ref{thm:B}, we consider the symmetric monoidal $\infty$-category $\Cob_2^\del$ of cobordisms between oriented compact 1-manifolds with boundary, where cobordisms are oriented surfaces with boundary and corners. We further select the wide symmetric monoidal $\infty$-subcategory $\OC\subset\Cob_2^\del$ spanned by cobordisms whose handle dimension relative to the outgoing boundary is at most 1.\footnote{This condition is usually expressed by requiring that each component of the cobordism must have some incomining or free boundary.} 
The $\infty$-category $\OC$ has been extensively studied in the context of string topology and higher operations on Hochschild complexes of Frobenius algebras; an incomplete list of sources includes \cite{CostelloOC,Godin,Kaufmann3,WahlWesterland,WahlUniversal,BSZ}. According to Costello \cite{CostelloOC}, $\OC$ was first introduced by Segal \cite{SegalStanford} and Moore \cite{Moore}.

The canonical symmetric monoidal $\infty$-functor $\OC\to\Cospan(\cS)$, sending a cobordism of manifolds to the underlying cospan of spaces, evidently factors through $\GrCob_\asph^\op$, thus giving rise to a symmetric monoidal $\infty$-functor $\OC\to\GrCob^\op$. Composing the latter with $\GFT_M$, we obtain an open-closed field theory $\OC\to\Mod_R$ attaining values $\Sigma^\infty_+LM\otimes_\bS R$ and  $\Sigma^\infty_+M\otimes_\bS R$ at the circle and the interval, respectively.
This connection justifies our use of the terminology of ``graph cobordism'' and ``graph field theory'', which we take from \cite{HepworthLahtinen}.
The above open-closed field theory allows us to define string operations using the homology of moduli spaces of Riemann surfaces with boundary; we expect such operations to agree with analogous ones constructed so far in the literature \cite{CohenJones,CohenGodin,VoronovUniversalAlgebra,Chataur,CohenVoronov,Godin,PoirerRounds,DCPR,BSZ}.
\end{rem}

\subsection{Related work}
The results of this article generalise and complement other results in the literature; we make here a short and necessarily incomplete summary.

\begin{description}[style=unboxed,leftmargin=0cm]
\item[Classifying spaces of Lie groups] Hepworth--Lahtinen \cite{HepworthLahtinen} define similar category of graph cobordism, whose objects are graphs. They also associate with a compact Lie group $G$ a ``graph field theory'', attaining the value $H_*(\map(X,\bB G);\mathbb{F}_2)$ at the graph $X$. Our category $\GrCob$ is a variation of theirs. The main (parallel) differences are the following:
\begin{itemize}
\item we consider mapping spaces into $M$, while \cite{HepworthLahtinen} consider mapping spaces into $\bB G$; hence in their case it is the loop space of the target space, and not the target space itself, to be a Poincar\'e duality space.
\item in our graph cobordisms $Y\to W\hookleftarrow X$ we attach 0-cells and 1-cells to $X$; in \cite{HepworthLahtinen}, only 1-cells seem to be allowed (but we believe that an extended construction allowing 2-cells should be possible in their setting).
\end{itemize}
String operations for the homology of free loop spaces of classifying spaces of compact Lie groups have also been studied in \cite{ChataurMenichi}.
\item[Homotopy invariance] The invariance under homotopy equivalences of manifolds for string product and string bracket is proved in \cite{CohenKleinSullivan, Crabb, GruherSalvatore}. In this article we explicitly recover the invariance of the string product; similar arguments should recover the invariance of the string bracket.
\item[String operations for non-manifolds]
For an $R$-oriented $d$-dimensional Poinca\-r\'e duality space $M$, Klein \cite{Kleinfibre} defines an $E_1$-algebra structure on $\Sigma^{\infty}_+(LM)\otimes R[-d]$; this endows in particular $H_*(LM;R)$ with a string product, coinciding with the one of Chas--Sullivan when $M$ is a manifold, and gives an alternative proof of the homotopy invariance of the string product. Over the rationals, string operations for (triangulated) Poincar\'e duality spaces have been constructed in \cite{TradlerZeinalianSullivan}. 
String operations for orbifolds \cite{Lupercioetal}, Gorenstein spaces \cite{FelixThomas}, and differentiable stacks \cite{BGNX1,BGNX2} have also been constructed.
\item[Homology of moduli of Riemann surfaces] Higher string operations on the homology of the free loop space of a manifold have been extensively constructed using moduli spaces of Riemann surfaces, or various versions of chord diagrams; in some cases, lifts of these operations to the spectral/chain level are also provided, as well as generalisations to arbitrary multiplicative (co)homology theories \cite{CohenJones,CohenGodin,VoronovUniversalAlgebra,Chataur,CohenVoronov,Godin,PoirerRounds,DCPR,BSZ}
\item[String operations on Hochschild homology] Operations on the Hochschild (co)homology or the Hochschild (co)chains of (differentiable) graded (commutative, Frobenius) algebras, parametrised by various complexes of graphs and chord diagrams, have been studied by several authors \cite{FTVP,CostelloOC,TradlerZeinalian,TradlerZeinalian2,CostelloTradlerZeinalian,Kaufmann1,Kaufmann2,Kaufmann3,WahlWesterland,WahlUniversal}
\end{description}
\begin{rem}[The string coproduct]
Together with the string product, the string coproduct $H_{*+d-1}(LM,M)\to H_*((LM,M)^2)$ introduced by Goresky--Hingston \cite{GoreskyHingston} is arguably the most studied string operation in the literature. The string coproduct can be made into an non-relative operation $H_{*+d-1}(LM)\to H_*(LM^2)$ \cite{HingstonWahl}. It has originally been defined under the assumption that $M$ is a closed manifold, but it can in fact be defined also when $M$ is an oriented homology manifold \cite{RiveraTakeda}, and even when $M$ is an oriented Poincar\'e duality spaces endowed with a refinement of the diagonal $\Delta\colon M\to M\times M$ to a Poincar\'e embedding, by a straightforward generalisation of \cite[Section 4.3]{NaefSafronov}. Naef \cite{Naef} has proved that the string coproduct is in general \emph{not} invariant under homotopy equivalences of manifolds; see also the discussion in \cite{NaefRiveraWahl}. We mention however that the string coproduct is invariant under \emph{simple} homotopy equivalences \cite{NaefSafronov,KenigsbergPorcelli}; see also \cite{WahlInvariance}.
Combining Naef's result with the last statement of Theorem \ref{thm:A}, we deduce that $\GFT_M$ \emph{cannot possibly recover the string coproduct}.
\end{rem}

\subsection{Acknowledgments}
I would first like to thank Nathalie Wahl for introducing me to the world of string topology and for several comments and corrections related to a first version of the results of this article. I am further deeply indebted with Jan Steinebrunner for several discussions and explanations about his work in progress with Shaul Barkan, most crucially the statement of Theorem \ref{thm:BarkanSteinebrunner}, which has allowed for a dramatic simplification of the current article in comparison with its original draft (where we meant to use directly the $\inftwo$-categorical \cite[Theorem A]{Bianchi:graphcobset}).
I have moreover benefited from conversations with
Shaul Barkan,
Alexander Berglund,
Schachar Carmeli,
Bastiaan Cnossen, 
Johannes Ebert,
Allen Hatcher, 
Kaif Hilman, 
Dominik Kirstein, 
Christian Kremer, 
Anssi Lahtinen, 
Markus Land, 
Tobias Lenz, 
Sil Linskens, 
Florian Naef, 
Nils Prigge, 
Maxime Ramzi,
Manuel Rivera,
Peter Teichner and Adela Zhang.

This work was partially supported
by the
European Research Council under the European Union’s Horizon 2020 research and innovation
programme (grant agreement No. 772960), by the Danish National Research Foundation
through the Copenhagen Centre for Geometry and Topology (DNRF151), and by the Max Planck Institute for Mathematics in Bonn.

\section{Preliminaries}
\label{sec:preliminaries}
In this section we recollect most of the background material that we will use throughout the article. Some of the material has appeared in essentially the same form in \cite{Bianchi:graphcobset}, and is included here for the sake of completeness. We encourage the reader to proceed directly to the next section on a first reading of the article.

We fix three nested Grothendieck universes 
$\fU_1\subsetneq\fU_2\subsetneq\fU_3$, 
and we refer to objects whose size is bounded above by $\fU_1$, $\fU_2$ and $\fU_3$ as ``small'', ``large'' and ``huge'', respectively.
For $i=1,2$, an $\infty$-category is \emph{$\fU_i$-presentable} if it is $\fU_{i+1}$-small, admits $\fU_i$-small colimits, and is $\kappa$-accessible for some regular cardinal $\kappa<\fU_i$. We abbreviate ``$\fU_1$-presentable'' by ``presentable''.

\subsection{Spaces and \texorpdfstring{$\infty$}{infty}-categories}
For $i=1,2$ we denote by $\cS^{\fU_i}$ the $\fU_i$-presentable $\infty$-category of $\fU_i$-small $\infty$-groupoids; for $i=1$ we write $\cS=\cS^{\fU_1}$ and refer to it as the $\infty$-category of ``spaces''.
We will sometimes consider $\Top$, the (ordinary) category of compactly generated weakly Hausdorff topological spaces with $\fU_1$-small underlying set, and use ``topological space'', as opposed to just ``space'', to denote an object of $\Top$.

For $i=1,2$ we denote by $\Catinfty^{\fU_i}$ the $\fU_i$-presentable $\infty$-category of $\fU_i$-small $\infty$-categories. For $i=1$ we abbreviate $\Catinfty=\Catinfty^{\fU_1}$.
The fully faithful inclusion $\cS\hookrightarrow\Catinfty$ has a left adjoint $|-|\colon\Catinfty\to\cS$, called the \emph{classifying space functor}, and a right adjoint $(-)^\simeq\colon\Catinfty\to\cS$, called the \emph{core groupoid functor}.

For $\cC,\cD\in\Catinfty$ we denote by $\Fun(\cC,\cD)\in\Catinfty$ the (small) $\infty$-category of functors from $\cC$ to $\cD$ and natural transformations; it is the internal hom in $\Catinfty$, which is cartesian closed. The morphism \emph{space} from $\cC$ to $\cD$ in $\Catinfty$ is recovered as the core groupoid $\Fun(\cC,\cD)^\simeq$.

For $i=1,2$ we denote by $\CAlg(\Catinfty^{\fU_i})$ the $\fU_i$-presentable $\infty$-category of $\fU_i$-small symmetric monoidal $\infty$-categories, with morphisms given by symmetric monoidal $\infty$-functors. This can be equivalently defined as the category of commutative monoid objects \cite[Definition 2.4.2.1]{HA}, or of algebras over the $E_\infty$-operad, in $\Catinfty^{\fU_i}$: for this we use that $\Catinfty^{\fU_i}$ is considered as a \emph{cartesian} symmetric monoidal $\infty$-category: see \cite[Proposition 2.4.2.5]{HA}. We abbreviate $\CAlg(\Catinfty)=\CAlg(\Catinfty^{\fU_1})$.

For $\cC\in\CAlg(\Catinfty^{\fU_i})$ we usually denote by $\one\in\cC$ the monoidal unit and we denote by $-\otimes-\colon\cC\times\cC\to\cC$ the monoidal product.
When the monoidal structure on $\cC$ comes from a categorical coproduct, either on $\cC$ or on a closely related $\infty$-category, we will instead use the notations $\emptyset$ and $-\sqcup-$ for the monoidal unit and the monoidal product of $\cC$.
For $\cC,\cD\in\CAlg(\Catinfty^{\fU_i})$ we let $\Fun^\otimes(\cC,\cC')^\simeq\in\cS^{\fU_i}$ denote the corresponding morphism space in $\CAlg(\Catinfty^{\fU_i})$.

We denote by $\Fin$ the category of finite sets;
we regard $\Fin$ as a full $\infty$-subcatego\-ry of $\cS$.
For $k\ge0$ we denote by $\ul k\in \Fin$ the finite set $\set{1,\dots,k}$. We have $\emptyset=\ul 0$.
We usually endow $\Fin\in\Catinfty$ and $\cS\in\Catinfty^{\fU_2}$ with the coproduct symmetric monoidal structure; when $\cS$ is endowed with the product symmetric monoidal structure, we stress this by writing $(\cS,\times)$.

We denote by $[1]$ the category given by the linearly ordered set with two elements $0<1$; for an $\infty$-category $\cC$ we denote by $d_1,d_0\colon\Fun([1],\cC)\to\cC$ the source and target functor, respectively.

\subsection{Presheaves}
\label{subsec:presheaves}
For $\cC\in\Catinfty$ we let $\PSh(\cC):=\Fun(\cC^\op,\cS)$ denote the category of space-valued presheaves over $\cC$, and we denote by $\yo\colon\cC\to\PSh(\cC)$ the Yoneda embedding. For $\cC\in\CAlg(\Catinfty)$, the category $\PSh(\cC)$ is endowed with a symmetric monoidal structure $-\otimes_\Day-$ given by \emph{Day convolution}, which can be characterised by the following properties \cite[Corollary 4.8.1.12]{HA}:
\begin{itemize}
\item the Yoneda embedding can be promoted to a symmetric monoidal functor;
\item the monoidal product $\PSh(\cC)\times\PSh(\cC)\to\PSh(\cC)$ preserves colimits separately in each variable.
\end{itemize}
These properties give rise to the following formula: for objects $A\simeq\colim_{i\in I}\yo(x_i)$ and $B\simeq\colim_{j\in J}\yo(x'_j)$ in $\PSh(\cC)$, written as colimits of representable presheaves, we have an equivalence
\[
A\otimes_{\Day}B\simeq\colim_{(i,j)\in I\times J}\yo(x_i\otimes x'_j).
\]

A functor $F\colon\cC\to\cC'$ induces a functor $F^*\colon\PSh(\cC')\to\PSh(\cC)$ by restriction, with left adjoint $F_!\colon\PSh(\cC)\to\PSh(\cC')$ given by left Kan extension. The functor $F_!$ is symmetric monoidal, whereas $F^*$ is in general only lax symmetric monoidal.

Suppose now that $\cD$ is a \emph{presentably symmetric monoidal} $\infty$-category, i.e. $\cD$ is $\fU_1$-presentable and the monoidal product $\cD\times\cD\to\cD$ preserves $\fU_1$-small colimits in each variable separately.
Then we may fix a regular cardinal $\kappa<\fU_1$ such that $\cD$ is $\kappa$-accessible, set $\cC:=\cD^\kappa\subset\cD$ for the ($\fU_1$-small) full symmetric monoidal $\infty$-subcategory spanned by $\kappa$-compact objects, and consider the (still fully faithful) restricted Yoneda embedding $\yo_\kappa\colon\cD\to\PSh(\cC)$. The functor $\yo_\kappa$ is still endowed with a symmetric monoidal structure, and moreover it preserves all limits and $\kappa$-filtered colimits, in particular it is accessible. By \cite[Corollary 5.5.2.9]{LurieHTT} we have that $\yo_\kappa$ admits a left adjoint 
$\yo'_\kappa\colon\PSh(\cC)\to\cD$, which is given by the following formula: if $A\simeq\colim_{i\in I}\yo(x_i)\in\PSh(\cC^\kappa)$, then $\yo'_\kappa(A)\simeq\colim_{i\in I}x_i$.

From this and the formula for $A\otimes_\Day B$ given above, together with the hypothesis that the monoidal product on $\cD$ preserves colimits in each variable, we obtain for $A,B\in\PSh(\cC)$ as above the chain of equivalences
\[
\begin{split}
\yo'_\kappa(A\otimes_\Day B)&\simeq\yo'_\kappa(\colim_{(i,j)\in I\times J}\yo_\kappa(x_i\otimes x'_j))\simeq\colim_{(i,j)\in I\times J}x_i\otimes x'_j\\
&\simeq(\colim_{i\in I}x_i)\otimes(\colim_{j\in J}x'_j)\simeq\yo'_\kappa(A)\otimes\yo'_\kappa(B),
\end{split}
\]
from which we learn that the oplax symmetric monoidal structure on the left adjoint $\yo'_\kappa$ is in fact a symmetric monoidal structure.

\subsection{Localisation of \texorpdfstring{$\infty$}{infty}-categories}
\label{subsec:localisation}
The following material appears in essentially the same form in \cite{Bianchi:graphcobset}.
\begin{defn}
\label{defn:RelCatinfone}
We denote by $\RelCatinfty\subset\Fun([1],\Catinfty)$ the full $\infty$-subcategory spanned by wide subcategory inclusions $W\to\cC$, i.e. functors inducing an equivalence on core groupoids and inclusions on morphism spaces.
An object in $\RelCatinfty$ is usually denoted as a pair $(\cC,W)$, and is called a \emph{relative $\infty$-category}.

The fully faithful inclusion $\Catinfty\hto\RelCatinfty$ sending $\cC\mapsto(\cC,\cC^\simeq)$ admits a left adjoint $L\colon\RelCatinfty\to\Catinfty$, sending $(\cC,W)$ to the localisation of $\cC$ at all morphisms in $W$.
\end{defn}

The following is an adaptation of \cite[Chapter 7.2]{Cisinski} for \emph{left} calculi of fractions.
\begin{defn}
\label{defn:leftcalculus}
Let $(\cC,W)\in\RelCatinfty$ and let $x\in\cC$. A \emph{left calculus of fractions at $x$} is a functor $\pi(x)\colon W(x)\to\cC_{x/}$ satisfying the following properties:
\begin{itemize}
\item the $\infty$-category $W(x)$ admits an initial object $x_0$, and $\pi(x)\colon x_0\mapsto\Id_x$;
\item $\pi(x)$ takes value in the full subcategory of $\cC_{x/}$ spanned by arrows that are contained in $W$;
\item the left Kan extension functor $(d_0\pi(x))_!\colon\PSh(W(x))\to\PSh(\cC)$ sends the terminal presheaf $*\in\PSh(W(x))$ to a presheaf $(d_0\pi(x))_!(*)\in\PSh(\cC)$ that factors through $L(\cC,W)$, i.e. morphisms in $W^\op$ are sent to invertible morphisms in $\cS$ along $(d_0\pi(x))_!(*)$.
\end{itemize}
\end{defn}
We remark that, for any functor $\pi(x)\colon W(x)\to\cC_{x/}$ and for $y\in\cC$, we have $(d_0\pi(x))_!(*)(y)\simeq \colim_{z\in W(x)}\cC(y,d_0\pi(z))$; we can also consider the $\infty$-category $\cC_{y/}\times_{\cC}W(x)$, where the fibre product is taken along the functors $d_0\colon\cC_{y/}\to\cC$ and $d_0\pi\colon W(x)\to\cC$; then we have an equivalence $(d_0\pi(x))_!(*)(y)\simeq|\cC_{y/}\times_{\cC}W(x)|$.
\begin{thm}[{\cite[Theorem 7.2.8]{Cisinski}}]
\label{thm:Cisinski}
Let $(\cC,W)$ be a relative $\infty$-category, let $x\in\cC$ and let $\pi(x)\colon W(x)\to\cC_{x/}$ be a left calculus of fractions at $x$. Then we have an equivalence in $\PSh(\cC)$
\[
L(\cC,W)(-,x)\simeq (d_0\pi(x))_!(*).
\]
\end{thm}

\subsection{\texorpdfstring{$\inftwo$}{inftwo}-categories, deloopings and Gray tensor product}
\label{subsec:deloopingGray}
For $i=1,2$ we denote by $\Catinftwo^{\fU_i}\in\Catinfty^{\fU_{i+1}}$ the $\infty$-category of $\fU_i$-small $\inftwo$-categories; it is endowed with the cartesian symmetric monoidal structure.
We denote by $|-|_2$ and $(-)^\twosimeq\colon\Catinftwo^{\fU_i}\to\Catinfty^{\fU_i}$ the left and right adjoint to the inclusion $\Catinfty\subset\Catinftwo$ of $\fU_{i+1}$-small $\infty$-categories, respectively. We further write $|-|$ and $(-)^\simeq\colon\Catinftwo^{\fU_i}\to\cS^{\fU_i}$ for the left and right adjoint to the inclusion $\cS^{\fU_i}\subset\Catinftwo^{\fU_i}$, respectively.

For $\cC\in\CAlg(\Catinfty^{\fU_i})$ we let $\bB\cC\in\CAlg(\Catinftwo^{\fU_i})$ denote the \emph{first delooping} of $\cC$, i.e. the essentially unique symmetric monoidal $\inftwo$-category with an equivalence
$\bB\cC(\one_{\bB\cC},\one_{\bB\cC})\simeq\cC$ and such that $(\bB\cC)^\simeq$ is a connected space.
Taking the first delooping gives a fully faithful inclusion $\bB\colon\CAlg(\Catinfty)\hto\CAlg(\Catinftwo)$ with essential image given by the symmetric monoidal $\inftwo$-categories whose space of objects is connected.
We usually denote by $*=\one_{\bB\cC}$ the monoidal unit of a first delooping, which we regard as a ``base object''.

If $\cC$ is a symmetric monoidal $\infty$-groupoid, then $\bB\cC\in\CAlg(\Catinfty^{\fU_i})$; and if $\cC$ is moreover group-like with respect to the tensor product, then $\bB\cC\in\CAlg(\cS^{\fU_i})$.

For $1\le i\le 2$ we denote by $\circledast$ the Gray tensor product on $\Catinftwo^{\fU_i}$; it gives $\Catinftwo^{\fU_i}$ a (non-symmetric) monoidal structure, which is \emph{not} equivalent to the cartesian symmetric monoidal structure $\infty$-category. An introduction to the subject can be found in \cite[Chapter 10, Section 3]{GR} and in \cite{GHL}.

For $\cC\in\Catinftwo^{\fU_i}$, the endofunctors $-\circledast\cC$ and $\cC\circledast-$ of $\Catinftwo^{\fU_i}$ admits right adjoints that we denote $\Fun(\cC,-)^\llax$ and $\Fun(\cC,-)^\rlax$.
Both in $\Fun(\cC,\cD)^\llax$ and in $\Fun(\cC,\cD)^\rlax$, an object is an $\inftwo$-functor $\cC\to\cD$; the difference between the two arises when considering 1-morphisms and 2-morphisms.

For instance, let $F,G$ be $\inftwo$-functors $\cC\to\cD$, and let $f\colon x\to y$ be a 1-morphism in $\cC$; then we obtain 1-morphisms $F(f)$ and $G(f)$ in $\cD$; let now $\alpha\colon F\to G$ be a 1-morphism in $\Fun(\cC,\cD)^\llax$ or in $\Fun(\cC,\cD)^\rlax$: then $\alpha$ produces in both cases 1-morphisms $\alpha_x\colon F(x)\to G(x)$ and $\alpha_y\colon F(y)\to G(y)$ in $\cD$, together with a 2-morphism filling the following square
\[
 \begin{tikzcd}[row sep=30pt, column sep =70 pt]
  F(x)\ar[r,"\alpha_x"]\ar[d,"F(f)"',""{name=tL,inner sep=2pt,right},""{name=sR,inner sep=2pt,right,very near end}]& G(x)\ar[d,""{name=sL,inner sep=2pt,left,very near start},""{name=tR,inner sep=2pt,left},"G(f)"] \\
  F(y)\ar[r,"\alpha_y"']&G(y);
  \ar[Rightarrow,from=sL, to=tL,"\llax"']\ar[Rightarrow, from=sR,to=tR,"\rlax"']
 \end{tikzcd}
\]
the difference between the two cases is the orientation of this 2-morphism.

There is a natural transformation $-\circledast-\Rightarrow-\times-$ of functors $(\Catinftwo^{\fU_i})^2\to\Catinftwo^{\fU_i}$; evaluating one variable at $\cC\in\Catinftwo^{\fU_i}$ and passing to right adjoints, we obtain natural transformations $\Fun(\cC,-)\Rightarrow\Fun(\cC,-)^\llax$ and $\Fun(\cC,-)\Rightarrow\Fun(\cC,-)^\rlax$ of endofunctors of $\Catinftwo^{\fU_i}$.

If $\cD\in\CAlg(\Catinftwo^{\fU_i})$ is a symmetric monoidal $\inftwo$-category, then all of $\Fun(\cC,\cD)$, $\Fun(\cC,\cD)^\llax$ and $\Fun(\cC,\cD)^\rlax$ are symmetric monoidal $\inftwo$-catego\-ries by pointwise operations, and the two natural $\inftwo$-functors $\Fun(\cC,\cD)^\llax\ot\Fun(\cC,\cD)\to\Fun(\cC,\cD)^\llax$ are symmetric monoidal.

\begin{nota}
\label{nota:laxcomma}
For $1\le i\le 2$ and for $\cC\in\Catinftwo^{\fU_i}$ we denote by 
\[
d_0,d_1\colon\Fun([1],\cC)^\llax\to\cC
\]
the target and source functors. For $x\in\cC$ we denote by $\cC_{x/}^\llax$ the pullback in $\Catinftwo^{\fU_i}$ of the cospan $*\overset{x}{\to}\cC\overset{d_1}{\ot}\Fun([1],\cC)^\llax$, and by $\cC_{/x}^\llax$ the pullback of the cospan $*\overset{x}{\to}\cC\overset{d_0}{\ot}\Fun([1],\cC)^\llax$. We define similarly $\cC_{/x}^\rlax$ and $\cC_{x/}^\rlax$.
If $\cC$ is symmetric monoidal and $x=\one\in\cC$ is the monoidal unit, then all $\inftwo$-categories $\cC_{x/}^\llax$, $\cC_{/x}^\llax$, $\cC_{x/}^\rlax$ and $\cC_{/x}^\rlax$ have an induced symmetric monoidal structure.
\end{nota}

\subsection{Stable presentable \texorpdfstring{$R$}{R}-linear categories}
\label{subsec:SPrLR}
We denote by $\Sp\in\Catinfty^{\fU_2}$ the stable, presentable $\infty$-category of spectra.
We further let $\SPrL\in\Catinftwo^{\fU_2}$ denote the large $\inftwo$-category having stable, presentable $\infty$-categories as objects, left adjoint functors as 1-morphisms, and natural transformations as 2-morphisms. Note that objects of $\SPrL$ are only $\fU_2$-small as $\infty$-categories (the only exception being $*\in\SPrL$, which is also $\fU_1$-small); however an $\fU_1$-presentable $\infty$-category can be described through an $\fU_1$-small amount of data, so the core groupoid of $\SPrL$ is $\fU_2$-small; similarly, the spaces of 1-morphisms and 2-morphisms in $\SPrL$ are $\fU_2$-small (in fact even $\fU_1$-small); this justifies our writing $\SPrL\in\Catinftwo^{\fU_2}$.

The $\inftwo$-category $\SPrL$ is endowed with a symmetric monoidal structure, introduced in \cite[Section 4.8]{HA} for the underlying $\infty$-category $(\SPrL)^\twosimeq$, and extended to the $\inftwo$-categorical setting in \cite[part I, chapter 1, section 8.3]{GR} and \cite[Section 4.4]{HSS}.

We usually denote by $R\in\CAlg(\Sp)$ an $E_\infty$-ring spectrum; this includes the case of a discrete commutative ring $R$, which is identified with its Eilenberg--Maclane $E_\infty$-ring spectrum.
We denote by $\Mod_R=\Mod_R(\Sp)\in\SPrL$ the $\infty$-category of $R$-modules in spectra. If $R$ is discrete, $\Mod_R$ is equivalent to the derived $\infty$-category $D(R)$; the category of $R$-modules in abelian groups shall instead be denoted by $\Mod_R(\Ab)$.
For $\zeta\in\Mod_R$, we denote by $\zeta[n]$ the $n$\textsuperscript{th} suspension of $\zeta$.

Tensor product over $R$ makes $\Mod_R$ into an object in $\CAlg(\SPrL)$, i.e. into a \emph{presentably symmetric monoidal stable} $\infty$-category. We denote by $\SPrL_R=\Mod_{\Mod_R}(\SPrL)\in\CAlg(\Catinftwo^{\fU_2})$ the symmetric monoidal $\inftwo$-category of $R$-linear presentable $\infty$-categories, with $R$-linear left adjoint functors as 1-morphisms and $R$-linear natural transformations as 2-morphisms. 

We let $\Pic(R)\subset\Mod_R^{\simeq}$ denote the full $\infty$-subgroupoid of the core groupoid of $\Mod_R$ spanned by tensor-invertible objects. Invertible modules are in particular dualisable and hence compact, so we have $\Pic(R)\subset\Mod_R^\omega$; in particular $\Pic(R)$ is equivalent to a $\fU_1$-small $\infty$-groupoid, so we can write $\Pic(R)\in\cS$. Tensor product over $R$ makes $\Pic(R)$ into a group-like symmetric monoidal $\infty$-groupoid, with first delooping $\bB\Pic(R)\in\cS$ (see the notation from Subsection \ref{subsec:deloopingGray}).

\subsection{Parametrised \texorpdfstring{$R$}{R}-modules}
We fix $R\in\CAlg(\Sp)$. For a space $X$ we denote by $\Mod_R^X=\Fun(X,\Mod_R)$ the $\infty$-category of parametrised $R$-modules over $X$, where $X$ is considered as an $\infty$-groupoid.
Pointwise tensor product over $R$ makes $\Mod_R^X$ into a symmetric monoidal $R$-linear $\infty$-category, i.e. a commutative algebra object in $\SPrL_R$. We denote by $\Pic(\Mod_R^X)$ the full $\infty$-subgroupoid of $(\Mod_R^X)^\simeq$ spanned by invertible objects.

Given a map of spaces $f\colon X\to Y$, we denote by $f^*=\Fun(f,\Mod_R)$ the induced symmetric monoidal functor $\Mod_R^Y\to\Mod_R^X$. It is a morphism in $\SPrL_R$.

Taking categories of parametrised $R$-modules gives a symmetric monoidal functor $\Mod_R^{(-)}:=\Fun(-,\Mod_R)\colon(\cS,\times)^\op\to\CAlg(\SPrL_R)$: in particular, for two spaces $X,Y$, there is a canonical symmetric monoidal and $R$-linear identification $\Mod_R^X\otimes_{\Mod_R}\Mod_R^Y\simeq\Mod_R^{X\times Y}$.

For a map of spaces $f\colon X\to Y$, the functor $f^*$ admits a left and a right adjoint in $\Catinftwo^{\fU_2}$, denoted $f_!,f_*\colon \Mod_R^X\to\Mod_R^Y$, corresponding to computing colimits and limits on the (homotopy) fibres of $f$. The morphism $f_!$ is also in $\SPrL_R$, i.e. it is colimit preserving and $R$-linear, whereas the morphism $f_*$ does not preserve $\fU_1$-small colimits in general, and it only carries a lax $R$-linear structure in general.

\begin{nota}
For $X\in\cS$ we denote by $p(X)\colon X\to *$ the terminal map; it induces functors $p(X)^*\colon \Mod_R\to\Mod_R^X$ and $p(X)_!,p(X)_*\colon\Mod_R^X\to \Mod_R$.

For $\zeta\in\Mod_R$ we sometimes abuse notation and write $\zeta$ for $p(X)^*(\zeta)\in\Mod_R^X$, i.e. the constant parametrised $R$-module over $X$ with value $\zeta$.
\end{nota}
\begin{nota}
 We usually denote by $\eta_{f_*}\colon\Id\to f_*f^*$ the unit of the adjunction involving $f_*$ (which is the adjunction between $f^*$ and $f_*$), by $\epsilon_{f_!}\colon f_!f^*\to\Id$ the counit of the adjunction involving $f_!$, and similarly for $\eta_{f_!}$ and $\epsilon_{f_*}$.
\end{nota}

Each $\zeta\in\Mod_R$ gives rise to a homology theory on spaces, associating with $X\in\cS$ the sequence of homotopy groups $H_*(X;\zeta):=\pi_*(p(X)_!p(X)^*\zeta)$.
Given a map of spaces $f\colon X\to Y$, we have $p(X)\simeq p(Y)f$ and we obtain a map $H_*(X;\zeta)\to H_*(Y;\zeta)$ by taking homotopy groups of the evaluation at $\zeta$ of the following composite of natural transformations between endofunctors of $\Mod_R$:
\[
\begin{tikzcd}[column sep=50pt]
 p(X)_!p(X)^*\simeq p(Y)_!f_!f^*p(Y)^*\ar[r,"p(Y)_!\epsilon_{f_!}p(Y)^*"] &p(Y)_!p(Y)^*.
\end{tikzcd}
\]
There is also a generalised cohomology theory $H^*(X;\zeta):=\pi_{-*}(p(X)_*p(X)^*(\zeta))$, whose functoriality with respect to maps $f\colon X\to Y$ relies on the units $\eta_{f_*}$.

\begin{nota}
\label{nota:fibcofalpha}
For a natural transformation $\alpha\colon F\Rightarrow G$ of functors $F,G\colon \cC\to\cD$, with $\cD$ a stable $\infty$-category, we denote by $\fib(\alpha),\cof(\alpha)\colon\cC\to\cD$ the functors sending an object $x\in\cC$ to the fibre, respectively the cofibre, of $\alpha_x\colon F(x)\to G(x)$.
\end{nota}
Given a map of spaces $i\colon Y\to X$, which we think of as an inclusion, and given $\zeta\in\Mod_R$, we can similarly define the relative homology groups $H_*(X,Y;\zeta)$ as the homotopy groups of $p(X)_!\cof(\epsilon_{i_!})p(X)^*(\zeta)$; similarly, the relative cohomology groups $H^*(X,Y;\zeta)$ are defined as $\pi_{-*}\big(p(X)_*\fib(\eta_{i_*})p(X)^*(\zeta)\big)$.

If $R$ is a discrete commutative ring and $\zeta\in\Mod_R(\Ab)\subset\Mod_R$, then $H_*(X;\zeta)$, $H^*(X;\zeta)$, $H_*(X,Y;\zeta)$ and $H^*(X,Y;\zeta)$ are isomorphic to the usual singular homology and cohomology of $X$ and $(X,Y)$ with coefficients in $\zeta$.

Homology and cohomology with twisted coefficients are defined by starting with a parametrised $R$-module $\zeta\in \Mod_R^X$ (e.g., a classical coefficient system $\zeta\colon X\to\Mod_R(\Ab)\subset\Mod_R$) and by omitting the terms ``$p(X)^*$'' in the above formulas.
\subsection{Beck--Chevalley transformations and projection formulae}
\label{subsec:BCprojformulae}
Given a pullback square of spaces
\[
 \begin{tikzcd}[row sep=15pt]
  A\ar[r,"\tilde g"]\ar[d,"\tilde f"']\ar[dr,phantom,"\lrcorner",very near start]&B\ar[d,"f"]\\
  C\ar[r,"g"]&D,
 \end{tikzcd}
\]
the following standard facts hold:
\begin{enumerate}
 \item by commutativity of the square we have equivalences of functors $\tilde g^*f^*\simeq \tilde f^*g^*$, $g_!\tilde f_!\simeq f_!\tilde g_!$ and $g_*\tilde f_*\simeq f_*\tilde g_*$;
 \item the following composites of natural transformations 
 are equivalences of functors: we will refer to them as \emph{Beck--Chevalley equivalences}:
 \[
 \begin{tikzcd}[column sep=40pt, row sep=10pt]
  \tilde g_!\tilde f^*\ar[r,"\tilde g_!\tilde f^*\eta_{ g_!}"]&\tilde g_!\tilde f^* g^* g_!\simeq \tilde g_!\tilde g^*f^* g_!\ar[r,"\epsilon_{g_!}f^*g_!"]& f^*g_!;\\
  f^*g_*\ar[r,"f^*g_*\eta_{\tilde f_*}"] &f^*g_*\tilde f_*\tilde f^*\simeq f^*f_*\tilde g_*\tilde f^*\ar[r,"\epsilon_{g_*}\tilde g_*\tilde f^*"] &\tilde g_*\tilde f^*,
 \end{tikzcd}
 \]
 and we will usually denote them by writing $\tilde g_!\tilde f^*\overset{BC}{\simeq}f^*g_!$ and $f^*g_*\overset{BC}{\simeq}\tilde g_*\tilde f^*$;
 \item using the previous, we obtain a composite natural transformation as follows
 \[
  \begin{tikzcd}[column sep=40pt]
   g_!\tilde f_*\ar[r,"\eta_{f_*}g_!f_*"]&f_*f^*g_!\tilde f_*\overset{BC}{\simeq}f_*\tilde g_!\tilde f^*\tilde f_*\ar[r,"f_*\tilde g_!\epsilon_{\tilde f_*}"] &f_*\tilde g_!,
  \end{tikzcd}
 \]
and we shall also denote this by $g_!\tilde f_*\overset{BC}{\to}f_*\tilde g_!$; this last transformation is in general not an equivalence.
\end{enumerate}

Given a map of spaces $f\colon X\to Y$, we already mentioned that $f^*\colon\Mod_R^Y\to\Mod_R^X$ is a symmetric monoidal functor; moreover the first of the following composite natural transformations of functors $\Mod_R^X\times\Mod_R^Y\to\Mod_R^Y$ is an equivalence, whereas the second is whenever $f$ has compact fibres:
\[
f_!(-\otimes_Rf^*(-))\xrightarrow{\eta_{f_!}} f_!(f^*f_!(-)\otimes_R f^*(-))\simeq f_!f^*(f_!(-)\otimes_R-)\xrightarrow{\epsilon_{f_!}}f_!(-)\otimes_R-;
\]
\[
f_*(-)\otimes_R-\xrightarrow{\eta_{f_*}}f_*f^*(f_*(-)\otimes_R-)\simeq f_*(f^*f_*(-)\otimes_R f^*(-))\xrightarrow{\epsilon_{f_*}}f_*(-\otimes_Rf^*(-)).
\]

\subsection{Cospans and spans}
\label{subsec:cospansspans}
Let $\cC$ be an $\infty$-category admitting finite colimits. We denote by $\bCospan(\cC)$ the $\inftwo$-category of cospans in $\cC$; we briefly recall the ideas of the construction. Objects in $\bCospan(\cC)$ are the same as in $\cC$. The $\infty$-category of morphisms from $x$ to $y$ in $\bCospan(\cC)$ is $\cC_{(x\sqcup y)/}$;  for instance
a 1-morphism $x\to y$ in $\bCospan(\cC)$ is a cospan diagram $x\to z\ot y$, and a 2-morphism from $x\to z\ot y$ to $x\to z'\ot y$ is a commutative diagram in $\cC$ of the following form:
\[
 \begin{tikzcd}[row sep=10pt]
  x\ar[r]\ar[d]&z\ar[dl]\\
  z'&y\ar[u]\ar[l].
 \end{tikzcd}
\]
The horizontal composition of 1-morphisms $x\to z\ot x'$ and $x'\to z'\ot x''$ is $x\to t\ot x''$, where $t$ is the pushout in $\cC$ of the span diagram $z\ot x'\to z'$.
We usually denote $\Cospan(\cC):=\bCospan(\cC)^\twosimeq$ the underlying $\infty$-category of $\bCospan(\cC)$, in which non-invertible 2-morphisms have been discarded.

The coproduct symmetric monoidal structure on $\cC$ gives rise to a symmetric monoidal structure on $\bCospan(\cC)$, in which we take coproducts (in $\cC$) of objects, cospan diagrams (pointwise) and morphisms between cospan diagrams. We denote the monoidal product of $\bCospan(\cC)$ by $-\sqcup-$, even though $x\sqcup y$ is usually \emph{not} a categorical coproduct of $x$ and $y$ in $\Cospan(\cC)$.
We have two canonical symmetric monoidal $\inftwo$-functors $\cC\to\bCospan(\cC)$ and $\cC^\op\to\bCospan(\cC)$.
Dually, for an $\infty$-category $\cC$ with finite limits, we have a $\inftwo$-category $\bSpan(\cC)$, inheriting a symmetric monoidal structure from the product monoidal structure of $\cC$. We denote $\Span(\cC):=\bSpan(\cC)^\twosimeq$.

We will use the following fact in the case $\cC=\cS$: we can extend the $\infty$-functor $\Mod_R^{(-)}:=\Fun(-,\Mod_R)\colon \cS^\op\to\SPrL_R$ 
to an $\inftwo$-functor $\Mod_R^{(-)}\colon \bSpan(\cS)\to\SPrL_R$, sending $X\in\cS$ to $\Mod_R^X$, and sending a span diagram $X\overset{f}{\ot} Z\overset{g}{\to} Y$ to the composite functor $\Mod_R^X\overset{f^*}{\to}\Mod_R^Z\overset{g_!}{\to}\Mod_R^Y$. This follows from \cite[Theorem 3.2.2]{GR}, applied to the setting in which all classes of ``horizontal'', ``vertical'' and ``admissible'' morphisms (in the terminology from loc.cit.) consist of all morphisms in $\cS$; the given theorem is applicable because the Beck--Chevalley transformations from Subsection \ref{subsec:BCprojformulae}, point (2), are equivalences.

Moreover, since $\Mod_R^{(-)}\colon(\cS,\times)^\op\to\SPrL_R$ is a symmetric monoidal $\inftwo$-functor, we obtain a symmetric monoidal structure also on the $\inftwo$-functor $\Mod_R^{(-)}\colon\bSpan(\cS)\to\SPrL_R$.

\subsection{Dualising modules and \texorpdfstring{$R$}{R}-Poincar\'e duality spaces}
\label{subsec:PDspaces}
A space $M\in\cS$ is a \emph{Poincar\'e duality space} if it satisfies a suitable form of Poincar\'e duality, akin to the one enjoyed by closed manifolds. The classical definition of Poincar\'e duality space of dimension $d$ \cite{Browder1,Browder2,Spivak,Wall} requires a coefficient system $O_M\in\Pic(\Z)^M$, with fibres isomorphic to $\Z[-d]$, together with a fundamental class $[M]\in H_0(M;O_M)$, such that for any $\zeta\in\Mod_\Z(\Ab)^M\subset\Mod_\Z^M$, cap product with $[M]$ induces an isomorphism $H_*(M;\zeta)\overset{\cong}{\to} H^{-*}(M;\zeta\otimes O_M)$.
We will use the following alternative approach based on categories of parametrised spectra, due to Klein \cite{Kleindual} (see also \cite[Definition 2]{LuriePoincare}, \cite[Appendix]{LandPD}, \cite[page 93]{nineauthors} and \cite{BHKK} for more modern accounts); we will adapt the approach to the setting of parametrised $R$-modules.

Let $X\in\cS$.
In general $p(X)_*\colon\Mod_R^X\to\Mod_R$ does not preserve colimits, and it only carries a lax $R$-linear structure. However, if $X$ is a compact space (i.e. it is a homotopy retracts of a finite cell complex), then $p(X)_*$ preserves colimits and its lax $R$-linear structure is in fact $R$-linear. By the classification of $R$-linear functors known as ``Morita theory'', in this case, there is a contractible choice of the following data:
\begin{itemize}
 \item a parametrised $R$-module $\omega^R_X\in \Mod_R^X$, called ``dualising module'';
 \item an $R$-linear natural isomorphism $p(X)_*\simeq p(X)_!(-\otimes_R\omega^R_X)$.
\end{itemize}

\begin{defn}
\label{defn:RPoincare}
We say that a compact space $M$ is an \emph{$R$-Poincar\'e duality space of dimension $d\ge0$} if $\omega_M^R$ has fibres isomorphic to $R[-d]$.

A map of spaces $f\colon X\to Y$ is an $R$-Poincar\'e fibration of dimension $d$ if its (homotopy) fibres are $R$-Poincar\'e duality spaces of dimension $d$.
\end{defn}
By means of parametrised homotopy theory, if $f\colon X\to Y$ is an $R$-Poincar\'e fibration of dimension $d$ then there is a canonical (yet not contractible!) choice of the following data, see \cite{BHKK} for details:
\begin{itemize}
 \item a parametrised $R$-module $\omega^R_f\in \Mod_R^X$, called ``dualising module'';
 \item an $R$-linear natural isomorphism $f_*\simeq f_!(-\otimes_R\omega^R_f)$.
\end{itemize}
Note that if $R\to R'$ is a map of $E_\infty$-ring spectra and $M$ is an $R$-Poincar\'e duality space of dimension $d$, then $M$ is also an $R'$-Poincar\'e duality space of dimension $d$, since $\omega_M^{R'}\simeq (p(M)^*(R'))\otimes_R\omega_M^R$.

Closed manifolds are examples of $\bS$-Poincar\'e duality spaces. However, there also exist $\bS$-Poincar\'e duality spaces that are not homotopy equivalent to a closed manifold, see for instance
\cite[Theorem 5.2]{GitlerStasheff}, 
\cite[Sections 7-8]{Spivak}, 
\cite[Theorems 1.5 and 4.3, Corollaries 5.4.1 and 5.4.2]{Wall}.

\begin{defn}
\label{defn:orientation}
An \emph{orientation} on an $R$-Poincar\'e duality space $M$ of dimension $d$ is an isomorphism of parametrised $R$-modules $\omega_M^R\simeq p(M)^*(R[-d])$. Similarly, an orientation on an $R$-Poincar\'e fibration $f\colon X\to Y$ is an isomorphism $\omega_f^R\simeq p(X)^*(R[-d])$
\end{defn}
We remark that an orientation on an $R$-Poincar\'e duality space $M$ yields Poincar\'e duality isomorphism between $R$-homology and $R$-cohomology of $M$, as we have a chain of equivalences of $R$-modules:
\[
\begin{split}
 p(M)_*p(M)^*(R)& \simeq p(M)_!\big(p(M)^*(R)\otimes_R\omega_M^R\big)\\
 &\simeq p(M)_!(p(M)^*(R[-d]))\simeq p(M)_!p(M)^*(R)[-d],
\end{split}
\]
yielding an isomorphism between the respective homotopy groups.

\subsection{Orientation-preserving equivalences}
\label{subsec:hAut+}
Let $M,M'$ be $d$-dimensional $R$-Poin\-car\'e duality spaces. Given a homotopy equivalence $f\colon M\overset{\simeq}{\to} M'$, we obtain an equivalence of categories $f^*\colon \Mod_R^{M'}\overset{\simeq}{\to} \Mod_R^M$, compatible with the functors $p(M)^*$ and $p(M')^*$ from $\Mod_R$; moreover we have an equivalence $f_!\simeq f_*$ obtained by identifying both functors with the inverse of $f^*$. We can then write a composite equivalence 
\[
\begin{split}
p(M)_*&\simeq p(M')_*f_*\simeq p(M')_!(f_!(-)\otimes_R\omega_{M'}^R)\simeq p(M')_!f_!(-\otimes_R f^*(\omega_{M'}^R))\\
&\simeq p(M)_!(-\otimes_Rf^*(\omega_{M'}^R))
\end{split}
\]
and by the contractibility of the choice of a parametrised module $\omega_M^R\in\Mod_R^M$ and together with a natural equivalence $p(M)_*\simeq(-\otimes_R\omega_M^R)$, we obtain in particular an equivalence of parametrised $R$-modules $\omega_M^R\simeq f^*(\omega_{M'}^R)$.

Suppose now that both $M$ and $M'$ are oriented; then we can write an automorphism of $p(M)^*(R[-d])$ as the following composite:
\[
 p(M)^*(R[-d])\simeq\omega_M^R\simeq f^*(\omega_{M'}^R)\simeq f^*p(M')^*(R[-d])\simeq p(M)^*(R[-d]).
\]

\begin{defn}
\label{defn:orientedequivalence}
Let $M,M'$ be oriented $d$-dimensional $R$-Poincar\'e duality spaces. An \emph{orientation-preserving homotopy equivalence} from $M$ to $M'$ is the datum of a homotopy equivalence $f\colon M\to M'$ and a homotopy between the above composite and the identity of $p(M)^*(R[-d])$.
\end{defn}
More formally, let $\Poin_d^R\subset\cS^\simeq$ denote the full $\infty$-subgroupoid spanned by $R$-Poincar\'e duality spaces of dimension $d$; observe that $\Poin_d^R\subset\cS^\omega$, in particular $\Poin_d^R$ is $\fU_1$-small and we may regard it as an object of $\cS$. Denote by $\cS_{/\Pic(R)}$ the overcategory over $\Pic(R)\in\cS$. 

There is a map of $\infty$-groupoids $\omega\colon\Poin_d^R\to(\cS_{/\Pic(R)})^\simeq$, given by endowing an $R$-Poincar\'e duality space $M$, up to contractible choice, with an invertible parametrised $R$-module $\omega_M^R$ and a natural equivalence $p(M)_*\simeq p(M)_!(-\otimes_R\omega_M^R)$, and by then forgetting the second datum, namely the natural equivalence. Informally, the functor $\omega$ sends $M$ to the arrow $\omega_M^R\colon M\to \Pic(R)$.
Combining this with the essentially unique functor $\Poin_d^R\to(\cS_{/*})^\simeq$ sending $M$ to $M\to *$, we obtain a functor of $\infty$-groupoids
\[
\omega'\colon\Poin_d^R\to\Fun(\del\Delta^2,\cS)^\simeq,
\]
sending $M$ to the functor $\omega'(M)\colon\del\Delta^2\to\cS$ which in turn sends $0\mapsto M$, $1\mapsto *$, $2\mapsto\Pic(R)$, $(0\to1)\mapsto(M\to *)$, $(0\to2)\mapsto(\omega_M^R\colon M\to \Pic(R))$, and $(1\to2)\mapsto(R[-d]\colon *\to\Pic(R))$
\begin{defn}
\label{defn:fMPoin}
We define the $\infty$-groupoid of oriented $R$-Poincar\'e duality spaces
$\fM_{\Poin,d}^{R,+}$ as the fibre product in $\cS^{\fU_2}$ of the cospan
\[
\begin{tikzcd}
\Poin_d^R\ar[r,"\omega'"]&\Fun(\partial\Delta^2,\cS)^\simeq&\Fun(\Delta^2,\cS)^\simeq.\ar[l,"\text{restrict}"']
\end{tikzcd}
\]
\end{defn}
An object in $\fM_{\Poin,d}^{R,+}$ is thus a commutative triangle of the form
\[
\begin{tikzcd}
M\ar[r]\ar[dr,"\omega_M^R"']&*\ar[d,"{R[-d]}"]\\
&\Pic(R),
\end{tikzcd}
\]
which is the same as the datum of an oriented $d$-dimensional $R$-Poincar\'e duality space.
A morphism in $\fM_{\Poin,d}^{R,+}$ is precisely an orientation-preserving equivalence as in Definition \ref{defn:orientedequivalence}.

We observe that the functor $\omega'$ takes value in $\Fun(\del\Delta^2,\cS^\kappa)$  for any regular cardinal $\kappa<\fU_1$ such that $\Pic(R)$ is $\kappa$-compact: here $\cS^\kappa\subset\cS$ denotes the $\fU_1$-small full $\infty$-subcategory of $\cS$ spanned by $\kappa$-compact objects. We can thus replace both occurrences of ``$\cS$'' by ``$\cS^\kappa$'' in Definition \ref{defn:fMPoin}, thus showing that $\fM_{\Poin,d}^{R,+}\in\cS^{\fU_2}$ is in fact $\fU_1$-small, i.e. $\fM_{\Poin,d}^{R,+}\in\cS$.
\begin{nota}
\label{nota:hAut+}
For an oriented $R$-Poincar\'e duality space $M$ we let $\hAut^+(M)$ denote the loop space of $\fM_{\Poin,d}^{R,+}$ at $M$, and we let $\bB\hAut^+(M)$ denote the component of $\fM_{\Poin,d}^{R,+}$ containing $M$.
\end{nota}
\begin{rem}
In Definition \ref{defn:RPoincare}, one could require that $M\in\cS$ is just \emph{$\Mod_R$-twisted ambidextrous} in the sense of \cite{Cnossen}, instead of compact. The $\infty$-groupoid $\fM_{\Poin,d}^{R,+}$ will be $\fU_2$-small in this case, but not obviously $\fU_1$-small.
\end{rem}

\subsection{Commutative Frobenius algebras}
\label{subsec:CFrob}
We start by recalling some material from \cite[Subsection 4.6.5]{HA}.
\begin{defn}
\label{defn:duality}
Let $\cC$ be a symmetric monoidal $\infty$-category, and let $x\in\cC$ be an object. A \emph{duality datum} for $x$ is the pair $(x^\vee,c)$ of an object $x^\vee\in\cC$ and a morphism $c\colon x\otimes x^\vee\to \one$ satisfying the following property: there exists a morphism $u\colon 1\to x^\vee\otimes x$ for which there exist homotopies $\Id_x\simeq (c\otimes\Id_x)(\Id_x\otimes u)$ and $\Id_{x^\vee}\simeq(\Id_{x^\vee}\otimes c)(u\otimes\Id_{x^\vee})$. In this case we say that $x^\vee$ is a dual of $x$, and that $x$ is dualisable; the maps $c$ and $u$ are called counit and unit of the duality.
\end{defn}

The following is a symmetric monoidal adaptation of \cite[Definition 4.6.5.1]{HA}.
\begin{defn}
\label{defn:Frobeniusalgebra}
A \emph{commutative Frobenius algebra} in a symmetric monoidal $\infty$-category $\cC$ is the datum $(\fI,\ft)$ of a symmetric monoidal functor $\fI\colon\Fin\to\cC$ and a morphism $\ft\colon\fI(\ul1)\to\fI(\ul0)$ such that $(\fI(\ul1),\ft\circ\fI(\mu))$ is a duality datum for $\fI(\ul1)$, where $\mu\colon\ul2\to\ul1$ denotes the unique map in $\Fin$.
\end{defn}
A recent result of Barkan--Steinebrunner \cite{BarkanSteinebrunner} allows one to use the datum of a commutative Frobenius algebra in a symmetric monoidal $\infty$-category $\cC$ in order to construct a symmetric monoidal functor from a certain $\infty$-category $\Gr$ of \emph{graph cobordisms between finite sets} to $\cC$. 
\begin{defn}
\label{defn:Gr}
Recall Subsection \ref{subsec:cospansspans}. We denote by $\Gr$ the symmetric monoidal $\infty$-subcategory of $\Cospan(\cS)$ whose objects are spaces homotopy equivalent to a finite set, and whose morphisms are cospans $B\to W\ot A$ with $W$ homotopy equivalent to a finite cell complex of dimension $\le1$.
\end{defn}
For $i=1,2$ we have a functor $\CFrob\colon\CAlg(\Catinfty^{\fU_i})\to\cS^{\fU_i}$, associating with a symmetric monoidal $\infty$-category $\cC$ the moduli space of commutative Frobenius algebras in $\cC$;
see \cite{BarkanSteinebrunner} or \cite[Definition 2.15]{Bianchi:graphcobset} for a precise definition. 

\begin{thm}[Barkan--Steinebrunner]
\label{thm:BarkanSteinebrunner}
For $i=1,2$, the functor
\[
\CFrob\colon\CAlg(\Catinfty^{\fU_i})\to\cS^{\fU_i},
\]
associating with a symmetric monoidal $\infty$-category its moduli space of commutative Frobenius algebras, is represented by $\Gr$. In other words, we have a natural equivalence
of functors $\CAlg(\Catinfty^{\fU_i})\to\cS^{\fU_i}$
\[
\Fun^\otimes(\Gr,-)^\simeq\simeq\CFrob.
\]
Similarly, the functor $\Omega\colon\CAlg(\cS^{\fU_i})\to\cS^{\fU_i}$, associating with a symmetric monoidal $\infty$-groupoid its loop space at the identity element, is represented by $|\Gr|$, i.e. we have a natural equivalence $\Fun^\otimes (|\Gr|,-)^\simeq\simeq\Omega$ of functors $\CAlg(\cS^{\fU_i})\to\cS^{\fU_i}$; the latter equivalence is induced by evaluation at the morphism $\ul1\to\ul1\ot\ul0\in\Gr(\ul1,\ul0)$.
\end{thm}
We refer to \cite[Remark 3.2]{Bianchi:graphcobset} for a comparison between the second part of Theorem \ref{thm:BarkanSteinebrunner} and the work of Galatius \cite{Galatius}; Jan Steinebrunner made us aware of this connection.

\begin{rem}
\label{rem:bothhomotopies}
Given a symmetric monoidal $\infty$-functor $\fI\colon\Fin\to\cC$ and a morphism $\ft\colon \fI(\ul1)\to\fI(\ul0)\simeq\one$, to check whether $(\fI,\ft)$ is a commutative Frobenius algebra in $\cC$ , after setting $x:=\fI(\ul1)$ and $c:=\ft\circ\fI(\mu)$, we need to select a morphism $u\colon \one\to x^{\otimes2}$ and prove that there exist both homotopies $\Id_x\simeq(c\otimes\Id_x)(\Id_x\otimes u)$ and $\Id_x\simeq(\Id_x\otimes c)(u\otimes\Id_x)$. However, if we manage to endow $u$ with a $C_2$-equivariant structure with respect to the trivial $C_2$-action on $\one$ and the $C_2$-action on $x^{\otimes2}$ swapping the factors, then the morphisms $(c\otimes\Id_x)(\Id_x\otimes u)$ and $(\Id_x\otimes c)(u\otimes\Id_x)$ automatically lie in the same path component of $\cC(x,x)$, so that it suffices to prove only the existence of the first, or of the second, homotopy.
\end{rem}

An alternative proof of Theorem \ref{thm:BarkanSteinebrunner} is given in \cite{Bianchi:graphcobset}; in the last work we introduced a symmetric monoidal $\inftwo$-category $\bGr$, which refines $\Gr$ in the sense that we have an equivalence of symmetric monoidal $\infty$-categories $|\bGr|_2\simeq \Gr$. 

\section{Graph cobordisms between spaces and local coefficient systems}
\label{sec:loccoeffgraphcobspaces}
In this section we introduce the symmetric monoidal $\infty$-category $\GrCob$, as well as the symmetric monoidal functor $\xi_d^R\colon\GrCob\to\bB\Pic(R)$ providing ``compatible'' local coefficient systems (with values invertible $R$-modules) over the morphism spaces of $\GrCob$. The functor $\xi_d^R$ depends on $d\in\Z$ and is natural in $R\in\CAlg(\Sp)$.

\subsection{Two categories of graph cobordisms between spaces}
\label{subsec:graphcobspaces}
We introduce two symmetric monoidal $\infty$-categories of graph cobordisms between spaces $\Gr_\cS$ and $\GrCob$. The first is a pullback of symmetric monoidal $\infty$-categories, the second is a localisation of the first.
\begin{defn}
\label{defn:GrS}
Recall the $\infty$-category $\Gr$ from Subsection \ref{subsec:CFrob}. We define $\Gr_\cS$ (respectively $\bGr_\cS$) 
as the symmetric monoidal $\infty$-category obtained as pullback of the following diagram in $\CAlg(\Catinfty^{\fU_2})$:
\[
\begin{tikzcd}[row sep=12pt]
\Gr_\cS\ar[r,"\cF"]\ar[d,"\cG"']\ar[dr,phantom,"\lrcorner",very near start]&((\bCospan(\cS))_{/\emptyset}^\rlax)^\twosimeq\ar[d,"d_1"]\\
\Gr\ar[r,hook]& \Cospan(\cS).
\end{tikzcd}
\]
\end{defn}
In order to reinterpret $\Gr_\cS$ in a more familiar way, we consider the symmetric monoidal $\infty$-category $\Cospan(\Fun([1],\cS))$.
\begin{nota}
\label{nota:D0D1}
We denote by $D_0,D_1\colon \Cospan(\Fun([1],\cS))\to\Cospan(\cS)$ the two symmetric monoidal functors induced by $d_0,d_1\colon\Fun([1],\cS)\to\cS$, respectively.
\end{nota}

We note that an object of $\Gr_\cS$ is a cospan $A\to X\ot \emptyset$, where $A$ is a finite set and $X$ is any space: this datum is essentially given by the arrow $A\to X$, i.e. by an arrow of spaces whose source is a finite set; we consider this as a particular kind of object in $\Cospan(\Fun([1],\cS))$. We further note that a morphism in $\Gr_\cS$ from $(B\to Y)$ to $(A\to X)$ is a morphism $G\in\Gr(B,A)$
together with a square in the $\inftwo$-category $\Cospan(\cS)$ as the following diagram:
\[
\begin{tikzcd}[row sep=10pt]
B\ar[r]\ar[d] &G& A\ar[l]\ar[d]\\
Y& & X\\
\emptyset\ar[r,equal]\ar[u]\ar[uurr,Rightarrow]&\emptyset&\emptyset.\ar[l,equal]\ar[u]
\end{tikzcd}
\]
Concretely, the left-bottom composition of the previous diagram, as a morphism in $\Cospan(\cS)$, is given by the cospan of spaces $B\to Y\ot\emptyset$, whereas the top-right composite is given by $B\to W\ot \emptyset$, where we set $W:= X\sqcup_AG$. Using that $\emptyset\in\cS$ is initial, we obtain the validity of the following notation.
\begin{nota}
\label{nota:morGrS}
We usually represent a morphism in $\Gr_\cS$ from $(B\to Y\ot\emptyset)$ to $(A\to X\ot\emptyset)$ as a commutative diagram in $\cS$ as follows:
\[
\begin{tikzcd}[row sep=10pt]
B\ar[d]\ar[r] &G\ar[d]&A\ar[l]\ar[d]\\
Y\ar[r] &W \ar[ur,phantom,very near start,"\urcorner"]&X.\ar[l]
\end{tikzcd}
\]
\end{nota}
The right square in the diagram in Notation \ref{nota:morGrS} has the property of being a pushout; if we interpret the same diagram as a morphism in $\Cospan(\Fun([1],\cS))$,its image along $D_1$ satisfies the additional property of lying in $\Gr\subset\Cospan(\cS)$. We obtain the following, which we state as a lemma for future reference.
\begin{lem}
\label{lem:rpo}
The symmetric monoidal $\infty$-category $\Gr_\cS$ is equivalent to the wide symmetric monoidal $\infty$-subcategory of $D_1^{-1}(\Gr)\subset\Cospan(\Fun([1],\cS))$ spanned by morphisms represented by a diagram as in Notation \ref{nota:morGrS} whose right square is a pushout square.
\end{lem}
Given a morphism in $\Gr_\cS$ as in Notation \ref{nota:morGrS}, we can think of its image along $D_0$, i.e. the cospan $Y\to W\ot X$, as a graph cobordism between the spaces $Y$ and $X$ as in Definition \ref{defn:graphcob}: the space $W$ is obtained from $X$ by attaching the graph $G$, and the space $Y$ is endowed with a map to $W$. 
The main discrepancy between Definition \ref{defn:graphcob} and the above description of morphisms in $\Gr_\cS$ is the presence of the finite sets $B$ and $A$. To get rid of this ``redundant'' data, we introduce the second $\infty$-category of graph cobordisms between spaces.
\begin{defn}
\label{defn:GrCob}
A morphism in $\Gr_\cS$ is \emph{idle} if its image along $D_0$ is an equivalence; concretely, representing such morphism as a diagram as in Notation \ref{nota:morGrS}, we are requiring that both maps of spaces $Y\to W\ot X$ are equivalences.

Idle morphism form a wide symmetric monoidal $\infty$-subcategory of $\Gr_\cS$.
We let $\GrCob$ denote the symmetric monoidal localisation of $\Gr_\cS$ at idle morphisms.
\end{defn}
\begin{rem}
\label{rem:idleshape}
Consider a morphism in $\Gr_\cS$ as in Notation \ref{nota:morGrS}.
It is clear that if $A\to G$ is an equivalence, then also $X\to W$ is an equivalence. Conversely,
if $X\to W$ is an equivalence, it is an (integral) homology equivalence; this implies that $A\to G$ is an integral homology equivalence, and since $G$ is equivalent to a finite cell complex of dimension $\le1$, this implies that $A\to G$ is in fact an equivalence of spaces.
All in all, we may think of an idle morphism as being just of the following form:
\[
\begin{tikzcd}[row sep=10pt]
B\ar[r]\ar[d]&A\ar[r,equal]\ar[d]&A\ar[d]\\
Y\ar[r,"\simeq"]&X\ar[r,equal]&X.
\end{tikzcd}
\]
In particular, taking $X=Y$ and $B=\emptyset$, we obtain that every object $(A\to X)\in\Gr_\cS$ is connected by some idle morphism to an object of the form $(\emptyset\to X)$. This justifies our thinking of objects in $\GrCob$ as being just spaces; in fact we will prove in Corollary \ref{cor:GrCobcore} the equivalence of core groupoids $\cS^\simeq\simeq\GrCob^\simeq$.
\end{rem}
\begin{rem}
\label{rem:D0localises}
By Definition \ref{defn:GrCob}, the symmetric monoidal functor $D_0\colon\Gr_\cS\to\Cospan(\cS)$ factors through a localised symmetric monoidal functor that we shall still denote $D_0\colon\GrCob\to\Cospan(\cS)$. For $Y,X\in\cS$, the induced map on morphism spaces $D_0\colon\fM_\Gr(Y,X)\to\Cospan(\cS)(Y,X)$ is easily seen to be an injection on $\pi_0$, with essential image spanned by those cospans $Y\to Y\ot X$ with the property that $W$ can be obtained from $X$ by attaching finitely many 0-cells and 1-cells. We will see in Proposition \ref{prop:GuirardelLevitt} that the above map of morphism spaces is in fact a subspace inclusion, provided that $X$ is aspherical.
\end{rem}

\begin{defn}
For spaces $X$ and $Y$ we define $\fM_\Gr(Y,X)$ as the mapping space
\[
\fM_\Gr(Y,X):=\GrCob(\emptyset\to Y,\emptyset\to X).
\]
\end{defn}
In Section \ref{sec:GrCobspaces} we will give a more concrete description of the homotopy type of $\fM_\Gr(Y,X)$ for $X,Y\in\cS$.

\subsection{Coefficient systems in the Picard space}
\label{subsec:picard}
In this subsection we define, for an $E_\infty$-ring spectrum $R$ and an integer $d\in\Z$, a symmetric monoidal $\infty$-functor $\xi_d^R\colon\Gr\to \bB\Pic(R)$. Since $\bB\Pic(R)$ is a symmetric monoidal \emph{space}, by Theorem \ref{thm:BarkanSteinebrunner} we have equivalences of spaces 
\[
\Fun^\otimes(\Gr,\bB\Pic(R))^\simeq\simeq\Fun^\otimes(|\Gr|,\bB\Pic(R))^\simeq\simeq\Omega\bB\Pic(R)\simeq\Pic(R).
\]
\begin{defn}
\label{defn:xidR}
For $d\in\Z$ we let $\xi_d^R\colon\Gr\to\bB\Pic(R)$ denote the symmetric monoidal $\infty$-functor corresponding to $R[d]\in\Pic(R)$ along the above equivalence.

By abuse of notation we also denote by $\xi_d^R$ both the composite symmetric monoidal functor $\xi_d^R\circ\cG\colon\Gr_\cS\to\bB\Pic(R)$,and the induced symmetric monoidal functor on localisation $\xi_d^R\colon\GrCob\to\bB\Pic(R)$; here we use that \emph{all morphisms} in $\Gr_\cS$, and in particular idle morphisms, are sent to invertible morphisms in $\bB\Pic(R)$.
\end{defn}

\begin{rem}
\label{rem:xiringfunctorial}
The construction of $\xi_d^R$ is evidently functorial in $R\in\CAlg(\Sp)$; moreover, since $R[d]\simeq R[1]^{\otimes d}$, we have that $\xi_d^R$ agrees with the composition of $\xi_d^1$ and the symmetric monoidal self map of $\bB\Pic(R)$ raising morphisms to their $d$\textsuperscript{th} tensor power.
\end{rem}
\begin{rem}
\label{rem:compatibleloccoef}
For $X,Y\in\cS$, the functor $\xi_d^R$ sends $X,Y\mapsto*\in\bB\Pic(R)$ and hence it induces a map of morphism spaces $\fM_\Gr(X,Y)\to\Pic(R)$. We interpret the latter as a local coefficient system over the space $\fM_\Gr(X,Y)$, valued in invertible $R$-modules. These local coefficient systems are ``compatible'' with each other in the sense that they assemble into the symmetric monoidal functor $\xi_d^R$. The homology groups $H_*(\fM_\Gr(Y,X);\xi_d^R)$ will parametrise the higher string operations from Corollary \ref{cor:A1}.
\end{rem}

If $R$ is a discrete commutative ring, then there is a remarkably explicit description of the functor $\xi_d^R$. Relying on Remark \ref{rem:xiringfunctorial}, we will focus on the case $R=\Z$; the cases of other discrete rings can be obtained by base-change. The following property of $\Z$ will make the discussion simpler: the natural functor $\Pic(\Mod_\Z(\Ab^\Z))\to\Pic(\Z)$ is an equivalence, where the source is the plain groupoid of invertible graded $\Z$-modules in abelian groups, and the target is the $\infty$-groupoid of invertible $\Z$-modules in spectra. Very explicitly, we have an equivalence of $\infty$-groupoids $\Pic(\Z)\simeq\coprod_{d\in\Z}\bB C_2$: the object $\Z[d]$ represents the $d$\textsuperscript{th} component, and has automorphism group $C_2$, generated by $-\Id_{\Z[d]}$. In particular $\bB\Pic(\Z)$ is a 2-category, all of whose 1-morphisms and 2-morphisms are invertible.

Given $G\in\Gr(B,A)$, we can compute the relative homology groups $H_i(G,A)=H_i(G,A;\Z)$; these groups are finitely generated, free $\Z$-modules, and they vanish for $i\neq0,1$. 
We have functors of $\infty$-groupoids
\[
H_i(-,A)\colon\Gr(B,A)\to(\Mod_\Z(\Ab)^{\mathrm{f.g.free}})^\simeq,\quad i=0,1,
\]
where the target is the core groupoid of the category of finitely generated, free $\Z$-modules in abelian groups. Putting $\Z$-modules in degree 1 gives a functor $(-)[1]\colon(\Mod_\Z(\Ab)^{\mathrm{f.g.free}})^\simeq\to(\Mod_\Z(\Ab^\Z)^{\mathrm{f.g.free}})^\simeq$, and by composition we obtain two functors $H_i(-,A)[1]$, for $i=0,1$.
We may now consider the determinant functor
\[
\det\colon (\Mod_\Z(\Ab^\Z)^{\mathrm{f.g.free}})^\simeq\to\Pic(\Mod_\Z(\Ab^\Z))\simeq\Pic(\Z),
\]
sending a finitely generated, $\Z$-graded free $\Z$-module to its maximal non-zero exterior power. By composition we obtain two functors $\det(H_i(-,A)[1])$.

We shall consider the functor $\hat\xi_d^\Z\colon\Gr(B,A)\to\Pic(\Z)$ given by
\[
\hat\xi_d^\Z:=\det(H_0(-,A)[1])^{\otimes d}\otimes \det(H_1(-,A)[1])^{\otimes -d}.
\]
The following observations are similar to \cite[Lemma 13]{Godin}:
\begin{itemize}
\item if $G\in\Gr(B,A)$ and $G'\in\Gr(C,B)$, then we have a long exact sequence
\[
\begin{tikzcd}
0\to H_1(G,A)\ar[r]&H_1(G\circ G',A)\ar[r] &H_1(G',B)\ar[dll]\\
H_0(G,A)\ar[r]&H_0(G\circ G',A)\ar[r] &H_0(G',B)\to0
\end{tikzcd}
\]
giving rise to an isomorphism of tensor products
\[
\begin{split}
\det(H_1(G,A)[1])\otimes \det(H_1(G',B)[1])\otimes \det(H_0(G\circ G',A)[1])\cong\\
\det(H_1(G\circ G',A)[1])\otimes \det(H_0(G,A)[1])\otimes \det(H_0(G',B)[1]),
\end{split}
\]
which, after taking $d$\textsuperscript{th} tensor powers and reassembling the terms, can be rewritten as $\hat\xi_d^\Z(G\circ G')\simeq\hat\xi_d^\Z(G')\otimes\hat\xi_d^\Z(G)$;
\item we similarly have natural isomorphisms $\hat\xi_d^\Z(G\sqcup G')\cong\hat\xi_d^\Z(G)\otimes\hat\xi_d^\Z(G')$ for $G\in\Gr(B,A)$ and $G'\in\Gr(B',A')$.
\end{itemize}
Using that the morphism spaces of $\bB\Pic(\Z)$ are aspherical, one can easily check that the above assignments give therefore rise to a symmetric monoidal functor of symmetric monoidal $\infty$-categories $\hat\xi_d^\Z\colon\Gr\to \bB\Pic(\Z)$. For a generic discrete ring $R$ we can then postcompose $\hat\xi_d^\Z$ with the functor $\bB\Pic(\Z)\to \bB\Pic(R)$ induced by the ring map $\Z\to R$, thus obtaining a symmetric monoidal $\infty$-functor $\hat\xi_d^R\colon\Gr\to \bB\Pic(R)$.
\begin{prop}
\label{prop:xihatxi}
For a discrete ring $R$ and any $d\ge0$, there is an equivalence of symmetric monoidal $\infty$-functors
\[
\hat\xi_d^R,\xi_d^R\colon\Gr\to \bB\Pic(\Mod_R).
\]
\end{prop}
\begin{proof}
By Theorem \ref{thm:BarkanSteinebrunner}, it suffices to provide an equivalence in $\Pic(\Mod_R)\simeq\Omega\bB\Pic(\Mod_R)$ between the evaluations of $\hat\xi_d^R$ and $\xi_d^R$ at the morphism $\ul1\to\ul1\ot\ul0\in\Gr(\ul1,\ul0)$: the first evaluation is by computation $\det(H_0(\ul1;R)[1])^{\otimes d}\simeq R[d]$, whereas the second evaluation is by definition $R[d]$.
\end{proof}
\begin{ex}
\label{ex:xifromliterature}
Let $n\ge2$, let $\bigvee_nS^1$ denote a wedge of $n$ circles, and consider the subspace $B\hAut(\bigvee_nS^1)\simeq B\Out(F_n)\subset\Gr(\emptyset,\emptyset)$, where $\Out(F_n)$ denotes the group of outer automorphisms of a free group of rank $n$. For a discrete ring $R$, Proposition \ref{prop:xihatxi} allows us to identify the restricted coefficient system $\xi_d^R$ on $B\Out(F_n)$ with the pullback on $\Out(F_n)$ of the sign $C_2$-representation on $R[d\cdot (1-n)]$ along the composite group homomorphism
\[
\Out(F_n)\to \Aut(F_n^{\mathrm{ab}})\cong GL_n(\Z)\overset{\det}{\to}GL_1(\Z)\cong C_2\overset{(-)^d}{\to}C_2.
\]
These representations (which are trivial for $d$ even) have played a prominent role in the literature: most notably, the homology of $\Out(F_n)$ with coefficients in $\xi_d^\Q$ is isomorphic to the homology of Kontsevich' graph complexes associated with the Lie operad \cite{Kontsevich1993,Kontsevich1994,ConantVogtmann,LazarevVoronov}.
\end{ex}

\section{Parametrised \texorpdfstring{$R$}{R}-modules over a Poincar\'e duality space}
\label{sec:ModRPoincare}
We fix $R\in\CAlg(\Sp)$ and an oriented $R$-Poincar\'e duality space $M$ throughout the section. Our goal in this section is to construct symmetric monoidal functors with values in $\infty$-categories that are closely related to $(\SPrL_R)^\twosimeq$, see Subsection \ref{subsec:SPrLR}. The main input will be the symmetric monoidal $\infty$-category $\Mod_R^M\in\CAlg(\SPrL_R)$.

\subsection{The functor \texorpdfstring{$\Psi_R(M)$}{PsiR(M)}}
\label{subsec:PsiRMdef}
In this subsection we construct a commutative Frobenius algebra in the symmetric monoidal $\infty$-category $(\SPrL_R)^\twosimeq$; we will denote by $\Psi_R(M)\colon\Gr\to\SPrL_R$ the corresponding symmetric monoidal functor.
\begin{nota}
\label{nota:delta}
We introduce notation for some maps between powers of $M$, induced along $M^{(-)}$ by maps of finite sets:
\begin{itemize}
\item We denote by $1$ the identity of $M$.
\item We denote by $\tau\colon M^{\ul2}\to M^{\ul2}$ the automorphism swapping coordinates.
\item We denote by $p=p(M)\colon M\to*$ the map obtained from $\iota\colon\ul0\to\ul1$.
\item We denote by $\Delta\colon M\to M^{\ul2}$ the diagonal map, induced by $\mu\colon\ul2\to\ul1$; \item We denote by $\bar\Delta\colon M\to M^3$ the map obtained from the unique map $\ul3\to\ul1$.
\item We denote by $\delta\colon M^{\ul2}\to M^{\ul4}=M^{\ul2\sqcup\ul2}$ the map obtained from the fold map $\ul2\sqcup\ul2\to\ul2$, i.e. the map $\ul4\to\ul2$ sending $1,3\mapsto1$ and $2,4\mapsto 2$.
\item We denote by $\partial\colon M^{\ul2}\to M^{\ul3}$ the map obtained from the map $\ul3\to\ul2$ sending $1,3\mapsto2$ and $2\mapsto1$.
\end{itemize}
\begin{nota}
To simplify notation, we remove all symbols ``$\times$'' from now on. For instance, we write ``$(p\,1)$'' for the map $(p\times 1)\colon M^{\ul2}\to M$. We also omit the symbol ``$\circ$'' denoting composition of maps and functors.
\end{nota}
Note that $\partial=(\tau\times 1)\circ(1\times\Delta)$, which we write as $(\tau\, 1)(1\,\Delta)$, and similarly $\delta=(1\,\partial)(\Delta\, 1)$.
\end{nota}
\begin{nota}
\label{nota:s_equals_Rd}
We sometimes abbreviate $R[d]=s$ and $R[-d]=\bar s$.
\end{nota}
\begin{nota}
Given a functor $F\colon \cC\to\Mod_R^X$ with values in a category of parametrised $R$-modules, and given $\zeta\in\Mod_R$, we denote by $F\otimes \zeta$ the functor obtained by pointwise tensor product of $F$ and the constant functor $\cC\to\Mod_R^X$ with value $p(X)^*(\zeta)$. 

If $\cC=\Mod_R^Y$ is also a category of parametrised $R$-modules, and if $F$ is of the form $f^*$, $f_!$ or $f_*$ for some map $f$ between $X$ and $Y$, we will identify $F\otimes\zeta$ with $F\circ(\Id_{\Mod_R^Y}\otimes\zeta)$ by the projection formula without further mention (in the case of $f_*$, we are assuming that $f$ has compact fibres).
\end{nota}

\begin{nota}
 Whenever we want to stress an action of $C_2$ on a finite set, we write the finite set as a coproduct of finite sets of the form $\ul1$ and $\ul2$. For instance, the  $C_2$-action on $\ul3=\ul2\sqcup\ul1$ swaps 1 and 2, and fixes 3.
\end{nota}
We have the following commutative diagrams of $C_2$-equivariant finite sets and spaces, respectively; the second is obtained from the first by applying $M^{(-)}$:
\[
\begin{tikzcd}[row sep=20pt,column sep=40pt]
\ul1\ar[dr,"\lrcorner",phantom, very near start]&\ul2\ar[l]\ar[d,bend left=15, "\Id_{\ul2}\sqcup\iota\sqcup\iota"]&\ul0\ar[l]
\\
\ul2\sqcup\ul1\ar[u]&\ul2\sqcup\ul2;\ar[l,"\Id_{\ul2}\sqcup\mu"']\ar[u,"\text{fold}"]
\end{tikzcd}
\hspace{1cm}
 \begin{tikzcd}[row sep=20pt,column sep=40pt]
 M\ar[r,"\Delta"']\ar[dr,"\lrcorner",phantom, very near start]\ar[d,"\bar\Delta"] &M^{\ul 2}\ar[d,"\delta"']\ar[r,"p^{\ul2}"']&*\\
 M^{\ul 2\sqcup\ul1},\ar[r,"(1^{\ul2}\,\Delta)"]
 &M^{\ul 2\sqcup\ul 2}.\ar[u,bend right=15,"(1^{\ul2}\,p^{\ul2})"']
 \end{tikzcd}
\]
We remark that the right square, involving $\bar\Delta$ and $\delta$, is cartesian, and that the composite $(1^{\ul2}\, p^{\ul2})\delta$ is homotopic to $1^{\ul2}$.
\begin{defn}
\label{defn:PsiRM}
We denote by $\Psi_R(M)\colon\Gr\to(\SPrL_R)^\twosimeq$ the symmetric monoidal $\inftwo$-functor corresponding, by virtue of Corollary \ref{cor:A1}, to the commutative Frobenius algebra $(\fI_\Psi,\ft_\Psi)$ given as follows:
\begin{itemize}
 \item We let $\fI_\Psi\colon\Fin\to(\SPrL_R)^\twosimeq$ be the symmetric monoidal functor classifying $\Mod_R^M\in\CAlg(\SPrL_R)$, endowed with the fibrewise tensor product over $R$. 
\item We let $\ft_\Psi\colon \fI_\Psi(\ul1)\to \fI_\Psi(\ul0)$ be the functor $p_*\colon\Mod_R^M\to\Mod_R$.
\end{itemize}
\end{defn}
The rest of the subsection is devoted to the proof that the data provided in Definition \ref{defn:PsiRM} indeed yield a commutative Frobenius algebra. For this we shall make use of the
$C_2$-equivariant functor $\fe_\Psi\colon \fI_\Psi(\ul0)\to \fI_\Psi(\ul2)$ given by the formula
\[
(1^{\ul2}\, p^{\ul2})_!(1^{\ul2}\,\Delta)_*\bar\Delta_! p^*\colon\Mod_R\to\Mod_R^{M^{\ul 2}}\simeq\Mod_R^M\otimes_{\Mod_R}\Mod_R^M.
\]

\begin{lem}
\label{lem:feisleftadjoint}
The functors $\ft_\Psi$ and $\fe_\Psi$ belong to $\SPrL_R$, i.e. they preserve colimits and their lax $R$-linear structure is in fact $R$-linear.
\end{lem}
\begin{proof}
Use Notation \ref{nota:s_equals_Rd}.
A priori a functor of the form $f_*$ is only a right adjoint functor and lax $R$-linear, so that neither of $\ft_\Psi$ and $\fe_\Psi$ is obviously in $\SPrL_R$. For
$\ft_\Psi=p_*$, we use that $M$ is an oriented $R$-Poincar\'e duality space and identify 
$p_*\simeq p_!(-\otimes p^*(\bar s))\simeq p_!\otimes \bar s$, where the last two expressions, related by the projection formula, are both evidently compositions of left adjoint functors.\footnote{Note that if $M$ is any compact space, we can still identify $p_*\simeq p_!(-\otimes\omega_M^R)$.}

For $\fe_\Psi=(1^{\ul2}\, p^{\ul2})_!(1^{\ul2}\,\Delta)_*\bar\Delta_! p^*$, it suffices to check that $(1^{\ul2}\, p^{\ul2})_!(1^{\ul2}\,\Delta)_*$ is in $\SPrL_R$. For this we use that, since $M$ is an oriented $R$-Poincar\'e duality space, the $C_2$-space $M^{\ul2}$ is a $C_2$-equivariantly oriented $C_2$-equivariant $R$-Poincar\'e duality space.

More precisely, we can identify $\omega_{M^{\ul2}}^R$ with the tensor product 
$(p\, 1)^*(\omega_M^R)\otimes (1\, p)^*(\omega_M^R)$ in $\Mod_R^{M^{\ul2}}$;
and we can further identify the latter with $(p^{\ul2})^*(\bar s^{\otimes 2})$, using twice the $R$-orientation on $M$. We get a $C_2$-equivariant equivalence 
\[
(p^{\ul2})_*\simeq (p^{\ul2})_!(-\otimes (p^{\ul2})^*(\bar s^{\ul2}))\simeq (p^{\ul2})_!\otimes\bar s^{\otimes\ul2},
\]
where $C_2$ acts on the functors $(p^{\ul2})_!,(p^{\ul2})_*,(p^{\ul2})^*$ by pre/postcomposition with $\tau_!\simeq\tau_*\simeq\tau^*$, and it acts on $\bar s^{\otimes 2}$ by swapping the factors.

Similarly, $(1^{\ul2}\, p)\colon M^{\ul3}\to M^{\ul2}$ is an oriented $R$-Poincar\'e fibration with $\omega^R_{(1^{\ul2}\, p)}\simeq (p^{\ul2}\, 1)^*(\omega_M)\simeq (p^{\ul3})^*(\bar s)$, and $(1^{\ul2}\, p^{\ul2})\colon M^{\ul4}\to M^{\ul2}$ is a $C_2$-equivariantly oriented $C_2$-equivariant $R$-Poincar\'e fibration with
$\omega^R_{(1^{\ul2}\, p^{\ul2})}\simeq (p^{\ul2}\,1^{\ul2})^*(\omega_{M^{\ul2}})\simeq (p^{\ul4})^*(\bar s^{\otimes\ul2})$.
We can now write a $C_2$-equivariant chain of $R$-linear equivalences of lax $R$-linear functors
\[
\begin{split}
(1^{\ul2}\, p^{\ul2})_!(1^{\ul2}\,\Delta)_*(-)&\simeq (1^{\ul2}\, p^{\ul2})_*\big((1^{\ul2}\,\Delta)_*(-)\otimes s^{\otimes\ul2}\big)
\simeq (1^{\ul2}\, p^{\ul2})_*(1^{\ul2}\,\Delta)_*(-\otimes s^{\otimes\ul2})\\
&\simeq  (1^{\ul2}\, p)_*(-\otimes s^{\otimes \ul2})
\simeq (1^{\ul2}\, p)_!(-\otimes s^{\otimes \ul2}\otimes \bar s)\\
&\simeq (1^{\ul2}\, p)_!(-)\otimes s^{\otimes \ul2}\otimes \bar s,
\end{split}
\]
and the last two expressions are evidently compositions of functors in $\SPrL_R$.\footnote{Note that if $M$ is any $R$-Poincar\'e duality space, even in the absence of an $R$-orientation we can still identify $(1^{\ul2}\, p^{\ul2})_!(1^{\ul2}\,\Delta)_*(-)$ with the analog of the second to last expression in the chain of equivalences, reading $(1^{\ul2}p)_!(-\otimes (p^{\ul2}1)^*((\omega_M^{-1})^{\otimes2}\otimes\omega_M)$.}
\end{proof}
By virtue of Remark \ref{rem:bothhomotopies}, the final
task is to provide a natural equivalence
\[
\beta_\Psi\colon ((\ft_\Psi\circ \fI_\Psi(\mu))\otimes\Id_{\fI_\Psi(\ul 1)})\circ(\Id_{\fI_\Psi(\ul 1)}\otimes \fe_\Psi)\simeq \Id_{\fI_\Psi(\ul 1)}
\]
The target of the natural equivalence can be identified with $\Id_{\Mod_R^M}$; to understand the source, we recall that $\Delta\colon M\to M^{\ul2}$ is the map induced by $\mu\colon\ul2\to\ul1$ under the contravariant functor $M^{(-)}$; applying further the functor $\Mod_R^{(-)}\colon\cS^\op\to\SPrL_R$ we obtain the identification $\fI_\Psi(\mu)\simeq\Delta^*$.
Hence, we need to define a natural equivalence $\beta_\Psi$ of endofunctors of $\Mod_R^M$ as follows:
\[
 \beta_\Psi\colon(p\, 1)_*(\Delta\, 1)^*(1^{\ul3}\, p^{\ul2})_!(1^{\ul3}\,\Delta)_*(1\, \bar\Delta)_!(1\, p)^* \xrightarrow{\simeq} \Id_{\Mod_R^M}.
\]
\begin{defn}
\label{defn:betaPsi}
We define $\beta_\Psi$ as the concatenation of the following two natural equivalences $\beta_{\Psi,1}$, on left, and $\beta_{\Psi,2}$, on right: 

\[
    \begin{tikzcd}
      (p\, 1)_*(\Delta\, 1)^*(1^{\ul3}\, p^{\ul2})_!(1^{\ul3}\,\Delta)_*(1\, \bar\Delta)_!(1\, p)^*\ar[d,"\simeq"',"BC/\text{s.m.}"]\\
  (p\, 1)_*(1^{\ul2}\, p^{\ul2})_!(1^{\ul2}\,\Delta)_*(\Delta\, 1^{\ul2})^*(1\, \bar\Delta)_!(1\, p)^* 
  \ar[d,"\simeq"',"BC"]\\
   (p\, 1)_*(1^{\ul2}\, p^{\ul2})_!(1^{\ul2}\,\Delta)_*\bar\Delta_!\Delta^*(1\, p)^*
  \ar[d,"\simeq"',"(1\, p)\Delta\simeq 1"]\\
    (p\, 1)_*(1^{\ul2}\, p^{\ul2})_!(1^{\ul2}\,\Delta)_*\bar\Delta_!;
 \end{tikzcd}
    \hspace{1cm}
     \begin{tikzcd}
      (p\, 1)_*(1^{\ul2}\, p^{\ul2})_!(1^{\ul2}\,\Delta)_*\bar\Delta_!\\
  (p\, 1)_*(1^{\ul2}\, p^{\ul2})_!\delta_!\Delta_*
  \ar[u,"BC"',"\simeq?"]\ar[d,"(1^{\ul2}p^{\ul2})\delta\simeq1^{\ul2}","\simeq"']\\
   (p\, 1)_*\Delta_*
  \ar[d,"\simeq"',"(p1)\Delta\simeq 1"] \\
  \Id_{\Mod_R^M}.
  \end{tikzcd}
\]
The first arrow of $\beta_{\Psi,1}$ is given by symmetric monoidality of $(\SPrL_R)^\twosimeq$, by which both functors $(1^{\ul2}\, p^{\ul2})_!(1^{\ul2}\,\Delta)_*(\Delta\, 1^{\ul2})^*$ and $(\Delta\, 1)^*(1^{\ul3}\, p^{\ul2})_!(1^{\ul3}\,\Delta)_*$ can be identified with the tensor product of three functors $
\Delta^*\otimes 1\otimes((p^{\ul2})_!\Delta_*)$.
The same arrow can be equivalently considered as a composition of two Beck--Chevalley equivalences, whence we have labeled it by ``$BC$/s.m.'', where ``s.m.'' stands for ``symmetric monoidality''.
The second arrow of $\beta_{\Psi,1}$
is the Beck--Chevalley equivalence given by the following pullback square, which is the image along $M^{(-)}$ of a pushout square of finite sets:

\[
\begin{tikzcd}[row sep=15pt]
 M\ar[d,"\Delta"]\ar[dr,"\lrcorner",phantom, very near start]\ar[r,"\bar\Delta"] &M^{\ul3}\ar[d,"\Delta\,1"]\\
 M^{\ul2}\ar[r,"1\,\bar\Delta"] &M^{\ul4}.
\end{tikzcd}
\]
\end{defn}
We observe that in the definition of $\beta_{\Psi,2}$, one of the occurring arrows is a Beck--Chevalley transformation going in the wrong direction; the following lemma shows that this 2-morphism, which a priori only belongs to $\SPrL_R$, is in fact an equivalence, so that the entire $\beta_\Psi$, in the end, is an invertible 2-morphism.
\begin{lem}
\label{lem:BCinvertible}
The following Beck--Chevalley transformation is invertible:
\[
\begin{tikzcd}
(p\, 1)_*(1^{\ul2}\, p^{\ul2})_!\delta_!\Delta_*\ar[r,"BC"]& (p\, 1)_*(1^{\ul2}\, p^{\ul2})_!(1^{\ul2}\,\Delta)_*\bar\Delta_!.
\end{tikzcd}
\]
\end{lem}
\begin{proof}
Use Notation \ref{nota:s_equals_Rd}.
The given Beck--Chevalley transformation is the left vertical map in the following commutative square, whose horizontal maps are obtained by combining the Poincar\'e duality equivalences for $(1^{\ul2}\,p^{\ul2})$ and $(p\,1)$:
\[
\begin{tikzcd}[column sep=90pt]
(p\,1)_*(1^{\ul2}\,p^{\ul2})_!\delta_!\Delta_*\ar[d,"BC"]\ar[r,"PD\text{ on }(1^{\ul2}\,p^{\ul2})\text{ and }(p\,1)","\simeq"']& (p\,1)_!(1^{\ul2}\,p^{\ul2})_*\delta_!\Delta_*\otimes s^{\otimes \ul2}\otimes \bar s\ar[d,"BC"]\\
(p\,1)_*(1^{\ul2}\,p^{\ul2})_!(1^{\ul2}\,\Delta)_*\bar\Delta_!\ar[r,"PD\text{ on }(1^{\ul2}\,p^{\ul2})\text{ and }(p\,1)","\simeq"']& (p\,1)_!(1^{\ul2}\,p^{\ul2})_*(1^{\ul2}\,\Delta)_*\bar\Delta_!\otimes s^{\otimes \ul2}\otimes \bar s
\end{tikzcd}
\]
It then suffices to prove that the right vertical map is invertible. After removing the factors ``$s^{\otimes \ul2}\otimes \bar s$'', the right vertical map fits as left vertical map in the following commutative diagram:
\[
\begin{tikzcd}
(p\,1)_!(1^{\ul2}\,p^{\ul2})_*\delta_!\Delta_*\ar[d,"BC"]\ar[r,"BC/\text{s.m.}","\simeq"']&(1\,p^{\ul2})_*(p\,1^{\ul3})_!\delta_!\Delta_*\ar[d,"BC"]\ar[r,"\simeq"'] &(1\,p)_*(1\,p\,1)_*\partial_!\Delta_*\ar[d,"BC"]\\
(p\,1)_!(1^{\ul2}\,p^{\ul2})_*(1^{\ul2}\,\Delta)_*\bar\Delta_!\ar[r,"BC/\text{s.m.}","\simeq"']& (1\,p^{\ul2})_*(1\,\Delta)_*(p\,1^{\ul2})_!\bar\Delta_!\ar[r,"\simeq"']&(1\,p)_*(1\,p\,1)_*(1\,\Delta)_*\Delta_!
\end{tikzcd}
\]
It suffices again to prove that the right vertical map is an equivalence. This is the Beck--Chevalley transformation corresponding to the left cartesian square in the following diagram; we can combine it with the Beck--Chevalley transformations corresponding to the right and the total cartesian squares, obtaining the commutative triangle on the right:
\[
\begin{tikzcd}
M\ar[r,"\Delta"]\ar[d,"\Delta"]\ar[dr,"\lrcorner",phantom, very near start] & M^{\ul2}\ar[d,"\partial"]\ar[r,"(p\,1)"]\ar[dr,"\lrcorner",phantom, very near start] & M\ar[d,"\Delta"]\\
M^{\ul2}\ar[r,"(1\,\Delta)"] &M^{\ul3}\ar[r,"(1\,p\,1)"] &M^{\ul2}
\end{tikzcd}
\hspace{1cm}
\begin{tikzcd}
(1\,p)_*(1\,p\,1)_*\partial_!\Delta_*\ar[d,"BC"]&(1\,p)_*\Delta_!(p\,1)_*\Delta_*\ar[dl,"BC"]\ar[l,"BC"]\\
(1\,p)_*(1\,p\,1)_*(1\,\Delta)_*\Delta_!
\end{tikzcd}
\]
We conclude by observing that the diagonal and horizontal arrows in the above triangle are equivalences. For the diagonal arrow, we observe that the total cartesian square has horizontal arrows being equivalences, hence also the corresponding Beck--Chevalley transformation is an equivalence. For the horizontal arrow we appeal to symmetric monoidality: we can identify $(1\,p\,1)_*\partial_!\simeq (p\,1^{\ul2})_*(\tau\,1)_*(\tau\,1)_!(1\,\Delta)_!\simeq (p\,1^{\ul2})_*(1\,\Delta)_!$, where we use that $(\tau\,1)_*$ and $(\tau\,1)_!$ are equivalent to each other (the adjunctions with $(\tau\,1)^*$ exhibit each of them as an inverse of $(\tau\,1)^*$) and that $\tau$ is an involution. Up to this replacement, the above Beck--Chevalley equivalence is just the identification of both $(p\,1^{\ul2})_*(1\,\Delta)_!$ and $\Delta_!(p\,1)_*$ with $p_*\otimes\Delta_!$.
\end{proof}

\subsection{A unital enhancement of \texorpdfstring{$\Psi_R(M)$}{PsiR(M)}}
\label{subsec:unitalenhancement}
In this subsection we enhance $\Psi_R(M)$ to a symmetric monoidal functor $\hPsi_R(M)\colon\Gr\to((\SPrL_R)_{\Mod_R/}^\llax)^\twosimeq$; see Notation \ref{nota:laxcomma} for the lax comma category. Recall in particular that an object in $((\SPrL_R)_{\Mod_R/}^\llax)^\twosimeq$ is an $R$-linear functor between $R$-linear categories $F\colon\Mod_R\to\cC$, and a 1-morphism between $F\colon\Mod_R\to\cC$ and $F'\colon\Mod_R\to\cC'$ is the datum of a functor $G$ and a natural transformation fitting in a square as follows:
\[
 \begin{tikzcd}[row sep=10pt]
  \Mod_R\ar[d,"F"']\ar[r,equal]&\Mod_R\ar[d,"F'"]\ar[dl,Rightarrow]\\
  \cC\ar[r,"G"'] &\cC'.
 \end{tikzcd}
\]
The symmetric monoidal structure on $((\SPrL_R)_{\Mod_R/}^\llax)^\twosimeq\subseteq(\Fun([1],\SPrL_R)^\llax)^\twosimeq$ is given by pointwise tensor product in $\SPrL_R$.
We have a symmetric monoidal target functor $d_0\colon((\SPrL_R)_{\Mod_R/}^\llax)^\twosimeq\to(\SPrL_R)^\twosimeq$, and $\hPsi_R(M)$ will be a ``unital enhancement'' of $\Psi_R(M)$ in that we will have a symmetric monoidal equivalence $d_0\circ \hPsi_R(M)\simeq \Psi_R(M)$.

\begin{defn}
\label{defn:hPsiRM}
We define $\hPsi_R(M)$ as the symmetric monoidal $\inftwo$-functor corresponding, by virtue of Corollary \ref{cor:A1}, to the commutative Frobenius algebra $(\fI_{\hPsi},\ft_{\hPsi})$ given by the following data:

\begin{itemize}
 \item We let $\fI_\hPsi\colon\Fin\to((\SPrL_R)_{\Mod_R/}^\llax)^\twosimeq$ be the symmetric monoidal functor classifying the algebra $(p^*\colon\Mod_R\to\Mod_R^M)\in\CAlg((\SPrL_R)_{\Mod_R/}^\llax)$. Here we use that $(p^*\colon\Mod_R\to\Mod_R^M)$ is a commutative algebra object in $((\SPrL_R)_{\Mod_R/})^\twosimeq$, and that the canonical functor $  ((\SPrL_R)_{\Mod_R/})^\twosimeq\to((\SPrL_R)_{\Mod_R/}^\llax)^\twosimeq$ is symmetric monoidal.
\item We let $\ft_\hPsi$ be the following rectangle
\[
 \begin{tikzcd}[row sep=12pt]
  \Mod_R\ar[r,equal]\ar[d,"p^*"']&\Mod_R\ar[d,equal]\ar[dl,Rightarrow,"\eta_{p_*}"]\\
  \Mod_R^M\ar[r,"p_*"']&\Mod_R.
 \end{tikzcd}
\]
\end{itemize}
\end{defn}
The rest of the subsection is devoted to the proof that the data provided in Definition \ref{defn:hPsiRM} indeed yield a commutative Frobenius algebra. Note that, if this is the case, the application of $d_0$ obviously yields the data $(\fI_{\Psi},\ft_{\Psi})$ of the commutative Frobenius algebra from Definition \ref{defn:PsiRM}.

We shall need the following $C_2$-equivariant 1-morphism $\fe_\hPsi\colon\fI_{\hPsi}(\ul0)\to\fI_{\hPsi}(\ul2)$, given by the horizontal composition of the following $C_2$-equivariant diagram of categories, functors and natural transformations:
\[
\begin{tikzcd}[column sep=80pt,row sep=30pt]
\Mod_R\ar[d,equal]\ar[r,equal]&\Mod_R\ar[d,"\delta_!\Delta_* 
\Delta^*(p^{\ul2})^*",very near 
start]\ar[dl,Rightarrow,"BC","p^{\ul2}\Delta\simeq 
p"']\ar[r,equal]&\Mod_R\ar[d,"(p^{\ul2})^*" , near 
start]\ar[dl,Rightarrow,"1^{\ul2}\simeq(1^{\ul2}\, 
p^{\ul2})\delta","\eta_{\Delta_*}"'near end]\\
\Mod_R\ar[r,"(1^{\ul2}\,\Delta)_*\bar\Delta_! 
p^*"']&\Mod_R^{M^{\ul2\sqcup\ul2}}\ar[r,"(1^{\ul2}\, 
p^{\ul2})_!"']&\Mod_R^{M^{\ul2}}.
\end{tikzcd}
\]

We remark that not all functors appearing in the previous diagram are necessarily colimit preserving and strictly $R$-linear; nevertheless the functors $\Id_{\Mod_R^M}$, $(p^{\ul2})^*$ and $\fe_\Psi=(1^{\ul2}\, 
p^{\ul2})_!(1^{\ul2}\,\Delta)_*\bar\Delta_! 
p^*$ have these properties (the latter by Lemma \ref{lem:feisleftadjoint}), and these are the only functors occurring in the rectangle giving $\fe_\hPsi$, which therefore is a well-defined 1-morphism in $((\SPrL_R)_{\Mod_R/}^\llax)^\twosimeq\subseteq(\Fun([1],\SPrL_R)^\llax)^\twosimeq$.

By virtue of Remark \ref{rem:bothhomotopies}, it suffices to construct a homotopy
$\beta_\hPsi$ connencting the 1-morphisms $((\ft_\hPsi\circ \fI_\hPsi(\mu))\otimes\Id_{\fI_\hPsi(\ul 1)})\circ(\Id_{\fI_\hPsi(\ul 1)}\otimes \fe_\hPsi)$ and $\Id_{\fI_\hPsi(\ul 1)}$. Concretely, we need to provide a filling of 
the following prism, in which all vertices, edges and vertical faces are already 
filled via the data provided so far; in particular the back vertical face is 
filled with the identity of $p^*$.
Moreover, the requirement on $d_1\hPsi_R(M)$ and $d_0\hPsi_R(M)$ forces us to 
fill the the top face with the identity of $\Id_{\Mod_R}$, and the bottom face 
with $\beta_\Psi$.

\[
\begin{tikzcd}[column sep=65pt,row sep=30pt]
\Mod_R\ar[dd,"p^*"',near start]\ar[dr,"="]\ar[rrr,"=",""{name=A,inner 
sep=2pt,below, pos=.6},""{name=E,inner 
sep=2pt,below, pos=.7}]& & &\Mod_R\ar[dd,"p^*",near 
start]\ar[dddl,Rightarrow,"\eta_{(p\,1)_*}","p^{\ul3}(\Delta1)\simeq 
p^{\ul2}"'pos=.3,"\simeq p(p1)"'pos=.35] \\
&\Mod_R\ar[dd,"(1\delta)_!(1\Delta)_*(1\Delta)^* (p^{\ul3})^*", pos=.1]\ar[dl,Rightarrow,"BC"'near 
end,"p^{\ul3}(1\Delta)\simeq 
p^{\ul2}"'pos=.2,"\simeq p(1p)"'pos=.4]\ar[r,"="pos=.3,""{name=B,inner 
sep=2pt,above}]&\Mod_R\ar[dd,"(p^{\ul3})^*",near 
start]\ar[ddl,Rightarrow,"\eta_{(1\,\Delta)_*}"pos=.3,"(1^{\ul3}p^{\ul2} 
)(1g)\simeq 1^{\ul3}"'pos=.4]\ar[ur,"="']\\
\Mod_R^M\ar[dr,"(1^{\ul3}\,\Delta)_*(1\,\bar\Delta)_! 
(1\,p)^*"',very near end]\ar[rrr,"="pos=.2,dashed,""{name=C,inner 
sep=2pt,below, pos=.6},""{name=F,inner 
sep=2pt,above, pos=.4}]& & &
\Mod_R^M\\
&\Mod_R^{M^{\ul5}}\ar[r,"(1^{\ul3}\, p^{\ul2})_!"',""{name=D,inner 
sep=2pt,above}]&\Mod_R^{M^{\ul3}}.\ar[ur,"(p\,1)_*(\Delta\,1)^*"']
\ar[from=B, to=A, equal]
\ar[from=D, to=C, dotted,Rightarrow, "\beta_\Psi"',"\simeq"near start]
\ar[from=E, to=F, dotted,equal]
\end{tikzcd}
\]
So all that is left to do, in order to complete the definition of $\beta_\hPsi$, is 
to fill the interior of the prism, which is equivalent to providing a 
filling of the following diagram.
\begin{lem}
\label{lem:fill3cell}
The identity of $p^*$ is homotopic to the cyclic composition in the following diagram in $\Fun(\Mod_R,\Mod_R^M)$, starting on the top left corner; moreover the right vertical arrow is invertible:
\[
\begin{tikzcd}[column sep=70pt]
p^*
\ar[d,"\eta_{(p\,1)_*}p^*","p^{\ul3}(\Delta\,1)\simeq p(p\,1)"'] & 
(p\,1)_*(\Delta\,1)^*(1^{\ul3}\,p^{\ul2})_!(1^{\ul3}\,\Delta)_*(1\,
\bar\Delta)_!(1\,p)^*p^*\ar[l,"\beta_\Psi p^*"',"\simeq"]\\
(p\,1)_*(\Delta\,1)^*(p^{\ul3})^*\ar[r,"(p1)_*\!(\Delta1)^*\eta_{(1\Delta)_*}\!(p^{\ul3})^*","(1^{\ul3}\,p^{
\ul2 })(1\,\delta)\simeq 1^{\ul3}"']
&(p\,1)_*(\Delta\,1)^*(1^{\ul3}\,p^{\ul2})_!(1\,\delta)_!(1\,\Delta)_*(1\,
\Delta)^*(p^{\ul3})^*.\ar[u,"\simeq?","BC\text{, }p^{\ul3}(1\,\Delta)\simeq p(1\,p)"']
 \end{tikzcd}
\]
\end{lem}
\begin{proof}
We apply the following procedure to replace the given hollow square by a more convenient one: by means of compatible Beck--Chevalley equivalences, we bring the term ``$(\Delta\,1)^*$'' to the right in all the three corners of the above hollow
square in which it appears; we then suitably recombine all terms 
``$(-)^*$'' that have accumulated at the right end of each corner. The result is the 
following hollow square, which needs to be suitably filled:
\[
\begin{tikzcd}[column sep=60pt]
p^*\ar[d,"\eta_{(p\,1)_*}p^*"] & 
(p\,1)_*(1^{\ul2}\,p^{\ul2})_!(1^{\ul2}\,\Delta)_*
\bar\Delta_!p^*\ar[l,"\beta_{\Psi,1}\,p^*"',"\simeq"]\\
(p\,1)_*(p\,1)^*p^*\ar[r,"(p1)_*\eta_{\Delta_*}(p1)^*p^*","(1^{\ul2}\,p^{
\ul2 })\delta\simeq 1^{\ul2}"']
&(p\,1)_*(1^{\ul2}\,p^{\ul2})_!\delta_!\Delta_*
\Delta^*(p\,1)^*p^*.\ar[u,"\simeq?","BC\text{, }(p\,1)\Delta\simeq 1"']
 \end{tikzcd}
\]
Note that the right vertical map of the last square is the horizontal composition of the Beck--Chevalley transformation from Lemma \ref{lem:BCinvertible} with $p^*$, hence it is invertible.

We next observe that the composition of $\eta_{(p\,1)_*}$ and 
$(p\,1)_*\eta_{\Delta_*}(p\,1)^*$ is homotopic to $\eta_{((p\,1)\Delta)_*}$, 
which is homotopic to the composite equivalence $\Id_{\Mod_R}\simeq (p\,1)_*\Delta_*\simeq (p\,1)_*\Delta_*\Delta^*(p\,1)^*$, as $(p\,1)\Delta\simeq 1$. This implies that the 
left-bottom-right composite in the above square is homotopic to the bottom 
composite in the following commutative diagram, whose top composite is by definition the inverse of
$\beta_{\Psi,2}\,p^*$:
\[
\begin{tikzcd}[column sep=6pt]
\! p^*\ar[r,"\simeq"',"(p1)\Delta\simeq 
1^{\ul2}"] &(p1)_*\Delta_*p^*\!\ar[d,"\simeq"',"(p1)\Delta\simeq 
1^{\ul2}"]\ar[r,"(1^{\ul2}p^{\ul2} 
)\delta\simeq 1^{\ul2}"]&\!(p1)_*(1^{\ul2}p^{\ul2})_! 
\delta_!\Delta_*p^*\!\!\ar[d,"\simeq"',"(p1)\Delta\simeq 
1^{\ul2}"]\ar[r,"BC","\simeq"']& \!\!(p1)_*(1^{\ul2}p^{\ul2})_! 
(1^{\ul2}\Delta)_*\bar\Delta_!p^* \\
&(p1)_*\Delta_*\Delta^*(p1)^*p^*\ar[r] &(p1)_*(1^{\ul2}p^{\ul2}
)_!\delta_!\Delta_*\Delta^*(p1)^*p^*.
\end{tikzcd}
\]
\end{proof}
\begin{conj}
In order to prove that the data $(\fI_\hPsi,\ft_\hPsi)$ from Definition \ref{defn:hPsiRM} yields a commutative Frobenius algebra in $((\SPrL_R)^\llax_{\Mod_R/})^\twosimeq$, we have only used that $M$ is an $R$-Poincar\'e duality space, whereas the datum of an $R$-orientation was only useful to simplify some formulas.

We conjecture that the converse holds true: if $M$ is a compact space such that the analogous data $(\fI_\hPsi,\ft_\hPsi)$ as in Definition \ref{defn:hPsiRM} is a commutative Frobenius algebra in $((\SPrL_R)^\llax_{\Mod_R/})^\twosimeq$, then $M$ is an $R$-Poincar\'e duality space.
\end{conj}

\subsection{The functor \texorpdfstring{$\Phi_R(M)$}{PhiR(M)}}
We conclude the section by introducing a variation $\Phi_R(M)$ of the functor $\Psi_R(M)$ from Subsection \ref{subsec:PsiRMdef}.
\begin{defn}
\label{lem:Phiviabardash}
Let $M$ be a generic space.
The symmetric monoidal $\infty$ functor $\Phi_R(M)\colon\Gr\to(\SPrL_R)^\twosimeq$ is defined as the composite symmetric monoidal functor
\[
\begin{tikzcd}
{\Gr}\ar[r,hook]&\Cospan(\cS)\ar[r,"M^{(-)}"]&\Span(\cS)\ar[r,"\Mod_R^{(-)}"]&(\SPrL_R)^\twosimeq,
\end{tikzcd}
\]
where $\Mod_R^{(-)}\colon\Span(\cS)\to(\SPrL_R)^\twosimeq$ 
is introduced in Subsection
\ref{subsec:cospansspans}.
\end{defn}

The next goal is to prove Proposition \ref{prop:PsitwistedPhi}, which provides a more precise connection between $\Psi_R(M)$ and $\Phi_R(M)$.

\begin{nota}
\label{nota:ell}
We consider the symmetric monoidal functor
\[
\ell\colon \bB\Pic(R)\to(\SPrL_R)^\twosimeq
\]
classifying the action of $\Pic(R)\subseteq\Mod_R$ on $\Mod_R$ by tensor product. Composing $\ell$ with the functor $\xi_d^R$ from Subsection \ref{subsec:picard}, we obtain a symmetric monoidal functor $\ell\xi_d^R\colon\Gr\to(\SPrL_R)^\twosimeq$.
\end{nota}
Taking pointwise tensor product in $(\SPrL_R)^\twosimeq$, we can now consider the symmetric monoidal functor $\Psi_R(M)\otimes \ell\xi_d^R\colon\Gr\to\SPrL_R$. 
\begin{prop}
\label{prop:PsitwistedPhi}
Let $M$ be an oriented, $d$-dimensional $R$-Poincar\'e duality space. Then there is an equivalence of symmetric monoidal functors
\[
 \Psi_R(M)\otimes \ell\xi_d^R\simeq\Phi_R(M)\colon\Gr\to(\SPrL_R)^\twosimeq.
\]
\end{prop}

\begin{proof}
We denote in the following by 
$(\fI_{\Pxi},\ft_\Pxi)$ the underlying commutative Frobenius algebra for 
$\Psi_R(M)\otimes \ell\xi_d^R$, and compare $(\fI_{\Pxi},\ft_\Pxi)$ with $(\fI_{\Phi},\ft_\Phi)$.
We compare the underlying commutative Frobenius algebras of $\Psi_R(M)\otimes \ell\xi_d^R$ and $\Psi_R(M)$.
\begin{itemize}
\item The underlying symmetric monoidal functor $\fI_\Phi\colon\Fin\to(\SPrL_R)^\twosimeq$ is the same as $\fI_\Psi$, i.e. the symmetric monoidal functor classifying $\Mod_R^M$.  This is immediate for $\fI_\Phi$. For $\fI_\Pxi=\fI_\Psi\otimes \ell\fI_{\xi_d^R}$,
we use that the symmetric monoidal functor $\fI_{\xi_d^R}\colon\Fin\to  \bB\Pic(R)$ is constant with value $*\in  \bB\Pic(R)$, and composing with $\ell$ we obtain the constant functor $\ell\fI_{\xi_d^R}\colon\Fin\to(\SPrL_R)^\twosimeq$ at $\Mod_R$, i.e. the functor classifying the monoidal unit $\Mod_R$ as a commutative algebra. The tensor product of $\Mod_R^M$ and $\Mod_R$ is then readily identified, as a commutative algebra, with $\Mod_R^M$.
\item Use Notation \ref{nota:s_equals_Rd}. The functor $\ft_\Pxi$ is the tensor product in $(\SPrL_R)^\twosimeq$ of the functors $\ft_\Psi=p_*\colon\Mod_R^M\to\Mod_R$ and $\ell\ft_{\xi_d^R}=(-\otimes s)\colon\Mod_R\to\Mod_R$; this gives an equivalence $\ft_\Pxi\simeq p_*(-)\otimes_R s$, which by Poincar\'e duality is equivalent to $p_!\simeq \ft_\Phi$.
\end{itemize}
\end{proof}

\begin{rem}
\label{rem:reformulation}
Since the symmetric monoidal functor $\ell\xi_d^R\otimes\ell\xi_{-d}^R\colon\Gr\to\SPrL_R$ is equivalent to the constant functor at $\Mod_R$, the statement of proposition \ref{prop:PsitwistedPhi} is equivalent to a symmetric monoidal equivalence $\Psi_R(M)\simeq\Phi_R(M)\otimes\ell\xi_{-d}^R$; we will in fact employ this reformulation in Section \ref{sec:GFTfunctor}.
\end{rem}

\section{The functor \texorpdfstring{$\GFT_M$}{GFTM}}
\label{sec:GFTfunctor}
In this section we construct the symmetric monoidal $\Mod_R$-enriched functor $\GFT_M\colon(\GrCob_!\xi_d^R)^\op\to\Mod_R$ from Theorem \ref{thm:A}, associated with an $E_\infty$-ring spectrum $R$ and a $d$-dimensional oriented $R$-Poincar\'e duality space $M$.
\subsection{The functors \texorpdfstring{$Z_1,Z_2,Z_3$}{Z1, Z2, Z3}}
We start by concatenating two symmetric monoi\-dal $\infty$-functors $\Gr_\cS\to(\Fun([1],\SPrL_R)^\llax)^\twosimeq$.
The first, denoted $Z_1$, is the composite functor
\[
\begin{tikzcd}
Z_1\colon{\Gr_\cS}\ar[r,"\cG"]&{\Gr}\ar[r,"\hPsi_R(M)"]&((\SPrL_R)_{\Mod_R/}^\llax)^\twosimeq
\end{tikzcd}
\]
obtained by composing the functor $\hPsi_R(M)$ from in Subsection \ref{subsec:unitalenhancement} with the functor $\cG$ from Definition \ref{defn:GrS}. The second, denoted  $Z_2\otimes\ell\xi_{-d}^R$, is the pointwise tensor product in $(\Fun([1],\SPrL_R)^\llax)^\twosimeq$ of the following two composite functors denoted $Z_2$ and, by abuse of notation, $\ell\xi_{-d}^R$:
\[
\begin{tikzcd}
Z_2\colon{\Gr_\cS} \ar[r,"\cF"]&(\bCospan(\cS,\sqcup)^\rlax_{/\emptyset})^\twosimeq \ar[r,"M^{(-)}"]& (\bSpan(\cS,\times)^\llax_{/*})^\twosimeq\ar[dl,"\Mod_R^{(-)}"'very near end]\\
&((\SPrL_R)^\llax_{/\Mod_R})^\twosimeq\ar[r,hook]&(\Fun([1],\SPrL_R)^\llax)^\twosimeq;
\end{tikzcd}
\]
\[
\begin{tikzcd}[column sep=15pt]
\ell\xi_{-d}^R\colon{\Gr_\cS} \ar[r,"\cG"]&{\Gr} \ar[r,"\xi_{-d}^R"]&\bB\Pic(R)\ar[r,"\ell"]&\SPrL_R\ar[r,"s_0"]&(\Fun([1],\SPrL_R)^\llax)^\twosimeq.
\end{tikzcd}
\]

In the composite $Z_2$, the first arrow is the the functor $\cF$ from Definition \ref{defn:GrS}. The second arrow is obtained from the functor $M^{(-)}\colon\cS\to\cS^\op$ by passing to right lax comma $\infty$-categories of cospan $\inftwo$-categories, and by then replacing ``$\cS^\op$'' back to ``$\cS$'' at the cost of also replacing ``$\bCospan$'' by ``$\bSpan$'' and ``$\rlax$'' by ``$\llax$''. The diagonal arrow is induced from the symmetric monoidal $\inftwo$-functor $\Mod_R^{(-)}\colon\bSpan(\cS)\to\SPrL_R$ from Subsection \ref{subsec:cospansspans}.

The second functor $\ell\xi_{-d}^R$ involves the homonymous functor $\ell\xi_{-d}^R$ from Notation \ref{nota:ell}; the last arrow labeled ``$s_0$'' is induced by the degeneracy $s^0\colon[1]\to[0]$ in $\bDelta$.

We next compare the composites $d_0Z_1$ and $d_1(Z_2\otimes\ell\xi_{-d}^R)$. By the fact that $\hPsi_R(M)$ is a ``unital enhancement'' of $\Psi_R(M)$, we have
$d_0Z_1\simeq d_0\hPsi_R(M)\cG\simeq \Psi_R(M)\cG$. Combining Lemma \ref{lem:Phiviabardash}, the pullback square defining $\Gr_\cS$, and the evident way in which $d_1$ interleaves the various functors labeled ``$M^{(-)}$'' and ``$\Mod_R^{(-)}$'', we have a composite equivalence:
\[
\begin{split}
d_1(Z_2\otimes\ell\xi_{-d}^R)& \simeq (d_1Z_2)\otimes\ell\xi_{-d}^R
\simeq(d_1\circ\Mod_R^{(-)}\circ M^{(-)}\circ\cF)\otimes\ell\xi_{-d}^R\\
&\simeq (\Mod_R^{(-)}\circ M^{(-)}\circ d_1\circ\cF)\otimes\ell\xi_{-d}\\
&\simeq (\Mod_R^{(-)}\circ M^{(-)}\circ \Re\circ\cG) \otimes\ell\xi_{-d}^R\\
&\simeq(\Phi_R(M)\cG)\otimes\ell\xi_{-d}^R\simeq (\Phi_R(M)\otimes\ell\xi_{-d}^R)\cG.
\end{split}
\]
By Proposition \ref{prop:PsitwistedPhi} (see in particular Remark \ref{rem:reformulation}), we then have an equivalence $d_0Z_1\simeq d_1(Z_2\otimes\ell\xi_{-d}^R)$.

\begin{nota}
\label{nota:Z3}
We denote by $Z_3\colon\Gr_\cS\to((\SPrL_R)^\llax_{\Mod_R/})^\twosimeq$ the concatenation of the two functors $Z_1$ and $Z_2\otimes\ell\xi_{-d}^R$.
\end{nota}
\begin{lem}
\label{lem:Zthreefactors}
The $\infty$-functor $Z_3$ factors through the localisation $\Gr_\cS\to\GrCob$.
\end{lem}
\begin{proof}
The image of $(A\to X)\in\Gr_\cS$ along $Z_3$ is the object in  $((\SPrL_R)^\llax_{\Mod_R/})^\twosimeq$ given by the composite functor
\[
\begin{tikzcd}[column sep=50pt]
\Mod_R\ar[r,"p(M^A)^*"]&\Mod_R^{M^A}\ar[r,"(M^X\to M^A)^*"] &\Mod_R^{M^X}\ar[r,"p(M^X)_!"]&\Mod_R,
\end{tikzcd}
\]
where the first arrow accounts for $Z_1$ and the second two for $Z_2\otimes\ell\xi_{-d}^R$. This shows that $Z_3(A\to X)$ is equivalent to the endofunctor $p(M^X)_!p(M^X)^*$ of $\Mod_R$, which in turn is equivalent to $-\otimes_R(M^X\otimes R)$.

Without loss of generality, we may now consider an idle 1-morphism as in Remark \ref{rem:idleshape}, and prove that it is inverted along $Z_3$.
This 1-morphism is sent along $Z_3$ to the outer square of the following commutative diagram in $(\SPrL_R)^\twosimeq$, considered as a morphism between vertical composites. Since the top and bottom rows are invertible functors, we obtain indeed an invertible 1-morphism in $((\SPrL_R)^\llax_{\Mod_R/})^\twosimeq$:
\[
\begin{tikzcd}[column sep=80pt, row sep=10pt]
\Mod_R\ar[rr,equal]\ar[d,"p(M^B)^*"]& &\Mod_R\ar[d,"p(M^A)^*"]\\
\Mod_R^{M^B}\ar[d,"(M^Y\to M^B)^*"]\ar[r,"(M^A\to M^B)^*"] &\Mod_R^{M^A}\ar[d,"(M^X\to M^A)^*"]\ar[r,equal,"(\Id_{M^A})_!"]&\Mod_R^{M^A}\ar[d,"(M^X\to M^A)^*"]\\
\Mod_R^{M^Y}\ar[d,"p(M^X)_!"]\ar[r,"(M^X\simeq M^Y)^*","\simeq"']&\Mod_R^{M^X}\ar[d,"p(M^X)_!"]\ar[r,equal,"(\Id_{M^X})_!"]&\Mod_R^{M^X}\ar[d,"p(M^X)_!"]\\
\Mod_R\ar[r,equal] &\Mod_R\ar[r,equal]&\Mod_R.
\end{tikzcd}
\]
\end{proof}

\subsection{The lax comma categories of \texorpdfstring{$\bB\Mod_R$}{BModR}}
Recall that a colimit-preserving, $R$-linear endofunctor of $\Mod_R$ is of the form $-\otimes_R\zeta$ for some $\zeta\in\Mod_R$, and that the space of $R$-linear natural transformations between $-\otimes_R\zeta$ and $-\otimes_R\zeta'$ is equivalent to $\Mod_R(\zeta,\zeta')$. In few words, the full $\inftwo$-subcategory of $\SPrL_R$ spanned by the object $\Mod_R$ is equivalent to the pointed $\inftwo$-category $\bB\Mod_R$ obtained by delooping once the symmetric monoidal $\infty$-category $\Mod_R$, see Subsection \ref{subsec:deloopingGray}. The core groupoid of $\bB\Mod_R$ is equivalent to $\bB\Pic(R)$, with basepoint denoted by $*$; the endomorphism $\infty$-category $\bB\Mod_R(*,*)$ is equivalent to $\Mod_R$, and horizontal composition is given by tensor product of $R$-modules.

\begin{defn}
We define $((\bB\Mod_R)_{*/}^\llax)^\twosimeq_{\bB\Pic}$ as the wide symmetric monoidal $\infty$-subcategory of $((\bB\Mod_R)_{*/}^\llax)^\twosimeq$ spanned by those morphisms that are sent along the target functor $d_0\colon((\bB\Mod_R)_{*/}^\llax)^\twosimeq\to (\bB\Mod_R)^\twosimeq$ to invertible morphisms. In symbols we have
\[
((\bB\Mod_R)_{*/}^\llax)^\twosimeq_{\bB\Pic}=((\bB\Mod_R)_{*/}^\llax)^\twosimeq\times_{(B\Mod_R)^\twosimeq}\bB\Pic(R).
\]
We define similarly $((\bB\Mod_R)_{*/}^\rlax)^\twosimeq_{\bB\Pic}\subset((\bB\Mod_R)_{*/}^\rlax)^\twosimeq$.
\end{defn}
Roughly speaking, objects in the $\infty$-category $((\bB\Mod_R)_{*/}^\rlax)^\twosimeq_\Pic$ are $R$-modu\-les; a morphism $\zeta\to\zeta'$ in $((\bB\Mod_R)_{*/}^\rlax)^\twosimeq_{\bB\Pic}$ is the datum of an invertible $R$-module $\zeta''$ and a morphism of $R$-modules $\zeta\otimes_R\zeta''\to\zeta'$. The description of $((\bB\Mod_R)_{*/}^\llax)^\twosimeq_{\bB\Pic}$ is similar, except that a morphism $\zeta\to\zeta'$ is the datum of an invertible $R$-module $\zeta''$ and a morphism $\zeta'\to\zeta\otimes_R\zeta''$.
\begin{rem}
Barkan--Steinebrunner consider an alternative construction of the symmetric monoidal $\infty$-category $((\bB\Mod_R)_{*/}^\rlax)^\twosimeq_{\bB\Pic}$. They consider the action of $\Pic(R)$ on $\Mod_R$ given by tensor product; there is an associated lax symmetric monoidal functor $\bB\Pic(R)\to\Catinfty^{\fU_2}$, and its colimit in $\Catinfty^{\fU_2}$, denoted $\Mod_R/\Pic(R)$, is a symmetric monoidal $\infty$-category equivalent to our $((\bB\Mod_R)_{*/}^\rlax)^\twosimeq_{\bB\Pic}$.
\end{rem}

As observed in the proof of Lemma \ref{lem:Zthreefactors}, the functor $Z_3$ takes image inside the $\infty$-subcategory of
$((\SPrL_R|_{\Mod_R})^\llax_{\Mod_R/})^\twosimeq$, comprising all objects and those 1-morphisms that are sent along $d_0$ to an invertible 1-morphism; this $\infty$-subcategory corresponds precisely to $((\bB\Mod_R)_{*/}^\llax)^\twosimeq_{\bB\Pic}$ along the equivalence
\[
((\SPrL_R|_{\Mod_R})^\llax_{\Mod_R/})^\twosimeq\simeq ((\bB\Mod_R)_{*/}^\llax)^\twosimeq.
\]
We shall henceforth abuse notation and regard $Z_3$ as a symmetric monoidal $\infty$-functor $\GrCob\to((\bB\Mod_R)_{*/}^\llax)^\twosimeq_{\bB\Pic}$.

We also have a commutative square in $\CAlg(\Catinfty^{\fU_2})$, whose vertical maps are the target functors, and whose bottom equivalence is induced by the symmetric monoidal automorphism $(-)^{-1}\colon\Pic(R)\to\Pic(R)$:
\[
\begin{tikzcd}
\big(((\bB\Mod_R)_{*/}^\llax)^\twosimeq_{\bB\Pic}\big)^\op\ar[r,"\simeq"']\ar[d,"d_0"]& ((\bB\Mod_R)_{*/}^\rlax)^\twosimeq_{\bB\Pic}\ar[d,"d_0"]\\
\bB\Pic(R)^\op\ar[r,"(-)^{-1}","\simeq"']&\bB\Pic(R).
\end{tikzcd}
\]
In particular we may pass to opposite categories and consider the symmetric monoi\-dal $\infty$-functor $Z_3^\op\colon\GrCob^\op\to((\bB\Mod_R)_{*/}^\rlax)^\twosimeq_{\bB\Pic}$. The composite functor
\[
d_0Z_3^\op\colon\GrCob^\op\to \bB\Pic(R)
\]
agrees with $(\xi_d^R)^\op$, and both functors are constant at $*$ on spaces of objects.

\begin{nota}
We let $\Seg^{\fU_2}\subset s\cS^{\fU_2}$ denote the full $\infty$-subcategory spanned by (possibly non-complete) $\fU_2$-small Segal spaces, i.e. by functors $X\colon\bDelta^\op\to\cS^{\fU_2}$ satisfying the Segal condition: for all $p\ge2$ the canonical map of spaces $X([p])\to X([1])\times_{X([0])}X([1])\times_{X([0])}\dots\times_{X([0])}X([1])$ is an equivalence.
\end{nota}
Note that $\Seg^{\fU_2}\subset s\cS^{\fU_2}$ is closed under products, i.e. it is a symmetric monoidal $\infty$-subcategory.
\begin{defn}
We denote by $\bbB \Pic(R)\in \CAlg(\Seg^{\fU_2})$ the Segal space given by the following pullback of symmetric monoidal Segal spaces
\[
\begin{tikzcd}[row sep=10pt]
\bbB\Pic(R)\ar[dr,phantom,"\lrcorner"very near start]\ar[r]\ar[d]&\bB\Pic(R)\ar[d,"\eta_{[0]_*}"]\\
*\ar[r]& {[0]_*[0]^*}\bB\Pic(R),
\end{tikzcd}
\]
where $[0]\colon*\to\bDelta^\op$ is the functor picking the object $[0]$, and where $\bB\Pic(R)$ is considered as a constant simplicial space. The bottom right corner is informally the Segal space sending $[p]\mapsto (\bB\Pic(R))^{p+1}$, with face maps forgetting coordinates and degeneracy maps inserting $*\in\bB\Pic(R)$; its symmetric monoidal structure relies on the fact that both $[0]^*$ and $[0]_*$ are product preserving.
\end{defn}
Intuitively, $\bbB\Pic(R)([p])\simeq\Pic(R)^p$; face maps are given by forgetting the first or last coordinate, or by multiplying two consecutive coordinates, using the group structure on $\Pic(R)$; and degeneracy maps are given by inserting $*\in\Pic(R)$. 

By definition, for $X\in\Seg^{\fU_2}$, a morphism $X\to\bbB\Pic(R)$ is the same as a morphism $X\to\bB\Pic(R)$ together with a nullhomotopy of the induced map of spaces $X([0])\to\bB\Pic(R)$.
The previous discussion allows us to consider $(\xi_d^R)^\op$ as a symmetric monoidal map of Segal spaces $\GrCob^\op\to\bbB\Pic(R)$, i.e. as an object in $\CAlg(\Seg^{\fU_2}_{/\bbB\Pic(R)})$. Here we model the $\infty$-category $\GrCob^\op$ as a \emph{complete} Segal space.
\begin{defn}
We let $((\bB\Mod_R)_{*/}^\rlax)^\twosimeq_{\bbB\Pic}\in \CAlg(\Seg^{\fU_2}_{/\bbB\Pic(R)})$ denote the pullback of symmetric monoidal Segal spaces
\[
\begin{tikzcd}[row sep=10pt]
((\bB\Mod_R)_{*/}^\rlax)^\twosimeq_{\bbB\Pic}\ar[r]\ar[d]\ar[dr,phantom,"\lrcorner"very near start]&((\bB\Mod_R)_{*/}^\rlax)^\twosimeq_{\bB\Pic}\ar[d]\\
\bbB\Pic(R)\ar[r]&\bB\Pic(R)
\end{tikzcd}
\]
\end{defn}

The previous discussion also allows us to regard $Z_3^\op$ as a morphism from $\GrCob^\op$ to $((\bB\Mod_R)_{*/}^\rlax)^\twosimeq_{\bbB\Pic}$ in $\CAlg(\Seg^{\fU_2}_{/\bbB\Pic(R)})$.

\subsection{Categories enriched in presheaf categories}
In this subsection we make use of the theory of enriched $\infty$-categories as developed by Gepner--Haugseng \cite{GepnerHaugseng}. For a (symmetric) monoidal $\infty$-category $\cC$, Gepner--Haugseng introduce an $\infty$-category $\Alg_\cat(\cC)$ of ``categorical algebras in $\cC$'', whose objects are pairs consisting of a space $S\in\cS$ and a map of generalised non-symmetric operads $\bDelta^\op_S\to\cC$. 
If $\cC$ is symmetric monoidal, then also $\Alg_\cat(\cC)$ admits a symmetric monoidal structures, so that it makes sense to talk of symmetric monoidal categorical algebras in $\cC$.
A lax (symmetric) monoidal functor of (symmetric) monoidal $\infty$-categories $\cC\to\cC'$ allows one to ``transform'' a categorical algebra in $\cC$ into a categorical algebra in $\cC'$. Gepner--Haugseng also introduce an $\infty$-category of ``$\cC$-enriched $\infty$-categories'', denoted $\Cat_\infty^\cC$, which is a certain reflective full $\infty$-subcategory of $\Alg_\cat(\cC)$ spanned by ``complete'' objects; we shall content ourselves with categorical algebras here, and leave the discussion about completions to the interested reader. We will therefore also use terminology that clashes with \cite{GepnerHaugseng} and refer to an object in $\Alg_\cat(\cC)$ as a ``$\cC$-enriched $\infty$-category'', and to a morphism in $\Alg_\cat(\cC)$ as a ``$\cC$-enriched $\infty$-functor''.

The identification $\Alg_\cat(\cS^{\fU_2})\simeq\Seg^{\fU_2}$ from \cite[Theorem 4.4.7]{GepnerHaugseng} is symmetric monoidal, as both $\infty$-categories carry the cartesian symmetric monoidal structure. Using this identification, we can get a symmetric monoidal identification $\Alg_\cat(\cS^{\fU_2}_{/\Pic(R)})\simeq\Seg^{\fU_2}_{/\bbB\Pic(R)}$. It follows that the symmetric monoidal map of Segal spaces $Z_3^\op\colon\GrCob^\op\to((\bB\Mod_R)_{*/}^\rlax)^\twosimeq_{\bbB\Pic}$ corresponds to a symmetric monoidal $\cS^{\fU_2}_{/\Pic(R)}$-enriched functor, which for simplicity we still denote 
$Z_3^\op\colon\GrCob^\op\to((\bB\Mod_R)_{*/}^\rlax)^\twosimeq_{\bbB\Pic}$.

We next recall that $\cS^{\fU_2}_{/\Pic(R)}\simeq\PSh(\Pic(R),\cS^{\fU_2})$ in $\CAlg(\Catinfty^{\fU_3})$, see \cite[Corollary D]{Ramzi}. In this light, the $\cS^{\fU_2}_{/\Pic(R)}$-enriched category $((\bB\Mod_R)_{*/}^\rlax)^\twosimeq_\Pic$ corresponds to the category $\Mod_R$, whose enrichment in $\PSh(\Pic(R),\cS^{\fU_2})$ is given by first considering the self-enrichment of $\Mod_R$ and then by applying the following sequence of lax symmetric monoidal functors, where $\kappa<\fU_1$ is a regular cardinal such that $\Mod_R$ is $\kappa$-accessible (see Subsection \ref{subsec:presheaves} for background):
\[
\begin{tikzcd}[column sep=40pt]
\Mod_R\ar[r,"\yo_\kappa"]&\PSh(\Mod_R^\kappa)\ar[rr,"(\Pic(R)\hto\Mod_R^\kappa)^*"]& &\PSh(\Pic(R))\subset\PSh(\Pic(R);\cS^{\fU_2}).
\end{tikzcd}
\]
We thus learn that $((\bB\Mod_R)_{*/}^\rlax)^\twosimeq_\Pic$ is in fact enriched in the smaller category $\PSh(\Pic(R))=\PSh(\Pic(R);\cS^{\fU_1})$; the same holds also for $\GrCob^\op$, and hence $Z_3^\op$ can be regarded as a $\PSh(\Pic(R))$-enriched symmetric monoidal $\infty$-functor. By applying the symmetric monoidal left adjoints $(\Pic(R)\hto\Mod_R^\kappa)_!$ and $\yo'_\kappa$ to the previous composite of lax symmetric monoidal right adjoints, we finally obtain a $\Mod_R$-enriched symmetric monoidal $\infty$-category $(\GrCob_!\xi_d^R)^\op$ as the image of $\GrCob^\op\in\Alg_\cat(\PSh(\Pic(R)))$, and a $\Mod_R$-enriched symmetric monoidal functor
\[
\GFT_M\colon(\GrCob_!\xi_d^R)^\op\to\Mod_R.
\]
This concludes the proof of Theorem \ref{thm:A}, up to checking that $\GFT_M$ indeed recovers the basic string operations (1)-(3) from Subsection \ref{subsec:basic}; this last step will be achieved in Section \ref{sec:CSproduct}. The proof of Corollary \ref{cor:A1} from Theorem \ref{thm:A} is also immediate, using the fact that taking homotopy groups is a lax symmetric monoidal functor $\Mod_R=\Mod_R(\Sp)\to\Mod_{\pi_*(R)}(\Ab^\Z)$.
\begin{rem}
The above composite $(\Pic(R)\hto\Mod_R^\kappa)^*\yo_\kappa$, combined with the equivalence $\cS_{/\Pic(R)}\simeq\PSh(\Pic(R))$, gives rise to a lax symmetric monoidal functor $\Mod_R\to\cS_{/\Pic(R)}$; this functor associates with an $R$-module $\zeta$ the moduli space of $R$-linear maps $\zeta'\to\zeta$ with $\zeta'\in\Pic(R)$; this is a space over $\Pic(R)$ by forgetting the map and only retaining $\zeta'$. The left adjoint functor $\cS_{/\Pic(R)}\to\Mod_R$ is the functor sending $\zeta\colon X\to\Pic(R)$ to $\colim_X\zeta\in\Mod_R$. In particular the morphism module from $(\emptyset\to Y)$ to $(\emptyset\to X)$ in $\GrCob_!\xi_d^R$ is indeed $\colim_{\fM_\Gr(Y,X)}\xi_d^R$ as mentioned in Subsection \ref{subsec:graphfieldtheories}.
\end{rem}
\subsection{Characteristic classes of homotopy automorphisms}
We conclude the section by proving Corollary \ref{cor:A2}. The functor $\GFT_M\colon\colon(\GrCob_!\xi_d^R)^\op\to\Mod_R$ depends naturally on $M\in\fM_{\Poin,d}^{R,+}$, so that we obtain a map of spaces
\[
\GFT_{(-)}\colon\fM_{\Poin,d}^{R,+}\to\CAlg(\Alg_\cat(\Mod_R))\,((\GrCob_!\xi_d^R)^\op,\Mod_R).
\]
The crucial observation is now that composing with evaluation at the monoidal unit $\emptyset\in(\GrCob_!\xi_d^R)^\op$ gives rise to the constant map $\fM_{\Poin,d}^{R,+}\to\Mod_R^\simeq$ at $R\in\Mod_R$; indeed, for any $M\in\fM_{\Poin,d}^{R,+}$, the mapping space $M^{\emptyset}$ is canonically identified with $*\in\cS$, and thus $\GFT_M(\emptyset)\simeq M^\emptyset\otimes R$ is canonically identified with $R$.

We can thus compose $\GFT_{(-)}$ with the evaluation at $\GrCob_!\xi_d^R(\emptyset,\emptyset)\in\Mod_R$, i.e. the endomorphism object of $\emptyset$; we thus obtain a map of spaces
\[
\fM_{\Poin,d}^{R,+}\to\Mod_R(\colim_{\fM_\Gr(\emptyset,\emptyset)}\xi_d^R ,\underline{\Hom}_R(R,R)),
\]
where $\underline{\Hom}_R(R,R)\simeq R\in\Mod_R$ is the internal endomorphism object of $R\in\Mod_R$. The last datum is then equivalent to an $R$-linear pairing
\[
(\fM_{\Poin,d}^{R,+}\otimes R)\otimes_R(\colim_{\fM_\Gr(\emptyset,\emptyset)}\xi_d^R)\to R,
\]
and the statement of Corollary \ref{cor:A2} is obtained by restricting to the component in $\fM_{\Poin,d}^{R,+}$ of a given oriented $R$-Poincar\'e duality space $M$.

\section{Recovering the basic string operations}
\label{sec:CSproduct}
Let $R\in\CAlg(\Sp)$ and let
$M$ be a topological, closed, $R$-oriented $d$-dimensional manifold. In this subsection we conclude the proof of Theorem \ref{thm:A}, by identifying the basic operations (1)-(3) from Subsection \ref{subsec:basic} with the analogous operations provided by the graph field theory $\GFT_M$.

\subsection{Recovering the operations (1) and (2)}
We start with the basic operation (1). Given a map of spaces $f\colon Y\to X$, we consider it as the graph cobordism corresponding to the following diagram of spaces:
\[
\begin{tikzcd}[row sep=8pt]
\emptyset\ar[d]\ar[r,equal]&\emptyset\ar[r,equal]\ar[d]&\emptyset\ar[d]\ar[dl,phantom,"\urcorner"very near end]\\
Y\ar[r,"f"]&X\ar[r,equal]&X.
\end{tikzcd}
\]
The functor $Z_3$ from Notation \ref{nota:Z3} sends $f$ to the following diagram in the $(\infty,2)$-category $\SPrL_R$, giving a morphism in $((\SPrL_R)^\llax_{\Mod_R/})^\twosimeq$ between horizontal composites; here the left and right squares represent the contributions of $Z_1$ and $Z_2\otimes\ell\xi_{-d}^R$:
\[
\begin{tikzcd}[row sep=15pt, column sep=60pt]
\Mod_R\ar[r,equal]\ar[d,equal]&\Mod_R\ar[d,equal]\ar[r,"p(M^Y)_!p(M^Y)^*"]&\Mod_R\ar[d,equal]\\
\Mod_R\ar[r,equal]&\Mod_R\ar[r,"p(M^X)_!p(M^X)^*"']\ar[ur,Rightarrow,"\epsilon_{M^f_!}"near start]&\Mod_R.
\end{tikzcd}
\]
Evaluating at $R\in\Mod_R$ in the top left corner, we obtain that the image along $\GFT_M$ of the previous graph cobordism is the following morphism in $\Mod_R$:
\[
p(M^X)_!p(M^X)^*(R)\simeq p(M^Y)_!M^f_!(M^f)^*p(M^Y)^*(R)\xrightarrow{\epsilon_{M^f_!}}p(M^Y)_!p(M^Y)^*(R).
\]
After taking homotopy groups we obtain the map $H_*(M^f;R)$ as desired.

We next consider the basic operation (2). Given a decomposition $Y=X\sqcup *$, we obtain the graph cobordism corresponding to the following diagram of spaces:
\[
\begin{tikzcd}[row sep=9pt]
\emptyset\ar[r]\ar[d]&*\ar[d]&\emptyset\ar[l]\ar[d]\ar[dl,phantom,"\urcorner"very near end]\\
Y\ar[r,equal]&Y&X.\ar[l]
\end{tikzcd}
\]
The image of the previous along $Z_3$ is the following diagram in $\SPrL_R$, in which we abbreviated $p=p(M)\colon M\to*$, and we denote by $[-d]\simeq-\otimes R[-d]$ the $d$-fold desuspension functor:
\[
\begin{tikzcd}[column sep=60pt,row sep=15pt]
\Mod_R\ar[d,equal]\ar[r,equal]&\Mod_R\ar[d,"p_*p^*\simeq"'near end,"{p_!p^*[-d]}"near end]\ar[r,"p(M^Y)_!p(M^Y)^*"]&\Mod_R\ar[d,"{[-d]}"]\\
\Mod_R\ar[r,equal]\ar[ur,Rightarrow,"\eta_{p_*}"]&\Mod_R\ar[r,"p(M^X)_!p(M^X)^*"']&\Mod_R.
\end{tikzcd}
\]
Evaluating at $R$, we obtain that the image along $\GFT_M$ of the previous graph cobordism is the following morphism in $\Mod_R$, with target shifted by $-d$:
\[
\begin{split}
p(M^X)_!p(M^X)^*(R)&\simeq R\otimes p(M^X)_!p(M^X)^*(R)\xrightarrow{\eta_{p_*}\otimes\Id}\\
&\to p_*p^*(R)\otimes p(M^X)_!p(M^X)^*(R)\\
&\simeq p_!p^*(R)\otimes p(M^X)_!p(M^X)^*(R)[-d]\\
&\simeq p(M^Y)_!p(M^Y)^*(R)[-d].
\end{split}
\]
On homotopy groups, the previous composite agrees with the map $H_*(M^X;R)\to H_{*+d}(M^Y;R)$ given by cross product with the fundamental class of $M$.

\subsection{The Chas--Sullivan product}
\label{subsec:CSclassical}
We next recall the classical definition of the Chas-Sullivan product.
Let $X\in\Top$ be a topological space, and let $Y$ be a topolo\-gical space obtained from $X$ by attaching a 1-cell $I=[0,1]$ along a map $\psi\colon\del I\to X$. Up to replacing $X$ with the mapping cylinder of $\psi$, we may assume that $\psi$ is a cofibration, and in particular we may regard $\del I$ as a subset of $X$; in this case $Y$ is homotopy equivalent to the quotient $Y/I\cong X/\del I$, and $M^Y$ is homotopy equivalent to $M^{Y/I}\cong M^{X/\del I}$.
Evaluation at $\del I$ gives a map $e\colon M^X\to M^2$, which is a fibre bundle. If $\Delta\colon M\hookrightarrow M^2$ denotes the diagonal inclusion, we obtain a pullback square in $\Top$ as follows, with $e_1$ defined as the restriction of $e$ over $M$:
\[
 \begin{tikzcd}[row sep=10pt]
  M^Y\ar[r,phantom,"\simeq"]&M^{X/\del I}\ar[r, hook,"\tilde\Delta"]\ar[d,"e_1"']\ar[dr,phantom,"\lrcorner",very near start] & M^X\ar[d,"e"]\\
  &M\ar[r,hook,"\Delta"] & M^2.
 \end{tikzcd}
\]
Let $U$ be a neighbourhood of $\Delta(M)$ in $M^2$ that is
homeomorphic to the total space of the normal microbundle\footnote{For simplicity we will not distinguish between microbundles and topological $\R^d$-bundles.} $N_\Delta M\to M$. Let $V=e^{-1}(U)$ be the corresponding neighbourhood of $M^{X/\del I}$ in $M^X$. Then the quotient $M^2/(M^2\setminus U)$ is homeomorphic to the Thom space $\Th(N_\Delta M\to M)$; similarly, $V$ is homeomorphic to the total space of the microbundle $e_1^*(N_\Delta M)\to M^{X/\del I}$ , and $M^X/(M^X\setminus V)$ is homeomorphic to the Thom space $\Th(e_1^*(N_\Delta M)\to M^{X/\del I})$.

Recall from Definition \ref{defn:orientation} that an $R$-orientation on $M$ is an isomorphism $\omega_M^R\simeq p^*(R[-d])$, where $p=p(M)\colon M\to *$ is the terminal map. An $R$-orientation on $M$ yields an $R$-orientation on $M^2$, as product of the orientations of the two factors: more precisely, we have a canonical identification $\omega_{M^2}^R\simeq p_1^*(\omega_M^R)\otimes p_2^*(\omega_M^R)$, where $p_1,p_2\colon M^2\to M$ are the two projections, and we may identify the latter expression with $(p^2)^*(R[-d])\otimes (p^2)^*(R[-d])\simeq (p^2)^*(R[-2d])$ \footnote{The reader will notice that the last identification $R[-d]\otimes R[-d]\simeq R[-2d]$ also depends on a choice of \emph{order} of the two factors, i.e. it can be sensitive to swapping the two factors; for instance this occurs when $R=\Z$ and $d$ is odd. This is reflected in the generic non-triviality of the coefficient systems $\xi_d^R$.}.
The normal microbundle $N_\Delta M$, which is a topological $\R^d$-bundle over $M$, gives rise to a fibration of pointed spaces on $M$, with fibre the pointed space $S^d$, by taking one-point compactifications fibrewise; by further taking suspension spectra fibrewise we obtain a spherical fibration $n_\Delta M\in \Sp^M$, with fibres isomorphic to $\bS[d]$. We have moreover a canonical identification
\[
\Delta^*(\omega_{M^2}^\bS)\otimes_\bS n_\Delta M\simeq \omega_M^\bS,
\]
coming from the isomorphism of microbundles $\Delta^*(TM^2)\cong TM\oplus N_\Delta (M)$ and the identifications $\omega_M^\bS\simeq TM^{-1}$ and $\omega_{M^2}^\bS\simeq (TM^2)^{-1}$. Here $TM$ and $TM^2$ denote the tangent microbundles.

The $R$-orientations on $M$ and $M^2$, together with the isomorphism $n_\Delta M \otimes_\bS R\simeq \omega_M^R\otimes_R \Delta^*(\omega_{M^2}^R)^{-1}$, give an equivalence $n_\Delta M\simeq p^*(R[d])$; by pullback along $e_1$ we also obtain an equivalence $e_1^*(n_\Delta M)\otimes_\bS R\simeq p(M^{X/\del I})^*(R[d])$.

We can then define an operation $H_*(M^X;R)\to H_{*-d}(M^W;R)$, which we shall refer to as the \emph{general Chas--Sullivan operation}, as the composite
\[
\begin{split}
 &H_*(M^X;R)\to H_*(M^X, M^X\setminus V;R)\cong\tilde H_*(\Th(e_1^*(N_\Delta M)\to M^{X/\del I});R)\cong\\
 &H_*(M^{X/\del I};e_1^*(n\Delta(M))\otimes_\bS R)\cong H_*(M^{X/\del I};R[d])\cong H_{*-d}(M^Y;R).
\end{split}
\]
Here the isomorphism connecting the two lines is the Thom isomorphism, identifying the homology of a Thom space with the homology of the base twisted in the associated invertible parametrised $R$-module. 

The mainly studied case of the previous string operation is the Chas-Sullivan product. We let $X$ be the disjoint union $S^1\sqcup S^1$, and we let $\psi$ send the two points of $\del I$ to two points lying on the two different copies of $S^1$. We then have $Y\simeq S^1\vee S^1$, and we obtain a map $H_*(LM\times LM;R)\to H_{*-d}(M^{S^1\vee S^1};R)$. We can further compose this with the type (1) operation $H_{*-d}(M^{S^1\vee S^1};R)\to H_{*-d}(LM;R)$ induced by the pinch map $S^1\to S^1\vee S^1$; the resulting composite $H_*(LM\times LM;R)\to H_{*-d}(LM;R)$ is the Chas-Sullivan product \cite{ChasSullivan}.

To interpret the above construction in terms of parametrised $R$-modules,
define the topological spaces $\del U:=U\setminus\Delta(M)$ and $\del V:=V\setminus \tilde\Delta(M^{X/\del I})=e^{-1}(\del U)$, and
consider the following commutative cube in $\cS$, whose top and bottom face are pushouts, and whose vertical faces are pullbacks:
\[
\begin{tikzcd}[row sep=13pt]
\del V\ar[rr,"\tilde a"]\ar[dr,"\tilde b"]\ar[dd,"\del e"]& &V\simeq M^Y\ar[dr,"\tilde\Delta"]\ar[dd,"e_1"near end]\\
&M^X\setminus V\ar[rr,"\tilde c" near start]\ar[dd,"e_2"near start]& &M^X\ar[dd,"e"]\\
\del U\ar[rr,"a"near end]\ar[dr,"b"]& &U\simeq M\ar[dr,"\Delta"]\\
&M^2\setminus U\ar[rr,"c"] & &M^2.
\end{tikzcd}
\]
Then the above composite giving the general Chas--Sullivan string operation at the level of homology can be identified with the evaluation at $R$ of the natural transformation obtained from the top composite in the following diagram, after precomposition by $p(M^X)^*$ and postcomposition by $(p^2)_!$:
\[
\begin{tikzcd}
e_!\ar[r]\ar[d]&e_!\cof(\epsilon_{\tilde c_!})&e_!\tilde\Delta_!\cof(\epsilon_{\tilde a_!})\tilde\Delta^*\ar[l,"\simeq"]&e_!\tilde\Delta_!\tilde\Delta^*[d]\ar[l,"TI"',"\simeq"]\\
\cof(\epsilon_{c_!})e_!\ar[ur,"\simeq"',"BC"]&\Delta_!\cof(\epsilon_{a_!})\Delta^*e_!\ar[l,"\simeq"]\ar[ur,"\simeq"',"BC"]&\Delta_!\Delta^*e_![d].\ar[l,"TI"',"\simeq"]\ar[ur,"\simeq"',"BC"]
\end{tikzcd}
\]
In the previous diagram we use Notation \ref{nota:fibcofalpha}, and the arrows labeled ``$TI$'' are the Thom isomorphisms. Concretely, the projection formula allows us to identify the endofunctor $\cof(\epsilon_{a_!})\colon\Mod_R^M\to\Mod_R^M$ with the endofunctor $-\otimes\cof(\epsilon_{a_!})(p^*(R))$; we may moreover identify $\cof(\epsilon_{a_!})(p^*(R))\simeq n_\Delta M\otimes_\bS R$; and as we saw above, the $R$-orientation on $M$ identifies the latter endofunctor with the $d$-fold suspension. The commutativity of the previous diagram allows us to identify the general Chas--Sullivan operation with the evaluation at $R$ of the top composite natural transformation in the following diagram, after precomposition by $p(M^X)^*\simeq e^*(p^2)^*$:
\[
\begin{tikzcd}[column sep=18pt]
p^2_!e_!\ar[r]\ar[d,"\eta_{\Delta_*}"]&p^2_!\cof(\epsilon_{c_!})e_!&p_!\cof(\epsilon_{a_!})\Delta^*e_!\ar[l,"\simeq"]&p_!\Delta^*e_![d]\ar[l,"TI"',"\simeq"]\ar[r,"\simeq"',"BC"]&p_!(e_1)_!\tilde\Delta^*[d]\\
p^2_!\Delta_*\Delta^*e_!&p^2_*e_![2d]\ar[ul,"PD"',"\simeq"]\ar[r,"\eta_{\Delta_*}"]&p_*\Delta^*e_![2d]\ar[u,"PL"',"\simeq"]\ar[ur,"PD","\simeq"']\ar[ll,bend left=20,"PD"',"\simeq"]
\end{tikzcd}
\]
In the last diagram we use that $(U,\del U)$ is an $R$-Poincar\'e duality pair, obtained as the total space of a fibration pair $(U,\del U)\to M$ with base an $R$-Poincar\'e duality space, and fibre an $R$-Poincar\'e duality pair (in fact, a pair of the form $(*,S^{d-1})$); we have denoted by ``$PL$'' the Poincar\'e--Lefschetz duality isomorphism. The middle pentagon commutes by \cite[Corollary 4.5.2]{BHKK}, whereas the right triangle commutes by \cite[Theorem 6.4.3, Example 6.5.1]{BHKK}.

We thus obtain that the general Chas--Sullivan operation is the evalutation at $R$ of the following composite transformation of endofunctors of $\Mod_R$:
\[
\begin{split}
&p(M^X)_!p(M^X)^*\simeq p^2_!e_!e^*(p^2)^*\xrightarrow{\eta_{\Delta_*}}p^2_!\Delta_*\Delta^*e_!e^*(p^2)^*\overset{PD}{\simeq}p^2_*\Delta_*\Delta^*e_!e^*(p^2)^*[2d]\simeq\\
&p_*\Delta^*e_!e^*(p^2)^*[2d]\overset{PD}{\simeq}p_!\Delta^*e_!e^*(p^2)^*[d]\overset{BC}{\simeq}p_!(e_1)_!\tilde\Delta^*e^*(p^2)^*\simeq p(M^Y)_!p(M^Y)^*.
\end{split}
\]
\subsection{Recovering the operation (3)}
Given a decomposition $Y\simeq X\sqcup_{\ul2}*\in\cS$, we obtain the graph cobordism corresponding to the following diagram of spaces:
\[
\begin{tikzcd}[row sep=9pt]
\emptyset\ar[d]\ar[r]&*\ar[d,"\phi"']&\ul2\ar[l]\ar[d,"\psi"]\ar[dl,phantom,"\urcorner"very near end]\\
Y\ar[r,equal]&Y&X.\ar[l,"i"']
\end{tikzcd}
\]
The image of the previous along $Z_3$ is the following diagram in $\SPrL_R$, in which the left square gives the 1-morphism $\fe_\hPsi$ in 
$((\SPrL_R)^\llax_{\Mod_R/})^\twosimeq$ from Subsection \ref{subsec:unitalenhancement}; we keep denoting $e=M^\psi$ and $e_1=M^\phi$:
\[
\begin{tikzcd}[column sep=100pt,row sep=30pt]
\Mod_R\ar[d,equal]\ar[r,equal]&\Mod_R\ar[d,"(1^2\, 
p^2)_!"'pos=.5,"(1^2\,\Delta)_*\bar\Delta_! 
p^*\simeq"'pos=.8,"{\Delta_!p^*[d]}"pos=.8]\ar[r,"p(M^Y)_!p(M^Y)^*"]&\Mod_R\ar[d,"{[d]}"]\\
\Mod_R\ar[r,"(p^2)^*"']\ar[ur,Rightarrow,"\eta_{\Delta_*}","BC"'near start]&\Mod_R^{M^2}\ar[r,"p(M^X)_!e^*"']&\Mod_R.
\end{tikzcd}
\]
The diagram yields the following composite transformation of endofunctors of $\Mod_R$, whose evaluation at $R$ gives the image of the above graph cobordism along $\GFT_M$:
\[
\begin{split}
p(M^X)_!p(M^X)^*&\simeq p(M^X)_!e^*(1^2\,p^2)_!\delta_!(p^2)^*\xrightarrow{\eta_{\Delta_*}} p(M^X)_!e^*(1^2\,p^2)_!\delta_!\Delta_*\Delta^*(p^2)^*\\
&\xrightarrow{BC} p(M^X)_!e^*(1^2\,p^2)_!(1^2\,\Delta)_*\bar\Delta_!\Delta^*(p^2)^*\\
&\overset{2PD}{\simeq} p(M^X)_!e^*(1^2\,p)_!\bar\Delta_!p^*[d]\simeq p(M^X)_!e^*\Delta_!\Delta^*(p^2)^*[d]\\
&\overset{BC}{\simeq} p(M^X)_!\tilde\Delta_!(e_1)^*\Delta^*(p^2)^*[d]\simeq p(M^Y)_!p(M^Y)^*[d].
\end{split}
\]
In the previous composite we have labeled ``$2PD$'' the composite of two Poincar\'e duality equivalences.
We next observe that the composite
\[
\delta_!\xrightarrow{\eta_{\Delta_*}}\delta_!\Delta_*\Delta^*\xrightarrow{BC}(1^2\,\Delta)_*\bar\Delta_!\Delta^*
\]
agrees with the composite $\delta_!\xrightarrow{\eta_{(1^2\,\Delta)_*}}
(1^2\,\Delta)_*(1^2\,\Delta)^*\delta_!\overset{BC}{\simeq}(1^2\,\Delta)_*\bar\Delta_!\Delta^*$; using this, we can simplify the previous transformation as the following one:
\[
\begin{split}
&p(M^X)_!p(M^X)^*\simeq (p^2)_!e_!e^*(1^2p^2)_!\delta_!(p^2)^*\xrightarrow{\eta_{(1^2\Delta)_*}}\\
&(p^2)_!e_!e^*(1^2p^2)_!(1^2\Delta)_*(1^2\Delta)^*\delta_!(p^2)^*\overset{2PD}{\simeq} (p^2)_!e_!e^*(1^2p)_!(1^2\Delta)^*\delta_!(p^2)^*[d]\\
&\overset{BC}{\simeq}(p^2)_!e_!e^*(1^2p)_!\bar\Delta_!\Delta^*(p^2)^*[d]\simeq(p^2)_!e_!e^*\Delta_!\Delta^*(p^2)^*[d]\\
&\overset{BC}{\simeq} (p^2)_!e_!\tilde\Delta_!(e_1)^*\Delta^*(p^2)^*[d]\simeq p(M^Y)_!p(M^Y)^*[d].
\end{split}
\]
We next observe that the endofunctor $e_!e^*\colon\Mod_R^{M^2}\to\Mod_R^{M^2}$ is equivalent by the projection formula to the endofunctor $-\otimes e_!p(M^X)^*(R)$, and in particular it commutes with any $R$-linear natural transformation between $R$-linear functors. This allows us to rewrite the previous as the following composite:
\[
\begin{split}
&p(M^X)_!p(M^X)^*\simeq (p^2)_!(1^2p^2)_!\delta_!e_!e^*(p^2)^*\xrightarrow{\eta_{(1^2\Delta)_*}}\\
&(p^2)_!(1^2p^2)_!(1^2\Delta)_*(1^2\Delta)^*\delta_!e_!e^*(p^2)^*\overset{2PD}{\simeq} (p^2)_!(1^2p)_!(1^2\Delta)^*\delta_!e_!e^*(p^2)^*[d]\\
&\overset{BC}{\simeq}(p^2)_!(1^2p)_!\bar\Delta_!\Delta^*e_!e^*(p^2)^*[d]\simeq p_!\Delta^*e_!e^*(p^2)^*[d]\\
&\overset{BC}{\simeq} p_!(e_1)_!\tilde\Delta^*e^*(p^2)^*[d]\simeq p(M^Y)_!p(M^Y)^*[d].
\end{split}
\]
Finally, we observe that the functor $(p^2)^*$ can be exchanged with all of the functors $(1^2p^2)_!$, $(1^2\Delta)_*$ etc., by symmetric monoidality of $\SPrL_R$. 
We may therefore replace the previous with the following composite
\[
\begin{split}
&p(M^X)_!p(M^X)^*\simeq p^2_!(p^21^2)_!\delta_!e_!e^*(p^2)^*\xrightarrow{\eta_{\Delta_*}}\\
&p^2_!\Delta_*\Delta^*(p^21^2)_!\delta_!e_!e^*(p^2)^*\overset{2PD}{\simeq} p_!\Delta^*(p^21^2)_!\delta_!e_!e^*(p^2)^*[d]\\
&\overset{BC}{\simeq}p_!(p^21)_!\bar\Delta_!\Delta^*e_!e^*(p^2)^*[d]\simeq p_!\Delta^*e_!e^*(p^2)^*[d]\\
&\overset{BC}{\simeq} p_!(e_1)_!\tilde\Delta^*e^*(p^2)^*[d]\simeq p(M^Y)_!p(M^Y)^*[d].
\end{split}
\]
which agrees with the one given at the end of Subsection \ref{subsec:CSclassical} by the equivalences $(1^2p^2)\delta\simeq1^2$ and $p(p^21)\Delta\simeq p$.

\section{Morphism spaces in \texorpdfstring{$\GrCob$}{GrCob}}
\label{sec:GrCobspaces}
In this final section we analyse the morphism spaces in $\GrCob$, and prove in particular Theorem \ref{thm:B}. Throughout this section we abbreviate $\cC=\Gr_\cS$.
\subsection{A colimit formula}
Recall Definition \ref{defn:GrCob}.
In this subsection we denote by $W\subset\cC$ the wide subcategory spanned by idle morphisms, so that we have $\GrCob:=L(\cC,W)$ according to the notation from Subsection \ref{subsec:localisation}.

As already observed in Remark \ref{rem:idleshape}, any object in $\GrCob$ is equivalent to the image of an object in $\cC$ of the form $(\emptyset\to X)$ along the localisation functor.
\begin{nota}
For $X\in\cS$ we abbreviate $\Fin_{/X}:=\Fin\times_\cS\cS_{/X}$.
\end{nota}

\begin{nota}
\label{nota:colimFinX}
We fix an object $x=(\emptyset\to X)\in\cC$ for the rest of the subsection; we let $W(x)$ denote the full subcategory of $\cC_{x/}$ spanned by objects that are morphisms in $W$ with source $x$, and we let $\pi(x)\colon W(x)\to\cC_{x/}$ denote the inclusion. 
\end{nota}

\begin{rem}
\label{rem:2of3}
Since $W$ is defined as the class of morphisms in $\cC$ that are inverted along $D_0\colon\cC\to\Cospan(\cS)$, we have that $W$ satisfies the 2-of-3 property: given morphisms $f\colon x\to y$ and $g\colon y\to z$ in $\cC$, we have that if any two of $f,g,gf$ belong to $W$, so does the third.
As a consequence, we have an equivalence $W(x)\simeq W_{x/}$.

By Remark \ref{rem:idleshape}, and in the light of Lemma \ref{lem:rpo}, we may identify $W$ with the $\infty$-subcategory of $\Cospan(\Fun([1],\cS))$ spanned by objects whose image along $D_1$ is a finite set, and by morphisms that are sent to maps of finite sets along $D_1$, and to equivalences along $D_0$. In other words, we have an identification $W\simeq\Fin\times_\cS\Fun([1],\cS)\times_\cS\cS^\simeq$. Similarly, we have an identification $W_{x/}\simeq\Fin_{/X}$.
\end{rem}
The goal of this subsection is to prove the following proposition.
\begin{prop}
\label{prop:colimFinX}
The data $\pi(x)\colon W(x)\to\cC_{x/}$ introduced in Notation \ref{nota:colimFinX} give a left calculus of fractions at $x$ for $(\cC,W)$ in the sense of Definition \ref{defn:leftcalculus}. 
\end{prop}
\begin{defn}
Given a graph cobordism between finite sets $B\to G\ot A$, we say that a component of $G$ is a \emph{based tree} if it is contractible and its preimage in $A$ consists of a unique point. Compare the terminology with \cite[Definition 3.5]{Bianchi:graphcobset}.
\end{defn}
\begin{defn}
For $y=(B\to Y)\in\cC$, we represent an object in $\cC_{y/}\times_\cC\Fin_{/X}$ by a diagram as in Notation \ref{nota:morGrS}.
We say that an object is \emph{reduced} if every based tree in $G$ has preimage in $B$ consisting of precisely one element.
We let $(\cC_{y/}\times_\cC\Fin_{/X})_\red$ denote the full subcategory of $\cC_{y/}\times_\cC\Fin_{/X}$ spanned by reduced objects.
\end{defn}

\begin{lem}
\label{lem:redrightadjoint}
For an object $y=(B\to Y)\in\cC$, the inclusion $(\cC_{y/}\times_\cC\Fin_{/X})_\red\hto\cC_{y/}\times_\cC\Fin_{/X}$ admits a right adjoint, which we denote
\[
\red_y\colon \cC_{y/}\times_\cC\Fin_{/X}\to(\cC_{y/}\times_\cC\Fin_{/X})_\red.
\]
\end{lem}
\begin{proof}
We abbreviate $\cD=\cC_{y/}\times_\cC\Fin_{/X}$.
It suffices to construct for any object $z\in\cD$ a ``candidate right adjoint object'' $\bar z\in\cD_\red$ together with a morphism $\epsilon\colon\bar z\to z$, and check that postcomposition by $\epsilon$ induces an equivalence of morphism spaces $\cD(\tilde z,\bar z)\xrightarrow{\simeq}\cD(\tilde z,z)$ for all $\tilde z\in\cD_\red$. 

We represent $z$ as a diagram as in Notation \ref{nota:morGrS}, and split $G=G_1\sqcup G_2$, where $G_1$ consists of all based trees in $G$. We similarly split $B=B_1\sqcup B_2$ and $A=A_1\sqcup A_2$, where $B_1,A_1$ are the preimages of $G_1$ in $B,A$; note the equivalence $A_1\simeq G_1$. We may thus represent $z$ by the following diagram:
\[
\begin{tikzcd}[row sep=10pt]
B=B_1\sqcup B_2\ar[r]\ar[d]&A_1\sqcup G_2\ar[d]&A_1\sqcup A_2\ar[l]\ar[d]\ar[dl,phantom,"\urcorner"very near end]\\
Y\ar[r]&W&X\ar[l]
\end{tikzcd}
\]
We define $\bar z$ and $\epsilon$ using the following diagram, in which all maps are restricted from the above diagram for $z$, except for the left and the middle top horizontal maps, whose restriction to $B_1$ is declared to be the identity of $B_1$:
\[
\begin{tikzcd}[row sep=15pt,column sep=50pt]
B=B_1\sqcup B_2\ar[r]\ar[d]&B_1\sqcup G_2\ar[d]&B_1\sqcup A_2\ar[l]\ar[d,"(B_1\to A_1\to X)"pos=.2,"\sqcup(A_2\to X)"pos=.8]\ar[dl,phantom,"\urcorner"very near end]\ar[r,dashed,"(B_1\to A_1)\sqcup\Id_{A_2}"]&A_1\sqcup A_2\ar[d]\\
Y\ar[r]&W&X\ar[l]\ar[r,dashed,equal]&X.
\end{tikzcd}
\]
The left and middle square provide $\bar z$, and the decompositions of the top objects as disjoins unions are analogous to those above for $z$, and witness that $\bar z$ is reduced.
The $\Fin_{/X}$-coordinate of the morphism $\epsilon$ is given by the map $(B_1\to A_1)\sqcup\Id_{A_2}$ of finite sets over $X$, i.e. by the dashed square above. If we consider this as a morphism in $\Cospan(\Fun([1],\cS))$ and compose it with $\bar z$, the resulting cospan of arrows of spaces admits an evident identification with $z$, upgrading the above dashed square to a morphism $\epsilon$ in $\cD$ as desired.

We next fix $\tilde z\in\cD_\red$, and represent it as a diagram as in Notation \ref{nota:morGrS} in which we decorate with a tilde the data which is specific of $\tilde z$, and we decompose $B$, $\tilde G$ and $\tilde A$ according to based trees in $\tilde G$ just as above:
\[
\begin{tikzcd}[row sep=10pt,column sep=50pt]
B=\tilde B_1\sqcup \tilde B_2\ar[r]\ar[d]&\tilde G=\tilde B_1\sqcup \tilde G_2\ar[d]&\tilde A=\tilde B_1\sqcup \tilde A_2\ar[l]\ar[d]\ar[dl,phantom,"\urcorner"very near end]\\
Y\ar[r]&\tilde W&X\ar[l].
\end{tikzcd}
\]
The datum of a morphism $f\colon \tilde z\to z$ in $\cD$ is then the same as a natural transformation of commutative diagrams of spaces as follows, in which the restricted natural transformation on top gives $\tilde z$, and the one on bottom gives $z$:
\[
\begin{tikzcd}[column sep=12pt,row sep=10pt]
B=\tilde B_1\sqcup\tilde B_2\ar[r]\ar[d,equal]&\tilde G=\tilde B_1\sqcup\tilde G_2\ar[d,"f_G"']&\tilde A=\tilde B_1\sqcup\tilde A_2\ar[d]\ar[l]\ar[dl,phantom,"\urcorner"very near end]\\
B=B_1\sqcup B_2\ar[r]&G=A_1\sqcup G_2&A=A_1\sqcup A_2\ar[l]
\end{tikzcd}
\hspace{.5cm}\Rightarrow\hspace{.5cm}
\begin{tikzcd}[column sep=8pt,row sep=10pt]
Y\ar[r]\ar[d,equal]&\tilde W\ar[d,"\simeq"']&X\ar[l]\ar[dl,phantom,"\urcorner"very near end]\ar[d,equal]\\
Y\ar[r]&W&X.\ar[l]
\end{tikzcd}
\]
We now note that $\tilde G_2$ must be mapped to $G_2$ along $f_G$. To see this, let $\tilde\Gamma$ be a component of $\tilde G_2$, and let $\tilde\gamma$ be its preimage in $\tilde A_2$; then either $\tilde\gamma=\emptyset$, or $\chi(\tilde\Gamma,\tilde\gamma)<0$, because $\tilde\Gamma$ is not a based tree; the image of $\tilde\Gamma$ in $G$ will have the analogous property.

It follows that along the decompositions $B=\tilde B_1\sqcup\tilde B_2=B_1\sqcup B_2$ we have a containment $\tilde B_2\subseteq B_2$ and hence also $B_1\subseteq\tilde B_1$. From this, it follows that the above morphism $f$ factors essentially uniquely as a vertical composite as follows:
\[
\begin{tikzcd}[column sep=12pt,row sep=10pt]
B=\tilde B_1\sqcup\tilde B_2\ar[r]\ar[d,equal]&\tilde G=\tilde B_1\sqcup\tilde G_2\ar[d,"f_G"']&\tilde A=\tilde B_1\sqcup\tilde A_2\ar[d]\ar[l]\ar[dl,phantom,"\urcorner"very near end]\\
B=B_1\sqcup B_2\ar[d,equal]\ar[r]&\bar G=B_1\sqcup G_2\ar[d]&\bar A=B_1\sqcup A_2\ar[l]\ar[d]\ar[dl,phantom,"\urcorner"very near end]\\
B=B_1\sqcup B_2\ar[r]&G=A_1\sqcup G_2&A=A_1\sqcup A_2\ar[l]
\end{tikzcd}
\hspace{.5cm}\Rightarrow\hspace{.5cm}
\begin{tikzcd}[column sep=8pt,row sep=10pt]
Y\ar[r]\ar[d,equal]&\tilde W\ar[d,"\simeq"']&X\ar[l]\ar[dl,phantom,"\urcorner"very near end]\ar[d,equal]\\
Y\ar[r]\ar[d,equal]&W\ar[d,equal]&X\ar[d,equal]\ar[l]\ar[dl,phantom,"\urcorner"very near end]\\
Y\ar[r]&W&X.\ar[l]
\end{tikzcd}
\]
\end{proof}

\begin{proof}[Proof of Proposition \ref{prop:colimFinX}]
The first two requirements from Definition \ref{defn:leftcalculus} are immediate, so we shall focus on the third.
Given an idle morphism $y'\to y$ in $\cC$, we want to show that the induced functor $\cC_{y/}\times_\cC\Fin_{/X}\to\cC_{y'/}\times_\cC\Fin_{/X}$ gives an equivalence on classifying spaces.
Letting $y=(B\to Y)$ and $y'=(B'\to Y')$, the idle morphism induces an equivalence $Y'\simeq Y$ and can be regarded as a morphism $B'\to B$ in $\Fin_{/Y}$. Letting $y_\emptyset:=(\emptyset\to Y)$, we obtain a composition of idle morphisms $y_\emptyset\to y'\to y$, so that it suffices to prove that the two functors with target $\cC_{y_\emptyset/}\times_\cC\Fin_{/X}$ and with sources $\cC_{y/}\times_\cC\Fin_{/X}$ and $\cC_{y'/}\times_\cC\Fin_{/X}$, respectively, induce equivalences on classifying spaces. We thus reduce to checking that for any $y=(B\to Y)\in\Gr_\cS$, the idle morphism $y_\emptyset\to y$ induces an equivalence of spaces
\[
|\cC_{y/}\times_\cC\Fin_{/X}|\xrightarrow{\simeq}|\cC_{y_\emptyset/}\times_\cC\Fin_{/X}|.
\]
This is immediate if $B=\emptyset$; an evident induction argument allows us to reduce the above to checking instead that any idle morphism $y'\to y$ induces an equivalence of spaces $|\cC_{y/}\times_\cC\Fin_{/X}|\xrightarrow{\simeq}|\cC_{y'/}\times_\cC\Fin_{/X}|$, under the following assumptions on $y$ and $y'$: $y$ has the form $(B\to Y)$ for a nonempty finite set $B$, $B'=B\setminus\set{b}$ for some $b\in B$, $y'=(B'\to Y)$, and the idle morphism $y'\to y$ is given by the inclusion $B'\subset B$, considered as a map of spaces over $Y$.

Let $\cD=\cC_{y/}\times_\cC\Fin_{/X}$ and $\cD'=\cC_{y'/}\times_\cC\Fin_{/X}$.
By Lemma \ref{lem:redrightadjoint} the left and right arrow in the following composite functor induce equivalences on classifying spaces:
\[
\cD_\red\hto\cD\xrightarrow{\mathrm{idle}}\cD'\xrightarrow{\red_{y'}}\cD'_\red.
\]
To prove that the middle functor induces an equivalence on classifying spaces, it thus suffices to check that the composite does.
We now observe that the composite $\cD_\red\to\cD'_\red$ is a cocartesian fibration. For this, let $z\in\cD_\red$ be an object as in Notation \ref{nota:morGrS}, let $z'$ be its image in $\cD'$, and let $f'\colon z'\to \bar z'$ be a morphism in $\cD'$; we represent $f'$ as the following natural transformation:
\[
\begin{tikzcd}[column sep=8pt,row sep=10pt]
B'\ar[r]\ar[d,equal]&G'\ar[d]&A'\ar[d]\ar[l]\ar[dl,phantom,"\urcorner"very near end]\\
B'\ar[r]&\bar G'&\bar A'\ar[l]
\end{tikzcd}
\hspace{.5cm}\Rightarrow\hspace{.5cm}
\begin{tikzcd}[column sep=8pt,row sep=10pt]
Y\ar[r]\ar[d,equal]& W'\ar[d,"\simeq"']&X\ar[l]\ar[dl,phantom,"\urcorner"very near end]\ar[d,equal]\\
Y\ar[r]&\bar W'&X.\ar[l]
\end{tikzcd}
\]
There are now two cases for the descrption of $z'$ in terms of $z$.
\begin{enumerate}
\item If the map $B\to G$ sends $b$ to a based tree, then $z'$ is obtained from $z$ by discarding the mentioned based tree (together with its preimage in $A$ and its preimage $b\in B$); in particular we have $A\simeq A'\sqcup\set{b}$ and $G\simeq G'\sqcup\set{b}$; in this case we claim that the morphism $f\colon z\to\bar z$, given by the following natural transformation, is a cocartesian lift of $f'$:
\[
\begin{tikzcd}[column sep=8pt,row sep=10pt]
B=B'\sqcup\set{b}\ar[r]\ar[d,equal]&G=G'\sqcup\set{b}\ar[d]&A=A'\sqcup\set{b}\ar[d]\ar[l]\ar[dl,phantom,"\urcorner"very near end]\\
B=B'\sqcup\set{b}\ar[r]&\bar G:=\bar G'\sqcup\set{b}&\bar A:=\bar A'\sqcup\set{b}\ar[l]
\end{tikzcd}
\hspace{.4cm}\Rightarrow\hspace{.4cm}
\begin{tikzcd}[column sep=8pt,row sep=10pt]
Y\ar[r]\ar[d,equal]& W\simeq W'\ar[d,"\simeq"']&X\ar[l]\ar[dl,phantom,"\urcorner"very near end]\ar[d,equal]\\
Y\ar[r]&\bar W:=\bar W'&X.\ar[l]
\end{tikzcd}
\]
\item If instead the map $B\to G$ does not send $b$ to a based tree, then $z'$ is obtained from $z$ by discarding $b\in B$; in particular we have $A\simeq A'$ and $G\simeq G'$; in this case we claim that the morphism $f\colon z\to \bar z$, given by the following natural transformation, is a cocartesian lift of $f'$:
\[
\begin{tikzcd}[column sep=8pt,row sep=10pt]
B=B'\sqcup\set{b}\ar[r]\ar[d,equal]&G=G'\ar[d]&A=A'\ar[d]\ar[l]\ar[dl,phantom,"\urcorner"very near end]\\
B=B'\sqcup\set{b}\ar[r]&\bar G:=\bar G'&\bar A:=\bar A'\ar[l]
\end{tikzcd}
\hspace{.4cm}\Rightarrow\hspace{.4cm}
\begin{tikzcd}[column sep=8pt,row sep=10pt]
Y\ar[r]\ar[d,equal]& W\simeq W'\ar[d,"\simeq"']&X\ar[l]\ar[dl,phantom,"\urcorner"very near end]\ar[d,equal]\\
Y\ar[r]&\bar W:=\bar W'&X.\ar[l]
\end{tikzcd}
\]
\end{enumerate}
Now let $f''\colon z''\to\tilde z''$ be a morphism in $\cD'_\red$, and represent without loss of generality $\tilde z''$ by the following diagram on left, and $f''$ by the following diagram on right:
\[
\begin{tikzcd}[row sep=10pt]
B'\ar[d]\ar[r]&\tilde G''\ar[d]&\tilde A''\ar[l]\ar[d]\ar[dl,phantom,"\urcorner"very near end]\\
Y\ar[r]&W&X;\ar[l]
\end{tikzcd}
\hspace{1cm}
\begin{tikzcd}[row sep=10pt]
B'\ar[r]\ar[d]&G''\ar[d]&A''\ar[l]\ar[d]\ar[dl,phantom,"\urcorner"very near end]\ar[r,dashed]&\tilde A''\ar[d]\\
Y\ar[r]&W&X\ar[l]\ar[r,dashed,equal]&X.
\end{tikzcd}
\]
The proof of both claims is a routine exercise: for any $\tilde z\in\cD_\red$ with projection $\tilde z'\in\cD'_\red$ one can directly check the equivalence of morphism spaces $\cD_\red(\bar z,\tilde z)\simeq\cD_\red(z,\tilde z)\times_{\cD'_\red(z',\tilde z')}\cD'_\red(\bar z',\tilde z')$; in case (2) one has in fact to distinguish two subcases, whether the map $B\to\tilde G$ sends $b$ to a based tree or not.

To prove that the cocartesian fibration $\cD_\red\to\cD'_\red$ induces an equivalence on classifying spaces, it suffices now to show that each fibre has contractible classifying space. Let therefore $z'\in\cD'_\red$ be represented by the following diagram, in which as usual $B'_1$ is the set of based trees in  $G'$:
\[
\begin{tikzcd}[row sep=10pt,column sep=10pt]
B'_1\sqcup B'_2\ar[d]\ar[r]&B'_1\sqcup G'_2\ar[d]&B'_1\sqcup A'_2\ar[l]\ar[d]\ar[dl,phantom,"\urcorner"very near end]\\
Y\ar[r]&W&X;\ar[l]
\end{tikzcd}
\]
and let $\cD^{z'}_\red$ denote the fibre at $z'$ of the functor $\cD_\red\to\cD'_\red$.
We distinguish two types of objects $z\in\cD_\red^{z'}$, represented by the following diagrams; in few words, in type (1) we have $b\in B_1$, whereas in type (2) we have $b\in B_2$:
\[
(1)\hspace{.2cm}
\begin{tikzcd}[row sep=10pt,column sep=10pt]
B'_1\sqcup\set{b}\sqcup B'_2\ar[d]\ar[r]&B'_1\sqcup\set{b}\sqcup G'_2\ar[d]&B'_1\sqcup \set{b}\sqcup A'_2\ar[l]\ar[d]\ar[dl,phantom,"\urcorner"very near end]\\
Y\ar[r]&W&X;\ar[l]
\end{tikzcd}
\]
\[
(2)\hspace{.2cm}
\begin{tikzcd}[row sep=10pt,column sep=10pt]
B'_1\sqcup\set{b}\sqcup B'_2\ar[d]\ar[r]&B'_1\sqcup G'_2\ar[d]&B'_1\sqcup A'_2\ar[l]\ar[d]\ar[dl,phantom,"\urcorner"very near end]\\
Y\ar[r]&W&X;\ar[l]
\end{tikzcd}
\]
In type (1) we lift $z'$ to $z$ by adjoining a new based tree marked by $b$; in type (2) instead we use $b$ to mark a component of $G'$ directly, and in order for the result to be reduced, we must send $b$ to $G'_2$ in this case, i.e. not to a based tree of $G'$. The moduli space of objects of type (1) is therefore the space of lifts of the map $b\to W$ to a map $b\to X$, i.e. $\fib_b(X\to W)$, whereas the moduli space of objects of type (2) is the space of lifts of the map $b\to W$ to a map $b\to G'_2$, i.e. $\fib_b(G'_2\to W)$. All morphisms in $\cD_\red^{z'}$ are either equivalences or go from an objecct of type (1) to an object of type (2). More precisely, a morphism between the objects (1) and (2) represented by the diagrams above is the datum of:
\begin{itemize}
\item a retraction of finite sets $B'_1\sqcup\set{b}\sqcup A'_2\twoheadrightarrow B'_1\sqcup A'_2$, which is the same as a point $a\in B'_1\sqcup A'_2$;
\item a homotopy between the following two maps:
\begin{itemize}
\item the map $\set{b}\to B'_1\sqcup G'_2$, which is a datum of the object (2) and factors through $G'_2$;
\item the composite map $\set{b}\to B'_1\sqcup A'_2\to B'_1\sqcup G'_2$; 
\end{itemize}
in particular we must have $a\in A'_2$ for this homotopy to exist, and this homotopy is the same as a path in $G'_2$ from the image of $\set{a}\to G'_2$ to the image of the lift $\set{b}\to G'_2$, the latter being a datum of the object (2);
\item a homotopy between the following two maps:
\begin{itemize}
\item the map $\set{b}\to X$, which is a datum of the object (1);
\item the composite map $\set{b}\to A'_2\to X$;
\end{itemize}
this homotopy is the same as a path in $X$ from the image of $\set{a}\to X$ to the image of the lift $\set{b}\to X$.
\end{itemize}
The moduli space of all morphisms of $\cD_\red^{z'}$ going from type (1) to type (2), with arbitrary sources and targets, is therefore equivalent to the space of lifts of the map $b\to W$ to a map $b\to A'_2$, i.e. $\fib_b(A'_2\to W)$. All in all the classifying space $|\cD_\red^{z'}|$ is equivalent to the pushout of spaces $\fib_b(X\to W)\sqcup_{\fib_b(A'_2\to W)}\fib_b(G'_2\to W)$, which is equivalent to $\fib_b(X\sqcup_{A'_2}G'_2\xrightarrow{\simeq} W)$, which is contractible as desired.
\end{proof}

The following is a direct consequence of Proposition \ref{prop:colimFinX} and Theorem \ref{thm:Cisinski}.
\begin{cor}
\label{cor:colimFinX}
For $X,Y\in\cS$, the moduli space of graph cobordisms $\fM_\Gr(Y,X)$ is equivalent to the following classifying space of a left fibration over $\Fin_{/X}$, or equivalently, to the following colimit over $\Fin_{/X}$ of a diagram of spaces:
\[
\left|(\Gr_\cS)_{(\emptyset\to Y)/}\times_{\Gr_\cS}\Fin_{/X}\right|\simeq\colim_{(A\to X)\in\Fin_{/X}}\Gr_\cS((\emptyset\to Y),(A\to X)).
\]
\end{cor}
As consequences of Corollary \ref{cor:colimFinX}, we obtain the following further corollaries.
\begin{cor}
\label{cor:GrCobcore}
The functor $D_0\colon\GrCob\to\Cospan(\cS)$ induces an equivalence on core groupoids. In particular $\GrCob^\simeq\simeq\cS^\simeq$.
\end{cor}
\begin{proof}
Every object $X\in\Cospan(\cS)$ can be lifted to an object in $\GrCob$ along $D_0$, for instance $(\emptyset\to X)$: this shows that $D_0$ is essentially surjective. Given $Y,X\in\cS$, we consider the objects $x=(\emptyset\to X)$ and $y=(\emptyset\to Y)$ in $\cC:=\Gr_\cS$.
Our goal is to compute the subspace of $\fM_\Gr(Y,X)$ comprising morphisms $y\xrightarrow{\simeq} x$ that are equivalences in $\GrCob$. The equivalence $\left|\cC_{y/}\times_\cC W_{x/}\right|\simeq\fM_\Gr(Y,X)$ sends an object $y\to x'\ot x$ in the category $\cC_{y/}\times_\cC W_{x/}$ to the composite of the morphism $y\to x'$ and the formal inverse of the idle morphism $x\to x'$; if we expand without loss of generality $x'=(A\to X)$, this morphism in $\GrCob$ is further sent along $D_0$ to the morphism $D_0(y\to x')$.
A necessary condition for $y\to x'\ot x$ to be an invertible morphism in $\GrCob$ is therefore that $y\to x'$ be idle in $\cC$, and this is clearly also a sufficient condition. We thus obtain that the subspace of $\fM_\Gr(Y,X)$ comprising invertible morphisms in $\GrCob$ is the classifying space of the $\infty$-subcategory $W_{y/}\times_W W_{x/}\subseteq \cC_{y/}\times_\cC W_{x/}$. We may now identify $W_{y/}\simeq\Fin_{/Y}$, $W_{x/}\simeq\Fin_{/X}$, and $W\simeq \Fin\times_\cS\Fun([1],\cS)\times_\cS\cS^\simeq$ as in Remark \ref{rem:2of3}; this leads to the identification of $W_{y/}\times_W W_{x/}$ with the product of $\Fin_{/Y}$ and the space $\cS^\simeq(Y,X)$ of equivalences $Y\xrightarrow{\simeq}X$; using that $|\Fin_{/Y}|\simeq*$, we obtain the desired equivalence $|W_{y/}\times_WW_{x/}|\simeq\cS^\simeq(Y,X)$.
\end{proof}
\begin{cor}
\label{cor:cSinGrCob}
Consider $\cS$ as a wide $\infty$-subcategory of $\Cospan(\cS)$, spanned by cospans of the form $Y\to W\overset{\simeq}{\ot}X$; then $D_0$ restricts to an equivalence of $\infty$-categories
\[
D_0\colon D_0^{-1}(\cS)\xrightarrow{\simeq}\cS.
\]
\end{cor}
\begin{proof}
By Corollary \ref{cor:GrCobcore} we have that $D_0^{-1}(\cS)\to\cS$ is essentially surjective. To prove it is fully faithful, fix $Y,X\in\cS$ and let $y=(\emptyset\to Y)$ and $x=(\emptyset\to X)$ be the corresponding objects in $\cC:=\Gr_\cS$. The preimage of $\cS(Y,X)\subset\Cospan(\cS)(Y,X)$ along $D_0$ is the subspace of $\fM_\Gr(Y,X)$ obtained as classifying space of the full $\infty$-subcategory $(\cC_{y/}\times_\cC W_{x/})'\subset\cC_{y/}\times_\cC W_{x/}$ spanned by objects $y\to x'\ot x$ such that $D_0(y\to x')$ is a morphism in $\cS\subset\Cospan(\cS)$; assuming without loss of generality that $x'$ takes the form $(A\to X)$, the requirement is that the morphism $y\to x'$ in $\cC$ be of the form
\[
\begin{tikzcd}[row sep=10pt]
\emptyset\ar[r]\ar[d]&G\ar[d]&A\ar[l,"\simeq"]\ar[d]\ar[dl,phantom,"\urcorner"very near end]\\
Y\ar[r]&W&X.\ar[l,"\simeq"]
\end{tikzcd}
\]
By Lemma \ref{lem:redrightadjoint} we have an equivalence $|(\cC_{y/}\times_\cC W_{x/})'|\simeq|(\cC_{y/}\times_\cC W_{x/})'_\red|$, where the index ``$\red$'' takes the full $\infty$-subcategory spanned by reduced objects. We now observe that all components of $G$ in an object as above are based trees, so the condition of being reduced forces $G=A=\emptyset$. It follows that $(\cC_{y/}\times_\cC W_{x/})'_\red$ is already a space, and as such it is equivalent to $\cS(Y,X)$ as desired.
\end{proof}
\begin{cor}
\label{cor:GrinGrCob}
The full $\infty$-subcategory of $\GrCob$ spanned by finite sets is equivalent to $\Gr$.
\end{cor}
\begin{proof}
Denoting $\GrCob_\Fin$ such full $\infty$-subcategory, the functor $D_0\colon\GrCob\to\Cospan(\cS)$ restricts to a functor $\GrCob_\Fin\to\Gr\subset\Cospan(\cS)$. We claim that $D_0\colon \GrCob_\Fin\to\Gr$ is an equivalence; since it is essentially surjective, it suffices to prove that it induces equivalences on morphism spaces. For $X\in\Fin$, the $\infty$-category $\Fin_{/X}$ has a terminal object $\Id_X\colon X\to X$, whence for $Y\in\Fin$ the colimit from Corollary \ref{cor:colimFinX} is equivalent to the space $\Gr_\cS((\emptyset\to Y),(X=X))$, and the latter can be readily identified with $\Gr(Y,X)$, precisely along $D_0\colon\Gr_\cS\to\Cospan(\cS)$.
\end{proof}

\begin{rem}
\label{rem:wlogYempty}
Another immediate consequence of Corollary \ref{cor:colimFinX} is that for all spaces $Y,X\in\cS$ we have a composite pullback of spaces as in the following diagram, in which the left square is induced by the functor $D_0$ and by precomposition with the graph cobordism $\emptyset\to Y=Y$, whereas the right square is given by forgetting the structure maps from $X$ and from $Y$:
\[
\begin{tikzcd}[row sep=10pt]
\fM_\Gr(Y,X)\ar[d]\ar[r,"D_0"]\ar[dr,phantom,"\lrcorner"very near start]&(\cS_{Y\sqcup X/})^\simeq\ar[d]\ar[r]\ar[dr,phantom,"\lrcorner"very near start]&(\cS_{Y/})^\simeq\ar[d]\\
\fM_\Gr(\emptyset,X)\ar[r,"D_0"]&(\cS_{X/})^\simeq\ar[r]&\cS^\simeq.
\end{tikzcd}
\]
In particular, if the bottom map $D_0$ is a subspace inclusion, then so is the top map $D_0$. This will allow us to reduce to the case $Y=\emptyset$ in the proof of Proposition \ref{prop:GuirardelLevitt}.
\end{rem}

\begin{rem}
The colimit formula for $\fM_\Gr(Y,X)$ in Corollary \ref{cor:colimFinX} captures the essence of Definition \ref{defn:graphcob}, as we . Even if we do not attach a graph directly to $X$, but we still attach it to a finite set mapping to $X$, we let this finite set vary in the $\infty$-category $\Fin_{/X}$; as this is a filtered $\infty$-category, the actual choice of the finite set over $X$ along which we attach the graph should be expected to be inessential. On the other hand, the colimit formula for $\fM_\Gr(Y,X)$ involves no specific relative cell structure on a pair $(G,A)$, but just the mere existence of such a cell structure (which boils down to the property of $G$ of being equivalent to a finite graph); in the light of \cite[Theorem 8.4]{Bianchi:graphcobset}, if we understand relative cell structures on $(G,A)$ as a datum to be given up to tree collapse maps, then also the choice of relative cell structures should be expected to be inessential. We refrain from further manipulating the colimit formula for $\fM_\Gr(Y,X)$ in Corollary \ref{cor:colimFinX} for general $Y,X\in\cS$, though in the next subsection we will expand on the previous remarks when $Y=\emptyset$ and $X$ is aspherical.
\end{rem}

We next present an example showing that
the functor $D_0\colon\GrCob\to\Cospan(\cS)$ does not induce an inclusion at the level of all morphism spaces.
\begin{ex}
\label{ex:Hatcher}
Let $Y=\emptyset$ and let $X\in\cS$ be a simply connected space. The space $\fM_\Gr(Y,X)$ contains a subspace $\fM_\Gr(Y,X)_1$ comprising graph cobordisms as in Notation \ref{nota:morGrS} such that $W\simeq X\vee S^1$. We abbreviate $\cC=\Gr_\cS$ and let $y=(B\to Y)=(\emptyset\to\emptyset)\in\cC$. By Lemma \ref{lem:redrightadjoint} and Proposition \ref{prop:colimFinX} we may compute $\fM_\Gr(Y,X)_1$ as the classifying space of the full $\infty$-subcategory $\cD$ of $(\cC_{y/}\times_\cC\Fin_{/X})_\red$ spanned by reduced objects as in Notation \ref{nota:morGrS} such that $W\simeq X\vee S^1$. There are two types of objects in $\cD$, according to the size of $A$:
\begin{enumerate}
\item $A\simeq\ul2$ and $G\simeq*$; the moduli space of such objects is equivalent to $X^2/C_2$;
\item $A\simeq*$ and $G\simeq S^1$; the moduli space of such objects is equivalent to $X\times \bB C_2$, i.e. the quotient $X/C_2$ for the trivial $C_2$-action.
\end{enumerate}
All morphisms in $\cD$ are either equivalences or go from objects of type (1) to objects of type (2); the moduli space of morphisms of the latter type is also equivalent to $X\times \bB C_2$, with source map given by the diagonal map $X/C_2\to X^2/C_2$, and target given by the identity of $X/C_2$; we thus obtain an equivalence $\fM_\Gr(Y,X)_1\simeq X^2/C_2$. Passing to loop spaces, we obtain in particular that $\Omega\fM_\Gr(Y,X)_1$ is equivalent to $\coprod_2(\Omega X)^2$, i.e. the disjoint union of two copies of the cartesian square of $\Omega X$.

We next consider the analogous subspace $\Cospan(\cS)(Y,X)_1\subset\Cospan(\cS)(Y,X)$ spanned by cospans $Y\to W\ot X$ such that $W\simeq X\vee S^1$ under $X$; this is a connected subspace, and its loop space is equivalent to $\hAut(X\vee S^1,X)$, the space of homotopy automorphisms of $X\vee S^1$ restricting to the identity on $X$. The space of self maps of $X\vee S^1$ fixing $X$ is equivalent to $\Omega(X\vee S^1)$, and its subspace of homotopy automorphisms corresponds to the union $\Omega_{\pm1}(X\vee S^1)$ of two connected components of $\Omega(X\vee S^1)$, yielding the two possible generators of $\pi_1(X\vee S^1)\cong\Z$. 

The map of spaces $\fM_\Gr(Y,X)\to\Cospan(\cS)(Y,X)$ induced by $D_0$ restricts to a map of connected subspaces $\fM_\Gr(Y,X)_1\to\Cospan(\cS)(Y,X)_1$. Passing to loop spaces we obtain a map $f\colon \coprod_2(\Omega X)^2\to\Omega_{\pm1}(X\vee S^1)$; we have that $f$ is a map of $E_1$-algebras for suitable $E_1$-algebra structures on the source and target; since both spaces consist of two components and since the map is a bijection on $\pi_0$, it suffices to study whether $f$ is an equivalence on one component; we henceforth study the restricted map $f\colon(\Omega X)^2\to\Omega_1(X\vee S^1)$, which we consider as a plain map of spaces.

Let us now specify $X=S^{k+1}$ for some $k\ge1$; then $\Omega_1(S^{k+1}\vee S^1)\simeq\Omega(\bigvee_\Z S^{k+1})$ is the underlying space of the free $E_1$-algebra in spaces generated by the pointed space $\bigvee_\Z S^k$, with basepoint serving as unit; similarly, $\Omega S^{k+1}$ is the underlying space of the free $E_1$-algebra generated by the pointed sphere $S^k$. We stress that $f\colon(\Omega S^{k+1})^2\to\Omega(\bigvee_\Z S^{k+1})$ is not a map of $E_1$-algebras when considering these free $E_1$-algebra structures and products thereof; instead, $f$ agrees under suitable identifications with the following composite map of spaces:
\[
\begin{tikzcd}[column sep=50pt]
E_1[S^k]\times E_1[S^k]\ar[r,"i_0\times i_1"]&E_1[\bigvee_\Z S^k]\times E_1[\bigvee_\Z S^k]\ar[r,"(-)\cdot(-)^{-1}"]&E_1[\bigvee_\Z S^k],
\end{tikzcd}
\]
where $i_0,i_1\colon S^k\to\bigvee_\Z S^k$ are two ``consecutive'' inclusions, and $(-)\cdot(-)^{-1}$ denotes multiplication of the first element with the inverse of the second in the group-like $E_1$-algebra $E_1[\bigvee_\Z S^k]$. We then observe that $f$ factors through $E_1[ S^k\vee S^k]$, whose homology is much smaller than $E_1[\bigvee_\Z S^k]$: to wit, $H_k(E_1[ S^k\vee S^k])\cong\Z^2$ whereas $H_k(E_1[\bigvee_\Z S^k])\cong\bigoplus_\Z\Z$. We conclude that the map of spaces $\fM_\Gr(Y,X)\to\Cospan(\cS)(Y,X)$ induced by $D_0$ is not an equivalence for $Y=\emptyset$ and $X=S^{k+1}$.

We invite the interested reader to generalise the above argument to the case of an arbitrary space connected space $X$ which is not aspherical. 
\end{ex}
Example \ref{ex:Hatcher} also shows that the conjecture from \cite{HatcherBanff} is wrong in general: for a space $X$ which is not aspherical, the canonical map from the classifying space of the category $G_1(X)$ from loc.cit. to $\bB\hAut(X\vee S^1,X)$ is not an equivalence. The identification $|G_1(X)|\simeq\fM_\Gr(\emptyset,X)_1$ will be discussed in the next subsection.

\subsection{A formula via tree collapse maps}
In this subsection, for $X\in\cS$, we give an alternative formula for the space $\fM_\Gr(\emptyset,X)$, in terms of the classifying space of an $\infty$-category $\bGr(X)$ of combinatorial graphs attached to $X$, with morphisms given by tree collapse maps. If $X$ is an actual topological space, the $\infty$-category $\bGr(X)$ has been constructed as a category internal to topological spaces in \cite{HatcherBanff} under the name ``$G(X)$''. The main application of this alternative formula will be Proposition \ref{prop:GuirardelLevitt}, which will identify the space $\fM_\Gr(\emptyset,X)$ as a subspace of $(\cS_{X/})^\simeq$; this will crucially rely on work of Guirardel--Levitt \cite{GuirardelLevitt}.

The main input towards the definition of $\bGr(X)$ is a variation of the (plain) category $\Gaf$ of ``marked graphs attached to a finite set'' introduced in \cite[Definition 3.6]{Bianchi:graphcobset}: we simplify the description by removing the finite set ``$B$'' giving the marking, and by extending the setting to an equivariant one: this extra generality will be needed later in the proof of Proposition \ref{prop:GuirardelLevitt}. We fix therefore a discrete group $\Gamma$ in the following; when $\Gamma$ is the trivial group we recover the theory from loc.cit., specialised to the case in which the finite set ``$B$'' giving the marking is empty.

\begin{nota}
We denote by $\Gamma\Fin$ the full subcategory of $\Fun(\bB\Gamma,\mathrm{Set})$ spanned by sets  admitting a $\Gamma$-action with finitely many orbits; we refer to an object in $\Gamma\Fin$ as a \emph{finite $\Gamma$-set}, even though the underlying set need not be finite. We further denote by $\Gamma\Fin_\free$ the full subcategory of $\Gamma\Fin$ spanned by sets with a free $\Gamma$-action, which we refer to as \emph{finite free $\Gamma$-sets}.
We further abbreviate $\Gamma\cS:=\Fun(\bB\Gamma,\cS)$, and regard $\Gamma\Fin_\free\subset\Gamma\Fin$ as nested full $\infty$-subcategories of $\Gamma\cS$; we refer to an object in $\Gamma\cS$ as a \emph{$\Gamma$-space}.
\end{nota}

\begin{defn}
\label{defn:Gammagaf}
A \emph{$\Gamma$-graph attached to a finite $\Gamma$-set}, shortly $\Gamma$-gaf, is a sequence $G=(A,V,H,\sigma,\upsilon)$ consisting of a finite $\Gamma$-set $A$, two finite free $\Gamma$-sets $V$ and $H$, and two $\Gamma$-equivariant maps $\sigma,\upsilon\colon A\sqcup V\sqcup H\to A\sqcup V\sqcup H$, satisfying the following requirements:
\begin{itemize}
\item $\sigma$ restricts to a map $H\to A\sqcup V$;
\item $\upsilon$ restricts to an involution of $H$ without fixpoints, and the residual action of $\Gamma$ on the quotient $E:=H/\upsilon$ is again free;
\item $\sigma$ and $\upsilon$ restrict to the identity on each of the subsets $A$ and $V$.
\end{itemize}
We say that $G$ is a $\Gamma$-gaf attached to $A$. A $\Gamma$-subgaf of $G$ is obtained by selecting $\Gamma$-subsets $A',V',H'$ of $A,V,H$, respectively, such that $\sigma$ and $\upsilon$ restrict to self maps of $A'\sqcup V'\sqcup H'$.
A $\Gamma$-gaf $G$ as above is a \emph{based (respectively, unbased) $\Gamma$-tree} if $A$ consists of a single $\Gamma$-orbit (respectively, $A=\emptyset$) and if the pushout of $\Gamma$-spaces $E\sqcup_H(A\sqcup V)$ is $\Gamma$-equivariantly homotopy equivalent to the set $A$ along the inclusion (respectively, to the set $\Gamma$ along some $\Gamma$-equivariant map).

A morphisms of $\Gamma$-gafs $G=(A,V,H,\sigma,\upsilon)\to G'=(A',V',H',\sigma',\upsilon')$ is a $\Gamma$-equivariant map $f\colon A\sqcup V\sqcup H\to A'\sqcup V'\sqcup H'$ satisfying the following requirements:
\begin{itemize}
\item $f$ restricts to a map $A\to A'$;
\item $f\sigma=\sigma f$ and $f\upsilon=\upsilon f$;
\item for $v'\in V'$, $f^{-1}(\Gamma v')$ is an unbased $\Gamma$-subtree of $G$;
\item for $a'\in A'$, $f^{-1}(\Gamma a')$ is a disjoint union of based $\Gamma$-subtrees of $G$, one for each $\Gamma$-orbit in $f^{-1}(\Gamma a')\cap A$;
\item for $h'\in H'$, $f^{-1}(h')$ is a singleton contained in $H$.
\end{itemize}
We denote by $\Gaf^\Gamma_\emptyset$ the category of $\Gamma$-gafs an morphisms of $\Gamma$-gafs. The index ``$\emptyset$'' only serves to warn the reader that we are working without markings, differently as in \cite{Bianchi:graphcobset}. We further denote by $\fA^\Gamma$ and $\AVH^\Gamma$ the evident functors $\Gaf^\Gamma_\emptyset\to\Gamma\Fin$ sending $G\mapsto A$ and $G\mapsto A\sqcup V\sqcup H$, respectively, and by $\alpha^\Gamma\colon\fA\to\AVH$ the natural transformation evaluating at $G$ to the inclusion $A\hto A\sqcup V\sqcup H$. We let $\Gaf_\emptyset^{\Gamma,\free}$ denote the full subcategory of $\Gaf_\emptyset^\Gamma$ given by the preimage $(\fA^\Gamma)^{-1}(\Gamma\Fin_\free)$, and also denote by $\alpha^\Gamma\colon\fA\to\AVH$ the restricted natural transformation of functors $\Gaf_\emptyset^{\Gamma,\free}\to\Gamma\Fin_\free$.
When $\Gamma=1$ is the trivial group, we omit it from the notation, as well as the decoration ``$\free$''.
\end{defn}
We will mostly consider the subcategory $\Gaf^{\Gamma,\free}_\emptyset$.
Intuitively, an object in $\Gaf^{\Gamma,\free}_\emptyset$ is a finite graph $G$ with a free $\Gamma$-action and a partition of its set of vertices as $A\sqcup V$, where we consider $A$ as the ``attaching vertices'' and $V$ as the ``inner vertices''; the set $H$ of ``half-edges'' is endowed with an involution without fixpoints (the restriction of $\upsilon$), whose orbits form the set $E$ of the ``edges'' of $G$; the restriction of $\sigma$ gives the information about how to attach the half-edges to $A\sqcup V$. A morphism in $\Gaf^{\Gamma,\free}_\emptyset$ from $G$ to $G'$ is intuitively the datum of a map of finite free $\Gamma$-sets $A\to A'$ and a $\Gamma$-equivariant map of graphs $G\to G'$ exhibiting $G'$ as the result of collapsing disjoint trees in $G$ (each intersecting $A$ in at most one point) to distinct vertices, and of ``base-changing'' the attaching locus of the obtained graph from $A$ to $A'$, along a $\Gamma$-equivariant map $A\to A'$.

\begin{defn}
We denote by $\Re^\Gamma\colon\Gaf^\Gamma_\emptyset\to\Gamma\cS$ the functor sending a $\Gamma$-gaf $G$ as in Definition \ref{defn:Gammagaf} to the pushout of the following span diagram of $\Gamma$-spaces
\[
\begin{tikzcd}
A\sqcup V&A\sqcup V\sqcup H\ar[r,"(-)/\upsilon"]\ar[l,"\sigma"']&A\sqcup V\sqcup E.
\end{tikzcd}
\]
We denote by $\beta^\Gamma\colon\AVH^\Gamma\Rightarrow\Re^\Gamma$ the resulting transformation of functors $\Gaf^\Gamma_\emptyset\to\Gamma\cS$. We use the same notation for the restricted functors on $\Gaf^{\Gamma,\free}_\emptyset$.
\end{defn}
Concretely, we may think of $\Re^\Gamma(G)$ as a ``geometric realisation'' of the combinatorial datum of a $\Gamma$-gaf $G$; we then have a $\Gamma$-equivariant map $\beta^\Gamma\alpha^\Gamma(G)\colon A\to\Re(G)$, corresponding to the inclusion of the attaching vertices.

\begin{rem}
\label{rem:fAcocartesian}
When $\Gamma=1$, the functor $\fA\colon\Gaf_\emptyset\to\Fin$ is a cocartesian fibration, whose straightening is the functor $\bGr(\emptyset,-)\colon\Fin\to\Catinfty$ sending a finite set $A$ to the (plain) category $\bGr(\emptyset,A)$, which is the morphism category in the 2-category $\bGr$ from \cite[Subsection 3.2]{Bianchi:graphcobset}. For generic $\Gamma$ we also have that $\fA^\Gamma\colon\Gaf_\emptyset^\Gamma\to\Gamma\Fin$, as well as its restriction $\fA^\Gamma\colon\Gaf_\emptyset^{\Gamma,\free}\to\Gamma\Fin_\free$, are cocartesian fibrations of categories.
\end{rem}

\begin{defn}
\label{defn:bGrX}
For $X\in\Gamma\cS$ we define the $\infty$-category $\bGr^\Gamma(X)$ as the following pullback in $\Catinfty$, in which we abbreviate $(\Gamma\Fin_\free)_{/X}:=\Gamma\Fin_\free\times_{\Gamma\cS}\Gamma\cS_{/X}$:
\[
\begin{tikzcd}[row sep=10pt]
\bGr^\Gamma(X)\ar[r,"\tilde\fA^\Gamma"]\ar[d,"\tilde d_0"']\ar[dr,phantom,"\lrcorner"very near start]&(\Gamma\Fin_\free)_{/X}\ar[d,"d_0"]\\
\Gaf^{\Gamma,\free}_\emptyset\ar[r,"\fA^\Gamma"]&\Gamma\Fin_\free,
\end{tikzcd}
\]
We further denote by $\Re^\Gamma_X\colon\bGr^\Gamma(X)\to(\Gamma\cS_{X/})^\simeq$
the functor sending $(G,A\to X)$ to the pushout of $\Gamma$-spaces $\Re^\Gamma(G)\sqcup_AX$, where the $\Gamma$-equivariant map $A\to\Re(G)$ is $\beta^\Gamma\alpha^\Gamma(G)$. When $\Gamma=1$ we omit it from the notation.
\end{defn}
Intuitively, an object in $\bGr^\Gamma(X)$ is the datum of a graph $G$ with a free $\Gamma$-action admitting finitely many free orbits of edges and vertices, a $\Gamma$-subset $A$ of its vertices, and a map of $\Gamma$-spaces $A\to X$; a morphism is given by collapsing trees and changing the attaching subset of vertices in the $\infty$-category $(\Gamma\Fin_\free)_{/X}$. 
The functor $\Re^\Gamma_X$ constructs a $\Gamma$-space by attaching the realisation of $G$ to $X$ along $A$. The fact that morphisms in $\Gaf^\Gamma_\emptyset$ are essentially given by collapsing trees implies that $\Re^\Gamma_X$ takes values in the core groupoid of $\Gamma\cS_{X/}$, making Definition \ref{defn:bGrX} into a good definition.
\begin{rem}
We give a slight alternative to Definition \ref{defn:bGrX}. The $\Gamma$-space $X\in\Gamma\cS$ corresponds to the object $X/\Gamma$ in the overcategory $\cS_{/\bB\Gamma}$; we may then identify $\bGr^\Gamma(X)$ with the fibre product $\Gaf^{\Gamma,\free}_\emptyset\times_{\Fin_{/\bB\Gamma}}\Fin_{/X}$, where the first leg is the functor $\fA^\Gamma\colon\Gaf^{\Gamma,\free}_\emptyset\to\Gamma\Fin_\free\simeq\Fin_{/\bB\Gamma}$. 
\end{rem}

\begin{lem}
\label{lem:bGrX}
For $X\in\cS$ we have equivalences of spaces
$\fM_\Gr(\emptyset,X)\simeq|\bGr(X)|$.
\end{lem}
\begin{proof}
By Remark \ref{rem:fAcocartesian} we have that $\tilde\fA\colon\bGr(X)\to\Fin_{/X}$ is the pullback of a cocartesian fibration, so it is again a cocartesian fibration; moreover, the straightening of $\tilde\fA$ is the composite $\bGr(\emptyset,-)\circ d_0\colon\Fin_{/X}\to\Catinfty$. We now recall from \cite[Theorem 8.4]{Bianchi:graphcobset} that the composite functor $\Fin\xrightarrow{\bGr(\emptyset,-)}\Catinfty\xrightarrow{|-|}\cS$ can be identified with $\Gr(\emptyset,-)\colon\Fin\to\cS$. It follows that the localisation of $\bGr(X)$ at the ``vertical'' wide subcategory $\bGr(X)\times_{\Fin_{/X}}(\Fin_{/X})^\simeq$ is the left fibration over $\Fin_{/X}$ corresponding to the composite functor $\Gr(\emptyset,-)\circ d_0\colon\Fin_{/X}\to\cS$; the total $\infty$-category of this left fibration is $\cC_{y/}\times_\cC\Fin_{/X}$, where we abbreviate $\cC=\Gr_\cS$ and $y=(\emptyset\to\emptyset)$, and we also rely on Corollary \ref{cor:GrinGrCob}. We conclude by applying Proposition \ref{prop:colimFinX}.
\end{proof}
In the light of the equivalence from Lemma \ref{lem:bGrX}, the map of spaces $\fM_\Gr(\emptyset,X)\to(\cS_{X/})^\simeq$ induced by $D_0$ agrees with the map induced on classifying spaces by the functor $\Re_X\colon\bGr(X)\to(\cS_{X/})^\simeq$.

Our last goal is to prove the following proposition.
\begin{prop}
\label{prop:GuirardelLevitt}
Let $Y$ and $X$ be spaces with $X$ aspherical. Then the functor $D_0\colon\GrCob\to\Cospan(\cS)$ induces at the level of morphism spaces an inclusion of spaces $\fM_\Gr(Y,X)\hto(\cS_{Y\sqcup X/})^\simeq$, with essential image spanned by those cospans of spaces $Y\to W\ot X$ such that $W$ may be obtained from $X$ by attaching finitely many 0-cells and 1-cells.
\end{prop}
\begin{proof}
By Remark \ref{rem:wlogYempty} it suffices to consider the case $Y=\emptyset$. It is evident that a cospan of spaces $Y\to W\ot X$ is in the essential image if and only if the pair $(W,X)$ satisfies the given property. We may therefore fix $\bar W\in(\cS_{X/})^\simeq$ with the given property and aim at proving that the fibre at $\bar W$ of the map $\fM_\Gr(\emptyset,X)\to(\cS_{X/})^\simeq$ is contractible. 

If $\bar W$ is disconnected, write $\bar W=\coprod_{i\in I}\bar W_i$, and decompose $X=\coprod_{i\in I}\bar X_i$ by letting $\bar X_i$ be the preimage of $\bar W_i$ in $X$. Then we have an identification
\[
\fib_{\bar W}(\fM_\Gr(\emptyset,X)\to(\cS_{X/})^\simeq)\simeq\prod_{i\in I}\fib_{\bar W_i}(\fM_\Gr(\emptyset,\bar X_i)\to(\cS_{\bar X_i/})^\simeq),
\]
where all but finitely many factors in the product are contractible (so the above is essentially a finite product): this follows from the hypothesis that $\bar W$ can be obtained from $X$ by attaching \emph{finitely many} cells, and the fact that if $\bar X_i\xrightarrow{\simeq}\bar W_i$, then Corollary \ref{cor:cSinGrCob} implies that $\fib_{\bar W_i}(\fM_\Gr(\emptyset,\bar X_i)\to(\cS_{\bar X_i/})^\simeq)$ is contractible.

Thus it suffices to prove the claim under the additional hypothesis that $\bar W$ is connected. This also implies that $\pi_0(X)$ is finite, because it must be possible to make $X$ connected by attaching finitely many cells. Let therefore $X=\coprod_{i=1}^pX_i$, with $X_i\simeq B\Gamma_i$ for $p\ge0$ and for some discrete groups $\Gamma_i$. We stress that some of the $\Gamma_i$ may be the trivial group; if however this happens for all $0\le i\le p$, then the equivalence $\fib_{\bar W}(\fM_\Gr(\emptyset,X)\to(\cS_{X/})^\simeq)$ follows already from Corollary \ref{cor:GrinGrCob}. We may therefore assume  that $p\ge1$ and, without loss of generality, that $\Gamma_1$ is not the trivial group. Similarly, if $\chi(\bar W,X)=0$, then $X\xrightarrow{\simeq}\bar W$ and the equivalence 
$\fib_{\bar W}(\fM_\Gr(\emptyset,X)\to(\cS_{X/})^\simeq)$ follows already from Corollary \ref{cor:cSinGrCob}. We may therefore assume $\chi(\bar W,X)<0$: concretely this means that if $\bar W\simeq \bigvee_{i=1}^p X_i\vee\bigvee_{j=1}^kS^1$, then $p+k\ge2$. We now denote by $\Gamma$ the free product
\[
\Gamma:=\bigast_{i=1}^p\Gamma_i*\bigast_{j=1}^k\Z
\]
and fix an identification $(X\to\bar W)\simeq(\coprod_{i=1}^p\bB\Gamma_i\to\bB\Gamma)$ in $\Fun([1],\cS)$. We let $\tilde X:=X\times_{\bB\Gamma}*$; the equivalence $-\times_{\bB\Gamma}*\colon\cS_{/\bB\Gamma}\xrightarrow{\simeq}\Gamma\cS$ makes $\tilde X$ into an object in $\Gamma\cS$, and we observe that $\tilde X$ in fact belongs to $\Gamma\Fin\subset\Gamma\cS$; more precisely, the above data give a canonical identification $\tilde X\simeq\coprod_{i=1}^p\Gamma/\Gamma_i$.

By Lemma \ref{lem:bGrX} we may compute $\fib_{\bar W}(\fM_\Gr(\emptyset,X)\to(\cS_{X/})^\simeq)$ as the classifying space of the $\infty$-category $\bGr(X)_{\bB\Gamma}$ defined by the left pullback square in the following diagram in $\Catinfty$, or equivalently, as the total pullback:
\[
\begin{tikzcd}[row sep=10pt]
\bGr(X)_{\bB\Gamma}\ar[r]\ar[d]\ar[dr,phantom,"\lrcorner"very near start]&*\ar[d,"\bB\Gamma"']\ar[r,equal]\ar[dr,phantom,"\lrcorner"very near start]&*\ar[d,"\bB\Gamma"]\\
\bGr(X)\ar[r,"\Re_X"]&(\cS_{X/})^\simeq\ar[r,hook]&\cS_{X/}.
\end{tikzcd}
\]
We may regard $\bGr(X)_{\bB\Gamma}$ as a full $\infty$-subcategory of the lax pullback $\bGr(X)'_{\bB\Gamma}$ of the same total diagram above, i.e. the following pullback on left, using the under-overcategory $\cS_{X//\bB\Gamma}$ defined by the pullback on right:
\[
\begin{tikzcd}[row sep=10pt]
\bGr(X)'_{\bB\Gamma}\ar[r]\ar[d]\ar[dr,phantom,"\lrcorner"very near start]&\cS_{X//\bB\Gamma}\ar[d]\\
\bGr(X)\ar[r,"\Re_X"]&\cS_{X/};
\end{tikzcd}
\hspace{1cm}
\begin{tikzcd}[row sep=10pt]
\cS_{X//\bB\Gamma}\ar[d]\ar[r]\ar[dr,phantom,"\lrcorner"very near start]&\Fun([2],\cS)\ar[d,"d_1"]\\
*\ar[r,"(X\to\bB\Gamma)"]&\Fun([1],\cS).
\end{tikzcd}
\]
We may now identify $\bGr(X)'_{\bB\Gamma}$ with $\bGr^\Gamma(\tilde X)$. Concretely, we have a functor $\bGr^\Gamma(\tilde X)\to\bGr(X)'_{\bB\Gamma}$ sending a pair $(G,A\to \tilde X)$, consisting of a $\Gamma$-gaf together with a $\Gamma$-equivariant map $A\to \tilde X$, to the quotient pair $(G/\Gamma,A/\Gamma\to\tilde X/\Gamma\simeq X)$, where $G/\Gamma:=(A/\Gamma,V/\Gamma,H/\Gamma,\sigma/\Gamma,\upsilon/\Gamma)$; the pushout $\Re(G/\Gamma)\sqcup_{A/\Gamma}X$ is endowed with a canonical map to $\bB\Gamma$. The inverse functor $\bGr(X)'_{\bB\Gamma}\to\bGr^\Gamma(\tilde X)$ sends a triple $(G,A\to X,\Re(G)\sqcup_AX\to\bB\Gamma)$, consisting of a gaf $G$, a map $A\to X$, and a map $\Re_X(G)\sqcup_AX\to\bB\Gamma$, to the pair $(\tilde G,\tilde A\to\tilde X)$, where $\tilde A=A\times_{\bB\Gamma}*$, $\tilde V=V\times_{\bB\Gamma}*$ etc., and where the map $\tilde A\to \tilde X$ is also the image of $A\to X$ along the pullback functor $-\times_{\bB\Gamma}*$.

In particular $\bGr(X)'_{\bB\Gamma}$, and hence $\bGr(X)_{\bB\Gamma}$, are plain categories. Along the equivalence $\bGr(X)'\simeq\bGr^\Gamma(\tilde X)$, the full subcategory $\bGr(X)_{\bB\Gamma}\subset\bGr(X)'_{\bB\Gamma}$ corresponds to the full subcategory $\bGr^\Gamma_\tree(\tilde X)\subset\bGr^\Gamma(\tilde X)$ spanned by pairs $(G,A\to\tilde X)$ whose image along $\Re_{\tilde X}^\Gamma$ is a $\Gamma$-space under $X$ with contractible underlying space; we think of $\Re(G)\sqcup_A\tilde X$ as a tree (i.e. a contractible graph) endowed with a $\Gamma$-action that satisfies the following:
\begin{itemize}
\item there are finitely many $\Gamma$-orbits of vertices and edges, 
\item there are $p$ marked and pairwise distinct orbits of vertices, each endowed with an identification with the finite $\Gamma$-set $\Gamma/\Gamma_i$;
\item all other orbits of vertices, as well as all orbits of edges, are free.
\end{itemize}
Our goal is to show that $|\bGr^\Gamma_\tree(\tilde X)|\simeq*$. Note that we have a restricted cocartesian fibration $\tilde\fA^\Gamma\colon\bGr^\Gamma_\tree(\tilde X)\to\Gamma\Fin_\free$.
We may consider the full subcategory of 
$\bGr^\Gamma_\tree(\tilde X)_\red\subset\bGr^\Gamma_\tree(\tilde X)$ spanned by objects $(G,A\to\tilde X)$ satisfying the following property: for all $a\in A$, the preimage $\sigma^{-1}(a)\cap H$ is a singleton. The inclusion $\bGr^\Gamma_\tree(\tilde X)_\red\subset\bGr^\Gamma_\tree(\tilde X)$ admits a right adjoint, sending a generic pair $(G,A\to\tilde X)$ to the pair $(G',A'\to\tilde X)$ where $G':=(A',V,H,\sigma',\upsilon')$, $A':=\sigma^{-1}(A)\cap H$, $\sigma'$ agrees with the identity on $A'\subset H$ and with $\sigma$ elsewhere, and $\upsilon'$ agrees with the identity on $A'$ and with $\upsilon$ elsewhere; the counit of the adjunction $(G',A'\to\tilde X)\to(G,A\to\tilde X)$ is the $\tilde\fA^\Gamma$-cocartesian lift of $\sigma\colon A'\to A$ at $(G',A'\to\tilde X)$. It suffices therefore to prove that $|\bGr^\Gamma_\tree(\tilde X)_\red|\simeq*$.

Now let $\Gaf_{\emptyset,\tilde X}^\Gamma$ denote the fibre at $\tilde X\in\Gamma\Fin$ of the functor $\fA^\Gamma\colon\Gaf_\emptyset^\Gamma\to\Gamma\Fin$; then we may identify $\bGr^\Gamma_\tree(\tilde X)_\red$ with the full subcategory $\Gaf_{\emptyset,\tilde X}^{\Gamma,\tree}\subset\Gaf_{\emptyset,\tilde X}^\Gamma$ spanned by objects $G$ such that the underlying space of the $\Gamma$-space $E\sqcup_H(\tilde X\sqcup V)$ is contractible. It suffices therefore to prove that $|\Gaf_{\emptyset,\tilde X}^{\Gamma,\tree}|\simeq*$.

We may consider the full subcategory $\Gaf^{\Gamma,\min}_{\emptyset,\tilde X}\subset\Gaf^{\Gamma,\tree}_{\emptyset,\tilde X}$ spanned by objects $G$ such that $G$ admits no \emph{leaves}; here by a \emph{leaf} we mean an element $v$ in the set of $V$ of inner vertices of $G$ whose valence is $1$, i.e. such that $\sigma^{-1}(v)\cap H$ is a singleton. If we think of a generic object of $\Gaf^{\Gamma,\tree}_{\emptyset,\tilde X}$ as a tree with $\Gamma$-action endowed with a $\Gamma$-equivariant inclusion of $\tilde X$ into the $\Gamma$-set of vertices, then for an object in $\Gaf^{\Gamma,\min}_{\emptyset,\tilde X}$ we are requiring that the $\Gamma$-action on the tree be \emph{minimal}, in the sense that there is no proper $\Gamma$-equivariant subtree containing $\tilde X$.
The inclusion $\Gaf^{\Gamma,\min}_{\emptyset,\tilde X}\subset\Gaf^{\Gamma,\tree}_{\emptyset,\tilde X}$ admits a left adjoint, sending $G$ to the $\Gamma$-gaf $G'$ obtained from $G$ by repeatedly collapsing leaves along with their associated edges.
Our goal becomes therefore to show that $|\Gaf^{\Gamma,\min}_{\emptyset,\tilde X}|\simeq*$. 

We finally let $\Gaf^{\Gamma,\min}_{\emptyset,\tilde X,\ge3}\subset\Gaf^{\Gamma,\min}_{\emptyset,\tilde X}$ denote the full subcategory spanned by objects $G$ such that for all $v\in V$, the valence of $v$, i.e. the cardinality of $\sigma^{-1}(v)\cap H$, is at least 3. We have a functor $F\colon\Gaf^{\Gamma,\min}_{\emptyset,\tilde X}\to\Gaf^{\Gamma,\min}_{\emptyset,\tilde X,\ge3}$, sending $G$ to the $\Gamma$-gaf $G'$ obtained from $G$ by removing all inner vertices of valence 2 and concatenating the edges accordingly. We have that $F$ is a cocartesian fibration; the fibre $F^{-1}(G')$ can be thought of as the category of subdivisions of edges of $G'$, and is equivalent to $\bDelta^{E'}$, where $\bDelta$ denotes the simplex category; in particular all fibres of $F$ have contractible classifying space, so that we have $|\Gaf^{\Gamma,\min}_{\emptyset,\tilde X}|\simeq|\Gaf^{\Gamma,\min}_{\emptyset,\tilde X,\le3}|$.

We conclude by arguing that $|\Gaf^{\Gamma,\min}_{\emptyset,\tilde X,\le3}|\simeq*$. In \cite[Section 4]{GuirardelLevitt}, Guirardel--Levitt introduce a topological space $P\cO$ associated with a free product of groups $\Gamma$ as above, under the condition that $p\ge1$, $p+k\ge2$, and that \emph{all factors $\Gamma_i$ are non-trivial groups}. 

Points in $P\cO$ are given by isometry classes of metric simplicial $\Gamma$-trees $T$ such that the action of $\Gamma$ on $T$ has finitely many orbits of vertices and edges, and that the action is minimal (i.e. $T$ has no leaves). As trees are considered as a particular type of metric spaces, rather than as combinatorial objects, vertices of valence 2 are not allowed in $T$, or otherwise said, they are considered as points inside a longer edge, unless they have a non-trivial stabiliser. For the same reason, a tree $T$ is identified with any tree $T'$ obtained by actually collapsing edges of length 0, as the underlying metric spaces are $\Gamma$-equivariantly, and uniquely, isometric.
The ``$P$'' in the notation ``$P\cO$'' refers to the fact that the length function is considered as a projective function, i.e. up to rescaling by positive real numbers.
A metric simplicial $\Gamma$-tree $T$ must satisfy the following additional property in order to represent a point in $P\cO$: for all $1\le i\le p$ \emph{there is} exactly one vertex $v_i$ in $T$ whose stabiliser is $\Gamma_i$, whereas all other points of $T$ have trivial stabilisers; clearly, $v_i\neq v_j$ if $i\neq j$ because $\Gamma_i\neq\Gamma_j$ as subgroups of $\Gamma$, using that all $\Gamma_i$ are non-trivial.

We have emphasized two conditions that differ in the setting of Guirardel--Levitt from our setting: all $\Gamma_i$ are assumed to be non-trivial groups, and for $T\in P\cO$ the point $v_i\in T$ is not considered as an additional datum, but a property of the tree $T$, namely the unique fixpoint of $\Gamma_i$. We now observe that the construction of Guirardel--Levitt makes sense also if we instead relax the condition that all $\Gamma_i\neq1$, and if we specify as data of $T$ the vertices $v_i$ that have $\Gamma_i$ as stabiliser: for all $\Gamma_i\neq1$ there will be a unique choice if there is a choice at all (as we still require stabilisers of edges to be trivial), but for all $\Gamma_i=1$ we are choosing a point $v_i$ in $T$ with free $\Gamma$-orbit; we further require that the points $v_1,\dots,v_p$ be distinct.

The proof that $P\cO$ is contractible extends to this setting with minor modifications. If we now take the ``barycentric spine'' of $P\cO$, we obtain a topological space homeomorphic to the geometric realisation of the nerve of a skeleton of the category $\Gaf^{\Gamma,\min}_{\emptyset,\tilde X,\le3}$, in which we select one object in each isomorphism class. This concludes the proof.
\end{proof}

\bibliography{Bibliographystringtopology.bib,Bibliographygraphs.bib,Bibliographyother.bib}{}
\bibliographystyle{plain}

\end{document}